\documentclass[12pt]{amsart}
\usepackage[colorlinks=true,linkcolor=red,urlcolor=black]{hyperref}
\usepackage{blindtext}
\usepackage[active]{srcltx}

\usepackage[usenames, dvipsnames]{color}
\definecolor{amarillo}{RGB}{219, 48, 122}

\usepackage{amsmath,amsthm,amssymb,amsfonts}
\usepackage{multirow}
\usepackage{upgreek}
\usepackage{soul} 

\usepackage{tikz} 
\newtheorem{teo}{Theorem}
\newtheorem{prop}[teo]{Proposition}

\newtheorem{lema}[teo]{Lemma}

\newtheorem{maintheorem}{Theorem}
\newtheorem{maincor}[maintheorem]{Corollary}

\numberwithin{equation}{section}

\theoremstyle{remark}
\newtheorem{rem}[teo]{Remark}

\makeatletter
\def\referencia#1#2{\begingroup
    #2%
    \def\@currentlabel{#2}%
    \phantomsection\label{#1}\endgroup
}
\makeatother

\newcommand{\LL}{\mathcal L}

\title[Regularity of the Lyapunov Exponents for Dominated Splitting]{Regularity with respect to the parameter of Lyapunov Exponents for Diffeomorphisms with Dominated Splitting}
\author[R. Saghin]{Radu Saghin}
\address{Radu Saghin, Instituto de Matem\'atica,
Pontificia Universidad Cat\'olica de Valpara\'{\i}so, Blanco Viel 596,
Cerro Bar\'on, Valpara\'{\i}so-Chile.}
\email{radu.saghin@pucv.cl}
\thanks{R.S. was partially supported by Proyecto Fondecyt 1171477}

\author[P. Valenzuela-Henr\'iquez]{Pancho Valenzuela-Henr\'iquez}
\address{Francisco Valenzuela-Henr\'iquez, Instituto de Matem\'atica,
Pontificia Universidad Cat\'olica de Valpara\'{\i}so, Blanco Viel 596,
Cerro Bar\'on, Valpara\'{\i}so-Chile.}
\email{francisco.valenzuela@pucv.cl}
\thanks{P.V.H. was partially supported by Programa Regional Mathamsud 16-MATH-06 PHYSECO}

\author[C. H. V\'asquez]{Carlos H. V\'asquez}
\address{Carlos H. V\'asquez, Instituto de Matem\'atica,
Pontificia Universidad Cat\'olica de Valpara\'{\i}so, Blanco Viel 596,
Cerro Bar\'on, Valpara\'{\i}so-Chile.} \email{carlos.vasquez@pucv.cl}
\thanks{C.H.V.  was partially supported by Proyecto Fondecyt 1171427.}
\date{\today}

\subjclass{Primary: 37D25, 37D30.}

\keywords{partial hyperbolicity, Lyapunov exponents}

\begin{document}

\maketitle

\begin{abstract}

We consider families of diffeomorphisms with dominated splittings and preserving a Borel probability measure, and we study the regularity of the Lyapunov exponents associated to the invariant bundles with respect to the parameter. We obtain that the regularity is at least the sum of the regularities of the two invariant bundles (for regularities in $[0,1]$), and under suitable conditions we obtain formulas for the derivatives. Similar results are obtained for families of flows, and for the case when the invariant measure depends on the map.

We also obtain several applications. Near the time one map of a geodesic flow of a surface of negative curvature the metric entropy of the volume is Lipschitz with respect to the parameter. At the time one map of a geodesic flow on a manifold of constant negative curvature the topological entropy is differentiable with respect to the parameter, and we give a formula for the derivative. Under some regularity conditions, the critical points of the Lyapunov exponent function are non-flat (the second derivative is nonzero for some families). Also, again under some regularity conditions, the criticality of the Lyapunov exponent function implies some rigidity of the map, in the sense that the volume decomposes as a product along the two complimentary foliations. In particular for area preserving Anosov diffeomorphisms, the only critical points are the maps smoothly conjugated to the linear map, corresponding to the global extrema.

\end{abstract}


\section{Introduction}\label{sec:intro}

The theory of characteristic exponents originated over a century ago in the study of the stability if solutions of differential equations by A. M. Lyapunov \cite{L1907b}. The work of Furstemberg, Kesten, Oseledets , Kingman, Ledrappier and other built the study of Lyapunov exponents into a very active research field  in its own right, and one with an unusually vast array of interaction with others areas of the mathematics and physics,  as stochastic processes (random matrices \cite{FK1960,FK1983}, random walks on groups \cite{G1979}), spectral theory (Schr\"odinger-type operators \cite{D2007,D2017} ) and smooth dynamics (non uniform hyperbolicity \cite{BP2007}). Since then, an extensive literature has been written about it, we refer the reader to the books \cite{BP2002,V2014}  and the expository papers \cite{W2017,V2019} for an approach of the theme related with our work.

In the setting of smooth dynamics, Lyapunov exponents play a key role understanding the  behavior of a dynamical system. On one hand, when  the Lyapunov exponents are nonzero, the theory initiated by Pesin  \cite{P1976} provides detailed geometric information on the dynamics and several deep results have been proved: entropy formula for smooth measures \cite{P1977} and its converse \cite{L1984b, LY1985}, the interplay with the Hausdorff dimension \cite{BPS1999}, the existence of uniformly hyperbolic sets having many periodic orbits (in particular, the number of orbits of period $n$ grows exponentially in $n$) and carrying large entropy \cite{K1980}, statistical description for the orbits of a large set of  points \cite{ABV2000, BV2000,D2000}.

On the other hand, vanishing exponents is an exceptional situation that also can be exploited. In the sixties, Furstenberg \cite{F1963}   proved in the setting of random matrices in ${\rm SL}(2,\mathbb{R})$ that if the exponent vanishes, then the matrices either leave invariant a common line or pair of lines, or they generate a precompact group (see \cite{L1984c} for a generalization to any dimension). Such possibilities are degenerate and they can be easily destroyed by perturbing the matrices. This suggests that in general zero exponents should be a special situation, associated with some rigidity of the system. There are several significant recent advances in this direction, and various versions of the so-called "Invariance principle" have been formulated: a dynamical setting in \cite{BGMV2003}, which was further refined and applied in various works, see for example \cite{AV2010,ASV2013,AVW2015}. 

An interesting question is  \textit{how do the Lyapunov exponents depend on parameters?}. Formally, Lyapunov exponents are quantities associated to a cocycle (linear bundle maps) over a measure-preserving dynamical system (base). So the question above involves the cocycle, the base map and the measure as underlying datas from which Lyapunov exponents depend.  Understanding the exact relationship between exponents, cocycles, measures, and dynamics is an area under intense exploration.

A first approach to the question above is to assume that the base map and the invariant measure are fixed, and to allow changes for the linear bundle maps. In the general situation, the Lyapunov exponents may not depend smoothly, or even continuously of the cocycle. In fact,  Bochi \cite{B2002} proved that for any fixed ergodic invertible dynamical system over a compact space there is a residual set (with respect to the $C^0$ topology) of continuous  $\textrm{SL}(2,\mathbb R)$-cocycles that are either uniformly hyperbolic or have zero Lyapunov exponents almost everywhere. These conclusions have been extended to arbitrary dimension by Bochi and Viana  \cite{BV2005} showing that if the Lyapunov exponents are continuous, and the Oseledets splitting is not trivial, then it must be dominated. This shows that one does not have continuity of the exponents if the exponents are nonzero and there is no dominated splitting. It is worth mentioning that recently Viana and Yang \cite{VY2019} proved that in the non-invertible setting Bochi's result does not hold, there are open sets of cocycles over some specific (highly hyperbolic) base maps where the exponents vary continuously in the $C^0$ topology.

On the other hand, when the cocycles have dominated splitting, one has even much better regularity than continuity. In fact,  Ruelle showed in \cite{R1979} that for any (fixed) base dynamics, for an analytic family of cocycles with dominated splitting, the Lyapunov exponents (corresponding to the invariant bundles of the splitting) are also analytic with respect to the parameter.

If we restrict our attention to some specific families of cocycles over some specific dynamics, one may obtain some continuity results. For example if the base dynamics is sufficiently random (a shift or hyperbolic, with an invariant measure with product structure), and the cocycle satisfies some other conditions (one-step, or H\"older continuous and with a bunching property or with holonomies), continuity of the Lyapunov exponents for two-dimensional cocycles was established (see  \cite{BV2017,BBB2018}), while a generalization for higher dimensional cocycles was announced by Avila-Eskin-Viana, at least for the case of random product of matrices. The reader can find an exposition of recents developments in this program (for cocycles) in the survey \cite{V2019}.

We are interested in this paper in the particular case when the cocycle is the \textit{derivative cocycle} induced from the derivative of a diffeomorphism $f:M\to M$ defined on a manifold $M$. In this case, the changes in the base dynamics $f$ are related to the changes in the cocycle, so a degree of freedom is lost and this interdependence makes the analysis more difficult. 

A natural approach to the question above in this setting is to try to understand the regularity of the Lyapunov exponents with respect to changes in the dynamics  fixing a invariant measure, for instance,  families of conservatives dynamics all of them preserving the Lebesgue measure. As well as in the case of  cocycles, some hyperbolicity is necessary if we wish obtain certain regularity. For the $C^1$ topology, R. Ma\~n\'e \cite{M1995}  observed that an area-preserving diffeomorphism of surfaces is a continuity point of the Lyapunov exponents only if either it is Anosov or its Lyapunov exponents vanish almost everywhere. His arguments were completed by J. Bochi \cite{B2002} and were extended to arbitrary dimensions by Bochi and M. Viana \cite{BV2005}, showing that in the absence of the dominated splitting, one cannot expect in general the continuity of the exponents. This suggests that in order to obtain some regularity of the Lyapunov exponents with respect to parameters, one should consider classes of diffeomorphisms with dominated splitting, such as Anosov diffeomorphisms or partially hyperbolic systems.

The regularity of various dynamical invariants (including the stable and unstable Lyapunov exponents) with respect to parameters was successfully investigated in the context of Anosov systems (motivated by hyperbolic geometry), along a series of works by Katok, Knieper, Pollicott, Weiss, Contreras, Ruelle, among others. The regularity of measure-theoretic entropy under smooth perturbations of $C^\infty$ geodesic flows on negatively curved surfaces is obtained in \cite{KW1989}. this result is extended in \cite{KKPW1990} and \cite{KKPW1989}, where it shown that the topological entropy of Anosov flows varies almost as smoothly as the perturbation. In \cite{P1991} the same results are obtained using equidistribution of periodic points. Also formulas for the derivatives of topological and measure theoretic entropy of the SRB measure were obtained in \cite{KKPW1990, KKW1991}, allowing, for instance to characterize the critical points of topological entropy on the space of negatively curved metrics. the paper \cite{W1992} gives formulas for the derivatives of entropy for Axiom A flows on compact manifolds and a formula for the derivative of  Hausdorff dimension of basic sets for AxiomA diffeomorphisms on compact surfaces. In \cite{C1992} it is proven that, for $C^r$ families of Anosov flows, the topological entropy is $C^r$, while the pressure function and in particular the metric entropy are $C^{r-2}$, and as a consequence it improves the degree of differentiability of the entropy obtained in \cite{KKW1991}. In \cite{R1997,R2003B} it is showed that the SRB measure of a mixing Axiom A attractor depends smoothly on the diffeomorphism and also formulas for the derivatives are obtained. The results are extended to hyperbolic flows in \cite{R2008} (see also \cite{BL2007}). Recall that in this setting the measure-theoretic entropy of SRB measures (or in particular of the volume) is exactly the sum of the positive Lyapunov exponents (the unstable exponent), so the results on the regularity of the metric entropy are in the same time results on the regularity of the stable and unstable exponents (the same holds for the formulas for derivatives). It is necessary to take in account that all these results use strongly the uniformly hyperbolic structure, in particular the structural stability of such systems, existence of Markov partitions and finite symbolic representation of the dynamic, analytic dependence of the topological pressure with respect to H\"older continuous potentials, equidistribution of periodic points, etc.


Beyond the uniform hyperbolicity, the study of the regularity of Lyapunov exponents with respect to the volume in the context of partially hyperbolic dynamics was initiated in a remarkable paper by Shub-Wilkinson \cite{SW2000}. The authors established the regularity of the center exponent within a specific family of volume preserving partially hyperbolic diffeomorphisms of the three-torus, they showed that the second derivative is nonzero, and in conclusion they constructed open sets of such diffeomorphisms which are stably ergodic, nonuniformly hyperbolic, and with pathological center foliations. These ideas were pushed further in \cite{RW2001}, while in \cite{R2003} the regularity of the exponents and formulas for the derivatives were obtained at linear automorphisms of the $n$-torus.

Another remarkable progress was obtained in \cite{D2004}, in the context of Anosov actions. Here, among other things, Dolgopyat established the regularity of the Lyapunov exponents at the time-one map of the geodesic flow on a surface of constant negative curvature, and the non-vanishing of the second derivative. He allows even that the measure on the base changes with the parameter, as long as it is a Gibbs u-state. Similar ideas were used also in \cite{DP2002} in order to construct completely hyperbolic diffeomorphisms on any manifold. In fact the two papers \cite{SW2000,D2004} are the main source of inspiration of our work.

It is easy to see that once we have a dominated splitting, the Lyapunov exponents corresponding to the invariant bundles are automatically continuous. However this may not be the case for individual Lyapunov exponents inside the bundle (if the dimension of the bundle is at least two). Bochi proved in \cite{B2010} that every partially hyperbolic symplectomorphism can be $C^{1}$-approximated by partially hyperbolic diffeomorphisms whose center Lyapunov exponents vanish. In particular, non-uniformly hyperbolic systems are not $C^{1}$-open. This situation changes if we consider the $C^{r}$ topology for $r\geq 2$ as showed recently by Liang, Marin and Yang  \cite{LMY2018,LMY2019}.

In the line of removing zero exponents and obtaining nonuniform hyperbolicity, Baraviera and Bonatti obtained in \cite{BB2003} that the Lyapunov exponents are not locally constant in the $C^1$ topology. In a parallel direction, there exists recent research relating the zero central exponents with rigidity properties of the system, and thus suggesting that the zero exponents are a highly non-generic situation, and they can also be easily removed using "Invariance Principle"-type results.

Our work is related to the Lyapunov exponents corresponding to bundles of invariant splittings. Since the continuity comes for free in this case, we are interested in higher regularity. We describe our setting in the following subsections.

\subsection{The basic setting}

Our setting is fairly general, we basically consider any (neighborhood of) diffeomorphisms with dominated splitting which preserve an invariant measure.

Let $M$ be an orientable compact Riemannian manifold without boundary of dimension $d$. Let $f:M\to M$ be a $C^r$ diffeomorphism of $M$, $r\ge 1$, which has a dominated splitting $TM=E^1\oplus E^2\oplus E^3$, meaning that the splitting is continuous, invariant under $Df$, and satisfies the following conditions:

$$
 \sup_{x\in M}\| Df\mid_{E^1(x)}\|\cdot \| Df^{-1}\mid_{E^2(f(x))}\|<1,
$$
$$
 \sup_{x\in M}\| Df\mid_{E^2(x)}\|\cdot \| Df^{-1}\mid_{E^3(f(x))}\|<1.
$$

We can allow that either $E^1$ or $E^3$ is trivial, however we assume that $E:=E^2$ and $F:=E^1\oplus E^3$ are not trivial, and let $k=\dim E\geq 1$. This means that our considerations can be applied in the context of partially hyperbolic diffeomorphisms to the stable, unstable, center, center-stable, center-unstable, or even intermediate bundles.

Assume that $f$ preserves the Borel probability measure $\mu$ on $M$. Oseledets Theorem (\cite{O1968}, see also \cite{M1987}, Chapter 4.10) gives the existence of $d$ Lyapunov exponents (counted with their multiplicity) for $\mu$ almost every point $p\in M$, and a corresponding Lyapunov splitting. From these $d$ exponents, $k$ will correspond to the invariant bundle $E$, meaning that the corresponding bundles of the Lyapunov splitting are inside $E$, and we denote their sum by $\lambda(p,f,E)$. The {\it integrated Lyapunov exponent of $f$ with respect to $\mu$ and associated to the bundle $E$} will be
$$
\lambda(f,E,\mu)=\int_M\lambda(p,f,E)d\mu.
$$

If the measure $\mu$ is ergodic for $f$, then of course $\lambda(p,f,E)=\lambda(f,E,\mu)$ for $\mu$ almost every $p\in M$.

From the Birkhoff Ergodic Theorem one can see that an alternative definition of the integrated Lyapunov exponent is

\begin{equation}\label{eq:defL}
\lambda(f,E,\mu)=\int_M\log\|Df^{\wedge k}|_E\|d\mu=\int_M\log|J(f|_E)|d\mu,
\end{equation}
where $J(f|_E)$ is the Jacobian of $f$ restricted to $E$.

Recall that if $f$  has a dominated splitting, then for any diffeomorphism $g$ which is $C^1$ close to $f$ the dominated splitting persists, i.e. there exists a dominated splitting for $g$, $TM=E^1_g\oplus E^2_g\oplus E^3_g$. If furthermore $g$ preserves the same measure $\mu$, then we can obtain again the integrated Lyapunov exponent $\lambda(g,E_g,\mu)$ of $g$ with respect to $\mu$ and associated to the bundle $E_g:=E^2_g$. 

The main goal of our paper is to study the regularity of the map $g\mapsto \lambda(g,E_g,\mu)$. This map is always continuous because of the continuous dependence of the dominated splitting with respect to the diffeomorphism. We will find sufficient conditions that guarantee better regularity of this map, we will obtain formulas for the derivatives along one-parameter families, and we will investigate the critical points in some specific situations. We also obtain several interesting applications.

\subsection{A simple example}

We start by presenting a simple example which gives insight into our results and also into the basic ideas of the proofs. It is in fact the particular case when the measure $\mu$ is the Dirac measure at a common fixed point.

Let $A\in GL_d(\mathbb R)$ be a matrix with a simple real positive eigenvalue $\eta$ and a corresponding eigenvector $v$, or $Av=\eta v$. Then any matrix $B$ in a neighborhood $\mathcal U$ of $A$ will also have a simple eigenvalue $\eta(B)$ which is the continuation of $\eta$. \textit{What can we say about the derivatives of the map $\log\eta:\mathcal U\subset GL_n(\mathbb R)\rightarrow\mathbb R$, $B\mapsto \log\eta(B)$?}

Let $E_{\eta}$ be the eigenspace $\mathbb Rv$, and $E_{\eta}^*$ the direct sum of the other (generalized) eigenspaces of $A$, so that we have the decomposition $\mathbb R^d=E_{\eta}\oplus E_{\eta}^*$, invariant by $A$. The adjoint matrix $A^*$ will also have the real simple eigenvalue $\eta$, and the corresponding eigenvector $v^*$ is a linear functional on $\mathbb R^d$ with the kernel equal to $E_{\eta}^*$ and satisfies $v^*A=\eta v^*$ (we write $v^*$ as a row vector). We assume the normalization $v^*v=1$.

Assume now that we have a smooth family of matrices $A(t)=B(t)A\in \mathcal U$ with $B(0)=Id$. The tangent vector to the family $A(t)$ in $t=0$ is $\mathfrak X=B'(0)\in\mathfrak{gl}_d(\mathbb R)$, or $B(t)\sim e^{t\mathfrak X}$ up to order one. Denoting by $\mathfrak Y=B''(0)-B'(0)^2\in\mathfrak{gl}_d(\mathbb R)$, we have the approximation up to order two: $B(t)\sim e^{t\mathfrak X}e^{\frac {t^2}2\mathfrak Y}$.

Each $A(t)$ will have a simple real eigenvalue $\eta(t)$ and corresponding eigenvectors $v(t)$, the continuations of $\eta$ and $v$. We normalize $v(t)$ such that $v^*v(t)=1$. Differentiating the relation $B(t)Av(t)=\eta(t)v(t)$ and evaluating in $t=0$ we get
\begin{equation}\label{eq:00}
\eta \mathfrak Xv+Av'(0)=\eta'(0)v+\eta v'(0).
\end{equation}
Since $v^*v(t)=1$ is constant we get that $v^*v'(t)=0$, or $v'(t)\in E_{\eta}^*=\ker v^*$. Applying $v^*$ to \eqref{eq:00} and dividing by $\eta$ we get
\begin{equation}\label{eq:lambda'}
[\log\eta]'(0)=\frac{\eta'(0)}{\eta}=v^*\mathfrak Xv.
\end{equation}

In other words, the derivative of $\log\eta$ at $A$ is  $D\log\eta (A) (\cdot) =v^*(\cdot)v$. It is worth observing that the derivative does not depend on the whole information of $A$, but only on the invariant decomposition $\mathbb R^d=E_{\eta}\oplus E_{\eta}^*$.

We can compute $v'(0)$ by projecting \eqref{eq:00} to $E_{\eta}^*$. Let $\mathcal P:\mathbb R^d\rightarrow E_{\eta}^*$ be the projection on $E_{\eta}^*$ along $E_{\eta}$, or $\mathcal Pu=u-(v^*u)v$. Then applying $\mathcal P$ to \eqref{eq:00}, and observing that $\mathcal Pv=0$ and $\mathcal Pv'(0)=v'(0)$, we get
\begin{equation*}
(\eta Id-A)v'(0)=\eta \mathcal P\mathfrak Xv.
\end{equation*}
The map $(\eta Id-A)$ is clearly not invertible on $\mathbb R^d$, however it is invertible if we restrict it to $\ker v^*=E_{\eta}^*$, so we obtain
\begin{equation}\label{eq:v'}
v'(0)=\left[\left.\left(Id-\frac 1{\eta}A\right)\right|_{E_{\eta}^*}\right]^{-1}\mathcal P\mathfrak Xv.
\end{equation}
Observe that $v'(0)$ is linear in $\mathfrak X$ (and does depend on $A$, not only on $E_{\eta}\oplus E_{\eta}^*$).

In order to compute $\eta''$ we differentiate twice the relation $B(t)Av(t)=\eta(t)v(t)$ and we evaluate in $t=0$
\begin{equation*}
\eta (\mathfrak X^2+\mathfrak Y)v+2\mathfrak XAv'(0)+Av''(0)=\eta''(0)v+2\eta'(0)v'(0)+\eta v''(0).
\end{equation*}
Applying $v^*$, using the fact that $v^*v'=v^*v''=0$, and dividing by $\eta$, we get
\begin{equation*}
\frac{\eta''(0)}{\eta}=v^*\mathfrak Yv+v^*\mathfrak X\mathfrak X v+\frac 2{\eta}v^*\mathfrak XAv'(0).
\end{equation*}
From here we obtain
\begin{equation}\label{eq:lambda''}
[\log\eta]''(0)=\frac{\eta''(0)}{\eta}-\frac{\eta'(0)^2}{\eta^2}=v^*\mathfrak Yv+v^*\mathfrak X\mathfrak Xv-\left(v^*\mathfrak Xv\right)^2+\frac 2{\eta}v^*\mathfrak XAv'(0).
\end{equation}
Let us remark that this is the sum of a linear map in $\mathfrak Y$ and a bilinear map in $(\mathfrak X,\mathfrak X)$ (recall that $v'(0)$ is linear in $\mathfrak X$ by \eqref{eq:v'}).

If we consider the map $\log \eta:\mathcal U\subset GL_d(\mathbb R)$, with $\mathcal U$ open in $GL_d(\mathbb R)$, then its derivative at $A$ does not vanish, there always exists $\mathfrak X\in\mathfrak {gl}_d(\mathbb R)$ such that $D\log\eta(A)\mathfrak X\neq 0$. However if we restrict the attention to perturbations $B(t)\in SO(d)$ (preserving the inner product), then we can talk about critical points of the map $\log\eta$. In this case $\mathfrak X\in\mathfrak{so}(d)$ and $\mathfrak Xv$ is orthogonal to $v$, and since $D\log\eta(A)\mathfrak X=v^*\mathfrak Xv$, we have a characterization of the critical points of the map $\log\eta$: the matrices $A$ for which $D\log\eta(A)$ vanishes are exactly the matrices $A$ such that the decomposition $\mathbb R^d=E_{\eta}\oplus E_{\eta}^*$ is orthogonal.

Now assume that the decomposition $E_{\eta}\oplus E_{\eta}^*$ is orthogonal, and $w$ is an eigenvector in $E_{\eta}^*$ for some other real eigenvalue $\nu$ of $A$. Let $B(t)$ be the rotation by angle $t$ in the ($v$, $w$)-plane, so $\mathfrak Y=0$ and $\mathfrak X$ be the rotation of angle $\frac{\pi}2$ in the plane generated by $v$ and $w$. Then $\mathfrak Xv=w(=\mathcal P\mathfrak Xv)$ and $\mathfrak Xw=-v$. We get that $v^*\mathfrak Xv=v^*w=0$, $v^*\mathfrak X\mathfrak Xv=v^*\mathfrak Xw=v^*(-v)=-1$. We also obtain that $v'(0)=\frac{\eta}{\eta-\nu}w$, so $v^*XAv'(0)=v^*\mathfrak X\frac{\eta\nu}{\eta-\nu}w=-\frac{\eta\nu}{\eta-\nu}$. Then
$$[\log\eta]''(0)=-1+\frac {2\nu}{\nu-\eta}=\frac{\eta+\nu}{\mu-\eta}.$$
We can conclude from here that if $\eta\neq-\nu$ then the second derivative of $\log\eta$ at $A$ is not degenerate (or not flat).

\subsection{More detailed setting and the correspondence with the example}

Keeping in mind the above example, let us move back to diffeomorphisms. We will see a correspondence between families of matrices and families of diffeomorphisms with dominated splittings.

To be more specific, let us fix a standing hypothesis and some notations which we will use throughout most of the paper:

{\bf Hypothesis \referencia{H}{(H)}}: {\it 
\begin{enumerate}
\item
Let $f$ be a $C^r$ ($r\geq 1$) diffeomorphism on the compact orientable Riemannian manifold $M$ of dimension $d$, with a dominated splitting $TM=E^1\oplus E^2\oplus E^3$, with all the sub-bundles oriented and the orientations preserved by $Df$, $E=E^2$ and $F=E^1\oplus E^3$ are nontrivial, $E$ has dimension $k$
\item
Let $f_t$, $t\in I\subset\mathbb R$, be a $C^r$ family of $C^r$ diffeomorphisms passing through $f$ ($0\in I$, $f_0=f$).
\item
All the maps $f_t$ preserve the same Borel probability $\mu$.
\end{enumerate}}

The first two conditions are very general, we just consider any family (with some regularity) passing through a diffeomorphism with a dominated splitting. The orientability conditions are made in order to avoid some technical difficulties, they are satisfied by most of the known examples, and they can be probably removed with little extra work. The third condition is much more restrictive, we require that all the diffeomorphisms preserve the same measure. Even if we have in mind mainly the situation when the measure $\mu$ is the volume (the conservative case), we prefer to write the results in this more general setting for various reasons. First of all, when establishing the regularity of the Lyapunov exponent and the formulas for the derivatives, it does not matter what the measure $\mu$ is, and assuming that $\mu$ is a volume does not simplify the proofs. Secondly, one can imagine possible applications when the measure $\mu$ is not the volume, for example it can be a volume on some submanifold which is preserved by the family, and we can consider Lyapunov exponents in the direction transversal to the submanifold (so one cannot reduce to the dynamics on the submanifold). We can even consider the case of a Dirac measure at a common fixed point, which is in fact the example in the previous subsection. A third and probably most important reason to consider any measure $\mu$ is the possibility of extending our results to the case when the invariant measure $\mu_t$ depends on the diffeomorphism $f_t$. We explore this possibility in Theorem \ref{teo:response12}, and we get a surprising application for the regularity of the topological entropy in Theorem \ref{teo:responsegeo}. We think that this direction deserves to be further investigated.

Next we list the notations which we will use throughout most of the paper:

{\bf Notations}:
{\it \begin{itemize}
\item
$TM=E_t^1\oplus E_t^2\oplus E_t^3$ is the corresponding dominated splitting for $f_t$; $E_t:=E_t^2$, $F_t:=E_t^1\oplus E_t^3$, so $TM=E_t\oplus F_t$ is a continuous invariant splitting for $f_t$.
\item
$\lambda(t)=\lambda(f_t,E_t,\mu)=\int_M\log\|Df_t^{\wedge k}|_{E_t}\|d\mu$.
\item
$X$ is the $C^{r-1}$ vector field on $M$ tangent to the family $h_t:=f_t\circ f_0^{-1}$ in $t=0$ (see subsection~\ref{ssec:XY}). If $r\geq 2$ let $\phi^X$ be the flow on $M$ generated by $X$.
\item
If $r\geq 2$ then $Y$ is the $C^{r-2}$ vector field which is the second order correction of $h_t$ (see subsection~\ref{ssec:XY}). If $r\geq 3$ let $\phi^Y$ be the flow on $M$ generated by $Y$.
\item
$\omega_E$ is a continuous nonzero $k$-form on $M$ such that $\ker\omega_F=F\wedge TM^{\wedge(k-1)}$.
\item
$V_t$ is a continuous nonzero $k$-multivector field in $E_t^{\wedge k}$. $V_E:=V_0$.
\item
We choose $\omega_F$ and $V_E$ (and $V_t$) such that: if $F$ is $C^{\alpha}$ then $\omega_F$ is $C^{\alpha}$; if $E$ is $C^{\beta}$ then $V_E$ is $C^{\beta}$; $\omega_F(V_t)=1$ (see subsection~\ref{ssec:omegaV}).
\item
$\eta_t,\tilde\eta_t:M\rightarrow\mathbb R$, $(f_t)_*V_t=\tilde\eta_t V_t$, $\eta_t=\tilde\eta_t\circ f_t$. Furthermore (see more comments below)
$$
\lambda(t)=\int_M\log\eta_t(p)d\mu=\int_M\log\tilde\eta_t(p)d\mu.
$$
\item
If $E$ is $C^1$ then $V_t$ is differentiable with respect to $t$ at $t=0$, and we denote its derivative $V'$ (for more on this and an explicit formula for $V'$ see subsection~\ref{ssec:regularidadVt}).
\item
If $t=0$ we will just drop the index $t$ from all the notations: $f:=f_0$, $E:=E_0$, $F:=F_0$, $\eta:=\eta_0$, $\tilde\eta:=\tilde\eta_0$, $\lambda:=\lambda(0)$, etc.
\end{itemize}}

The orientation assumptions are used in order to work with vector fields and forms, motivated by the book \cite{V2014}. We are making an abuse of the notations, considering that the $k$-form $\omega_F$ acts on $k$-multivectors, however the action is well defined since a differential form is multilinear and anti-symmetric.

We denote by $(f_t)_*$ the action induced by $Df_t$ on the tangent bundle (vectors, multivectors) and $(f_t)^*$ the action induced on the cotangent bundle (forms). Since the space $E_t(p)$ is $(f_t)_*$--invariant, there is a real number $\eta_t(p)$ such that
\begin{equation}\label{eq:eta}
(f_t)_*V_t(p)=\eta_t(p)\cdot V_t(f_t(p)).
\end{equation}
In other words, if we denote $\tilde\eta_t=\eta_t\circ f_t^{-1}$, we have
$$
(f_t)_*V_t(p)=\tilde\eta_t(p)\cdot V_t(p).
$$
Observe that in fact $\eta_t$ measures the volume expansion of $Df_t$ restricted to $E_t$ using a metric which gives norm one to the multivectors $V_t$. Since the Lyapunov exponent is independent of the metric this implies that
\begin{equation}\label{eq:Lyapunov}
\lambda(t)=\int_M\log\eta_t(p)d\mu=\int_M\log\tilde\eta_t(p)d\mu.
\end{equation}

A representation of $\omega, V, V_t$ and the action of the derivatives of $f$ and $h_t$ can bee seen in Figure~\ref{dib:perturb}.

\begin{figure}[h]
\includegraphics[scale=.95]{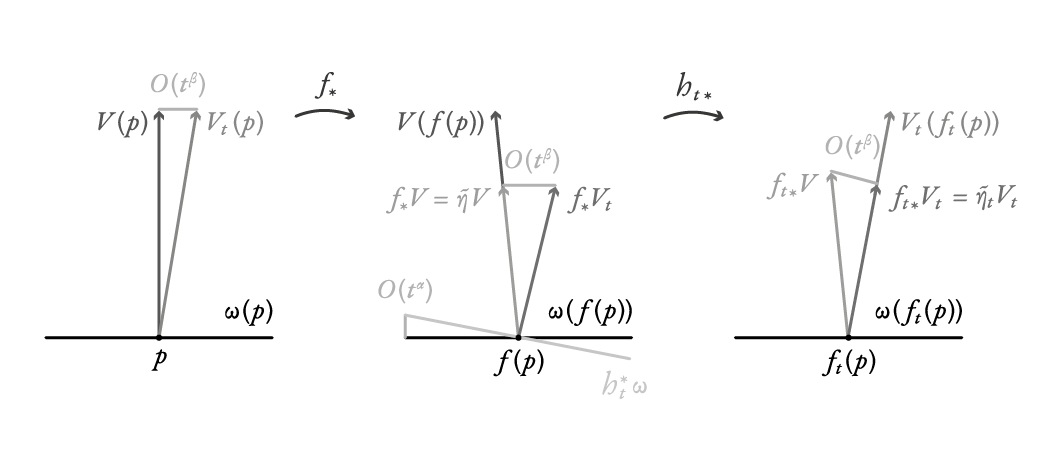}
\caption{\label{dib:perturb} Action of the derivative $Df_t$ on $V_t$.}
\end{figure}

Once we introduced all these objects, one can see the clear correspondence between them and the objects from the example of the family of matrices. The equivalences are included in Table \ref{table}, and we anticipate formulas for the derivatives of the Lyapunov exponents $\lambda'(0)$, $\lambda''(0)$, and the derivative $V'$ of $V(t)$, which we will state formally in the next sections.

There are only some small differences. While in the case of matrices the tangent vector $\mathfrak X$ acts on vectors and covectors by multiplication, in the case of diffeomorphisms the tangent vectorfield $X$ acts on other vectorfields or forms by the Lie derivative (this is the corresponding infinitesimal action of the family on vector fields or forms).

The formulas in Table \ref{table} for the derivative of the Lyapunov exponent are in terms of the Lie derivative, and they are defined if either $V_E$ or $\omega_F$ are $C^1$. Let us remark that the two formulas can be rewritten in the following form (recall that $\phi^X$ preserves the measure $\mu$):
\begin{equation}
\lambda'(0)=\left.\frac{\partial}{\partial t}\int_M\omega_F(\phi_{t*}^XV_E)d\mu\right|_{t=0}=\left.\frac{\partial}{\partial t}\int_M\phi_t^{X*}\omega_F(V_E)d\mu\right|_{t=0}.
\end{equation}
The expression $g(t)=\int_M\omega_F(\phi_{t*}^XV_E)d\mu$ is morally a convolution of two functions, and in general the regularity of a convolution of two functions is the sum of the regularities of the two functions. Using this observation we are able to show that even though $\omega_F(\phi_{t*}^XV_E)$ may not be differentiable with respect to $t$, the integral $\int_M\omega_F(\phi_{t*}^XV_E)d\mu$ will be indeed differentiable provided $\omega_F$ and $V_E$ are H\" older continuous and the sum of the H\" older exponents is larger than one. This allows us to obtain the derivative of the Lyapunov exponent even if the splitting is only H\" older with exponent better than $1/2$.

\newpage

\begin{table}[h!]\label{table}
  \begin{center}
    \caption{Correspondence Matrix-Diffeomorphism}
    \label{tab:table1}
    \begin{tabular}{l|r}
      \textbf{Family of matrices $A(t)$, $A=A(0)$} & \textbf{Family of diffeomorphisms $f_t$, $f=f_0$}\\
  \hline\hline
      Decomposition $\mathbb R^d=E_{\eta}\oplus E_{\eta}^*$, the  & Invariant splitting  $TM=E\oplus F$, where \\
eigenspace of $\eta$ and the sum of the  & $F=E_1\oplus E_3$, $E=E_2$ and $E_1\oplus E_2\oplus E_3$ is a  \\
other eigenspaces & dominated splitting, $\dim(E)=k$\\
    \hline
      Eigenvector $v$ of $A$ for $\eta$, & Continuous multivector field $V_E$ in $E^{\wedge^k}$,\\

      eigenvector $v(t)$ of $A(t)$ for $\eta(t)$ & continuous multivector field $V_t$ in $E_t^{\wedge^k}$ \\ 
      \hline
      Dual eigenvector $v^*$ for the same $\eta$, & Continuous form $\omega_F$, with the kernel \\
with the kernel $E_{\eta}^*$, and $v^*v(t)=1$ & $F\wedge TM^{\wedge^{k-1}}$, and $\omega_F(V_t)=1$ \\
\hline
       \multirow{3}{*}{Tangent matrix $\mathfrak X$ of $A(t)$ in $t=0$ }& Tangent vectorfield $X$ to $f_t$ in $t=0$; \\
& The infinitesimal action on $TM$ is $-\mathcal L_X$,\\
& the infinitesimal action on $T^*M$ is $\mathcal L_X$.\\
\hline
Second order tangent matrix $\mathfrak Y$, & Second order tangent vectorfield $Y$,\\
 so that $A(t)\sim e^{t\mathfrak X}e^{\frac{t^2}2\mathfrak Y}A$ & so that $f(t)\sim \phi_t^X\circ\phi_{\frac {t^2}2}^Y\circ f$\\
\hline
\multirow{3}{*}{The eigenvalue $\eta(t)$ of $A(t)$ ($\eta$ of $A$)} & Pointwise expansion $\tilde\eta_t:M\rightarrow \mathbb R$, \\
& $f_*V_E=\tilde\eta V_E$ ($\tilde\eta:=\tilde\eta_0$);\\
 & Lyapunov exponent $\lambda(t)=\int_M\log\tilde\eta_t d\mu$\\
\hline
\multirow{2}{*}{$[\log\eta]'(0)=v^*\mathfrak Xv$} & $\lambda'(0)=-\int_M\omega_F(\mathcal L_XV_E)d\mu$ (if $E$ is $C^1$)\\
&  $\lambda'(0)=\int_M\mathcal L_X\omega_F(V_E)d\mu$ (if $F$ is $C^1$)\\
\hline
$[\log\eta]''(0)=v^*\mathfrak Yv+v^*\mathfrak X\mathfrak Xv-$ & $\lambda''(0)=\int_M\mathcal L_Y\omega_f(V_E)-\mathcal L_X\omega_F(\mathcal L_XV_E)-$\\
\hspace{1.5cm}$-\left(v^*\mathfrak Xv\right)^2+\frac 2{\eta}v^*\mathfrak XAv'(0)$ & $-\left[\omega_f(\mathcal L_XV_E)\right]^2+\frac 2{\tilde\eta}\mathcal L_X\omega_F(f_*V')d\mu$\\
& (If $E$ and $F$ are $C^1$)\\
\hline
\multirow{2}{*}{$v'(0)=\left[\left.\left(Id-\frac 1{\eta}A\right)\right|_{E_{\eta}^*}\right]^{-1}\mathcal P\mathfrak Xv$ }& $V'=\left[\left.\left(Id-\frac {f_*}{\tilde\eta}\right)\right|_{F\wedge E^{\wedge(k-1)}}\right]^{-1}\mathcal P(\mathcal L_XV_E)$\\
& (If $E$ is $C^1$)\\
\hline
    \end{tabular}
  \end{center}
\end{table}

\newpage
\subsection{Regularity of the integrated Lyapunov exponent}\label{ssec:regularity}

Our first result relates the regularity of $\lambda$ at $t=0$ with the regularity of the splitting $TM=E\oplus F$ for $f_0$. It is well known that if $f$ is $C^2$ then the subbundles of the dominated splitting $E$ and $F$ are at least H\" older continuous, with the exponents given by the bounds on the derivative $Df$ restricted to the two subbundles. If $f$ is some special example of a partially hyperbolic diffeomorphism (skew product over Anosov, time one map of Anosov flow), then the regularity of one or both subbundles may be even better.

If the map $t\mapsto E_t$ is of class $C^{\beta}$ (with respect to $t$ in $t=0$), it is easy to see that the same holds for $t\mapsto\lambda(t)$, as long as $r\geq\beta+1$. If $E$ is $C^{\beta}$ (with respect to the points on $M$) for some $\beta \in[0,1]$, then $t\mapsto E_t$ is also $C^{\beta}$  at $t=0$. Dolgopyat proved in \cite{D2004} this fact for the case $\beta=1$, and we will discuss this point with more details in subsection~\ref{ssec:regularidadVt}. Thus the regularity of $\lambda$ at $t=0$ is at least the regularity of $E$, and by the symmetry it is also at least the regularity of $F$, at least up to the $C^1$ regularity. The next result says that in fact the regularity of $\lambda$ at $t=0$ is at least the sum of the regularity of $E$ and the regularity of $F$, if the two regularities of the bundles are in $[0,1)$.

\begin{maintheorem}\label{teo:lyapunovH}
Assume that \ref{H} is satisfied for $r\geq 3$, $F$ is of class $C^{\alpha}$ and $E$ is of class $C^{\beta}$ on a neighborhood of the support of $\mu$, for some $\alpha,\beta\in[0,1)$. Then
\begin{enumerate}
\item
If $\alpha+\beta<1$ then $\lambda(t)=\lambda(0)+O(t^{\alpha+\beta})$;
\item
If $\alpha+\beta=1$ then $\lambda(t)=\lambda(0)+O(t\log t)$;
\item
If $\alpha+\beta>1$ then $\lambda(t)=\lambda(0)+t\lambda'(0)+O(t^{\alpha+\beta-1})$. Furthermore
\begin{equation}\label{eq:derLyapH}
\lambda'(0)=\left.\frac{\partial}{\partial t}\int_M(\phi_t^X)^*\omega_F(V_E)d\mu\right|_{t=0}=\left.\frac{\partial}{\partial t}\int_M\omega_F((\phi^X_t)_*V_E)d\mu\right|_{t=0}
\end{equation}
and
\begin{equation}\label{eq:bderLyapH}
|\lambda'(0)|\leq C_{\alpha,\beta,M}\|X\|_{C^0}\|\omega_F\|_{C^{\alpha}}\|V_E\|_{C^{\beta}}+C_M\|X\|_{C^1}\|\omega_F\|_{C^0}\|V_E\|_{C^0}.
\end{equation}
(See subsection~\ref{ssec:omegaV} for a precise definition and comments on $\omega_F$ and $V_E$).
\end{enumerate}
\end{maintheorem}

The result is completely new because in particular it establishes differentiability of the Lyapunov exponents even if the subbundles of the dominated decomposition of $f$ are only H\" older continuous, as long as the sum of the H\" older exponents is larger than one. Our result also works for any family of perturbations in any direction, as long as it preserves the invariant measure $\mu$ of course. The previous results in the literature on differentiability of the Lyapunov exponents for dominated splittings considered either very specific maps $f$ with all the subbundles smooth (in fact algebraic: linear maps on the torus in \cite{SW2000,R2003}, time one map of the geodesic flow on a surface of negative curvature in \cite{D2004}), or some specific maps with some smooth subbundles and some specific perturbations (time-one maps of Anosov flows in \cite{DP2002}, perturbations preserving a subbundle).

In order to prove Theorem~\ref{teo:lyapunovH} we will use the following result which has its own interest. It treats basically the regularity of a convolution of two functions along a flow.

\begin{maintheorem}\label{teo:flow}
Let $\phi_t$ be a flow on the compact manifold $M$, generated by the $C^r$ vector field $X$. Let $f,g:M\rightarrow\mathbb R$ be continuous observables on $M$. Let $\mu$ be an invariant measure for $\phi$, and assume that $f$ is $C^{\alpha}$ and $g$ is $C^{\beta}$ in a neighborhood of the support of $\mu$, $\alpha,\beta\geq 0$, $r\geq\max\{\alpha,\beta\}-1$. Let $h(t)=\int_Mf(x)g(\phi_t(x))d\mu$. If either $\alpha+\beta$ is not an integer, or both $\alpha$ and $\beta$ are integers, then $h$ is $C^{\alpha+\beta}$; otherwise $h$ is $C^{\alpha+\beta-1+Zygmund}$.

We have the bound $\| h\|_{C^{\alpha+\beta}}\leq C_{X,\alpha,\beta}\| f\|_{C^{\alpha}}\| g\|_{C^{\beta}}$, where $C_{X,\alpha,\beta}$ depends on $\alpha,\beta$, and $X$. In particular if $\alpha,\beta\in(0,1)$, $\alpha+\beta>1$ then $\|h\|_{C^1}\leq C_{\alpha,\beta}\|X\|_{C^0}\|f\|_{C^{\alpha}}\| g\|_{C^{\beta}}$.
\end{maintheorem}


\subsection{A first application: Regularity of the metric entropy}\label{ssec:regmetricentropy}

As an application of Theorem~\ref{teo:lyapunovH} we will obtain good regularity of the metric entropy for a large class of conservative partially hyperbolic diffeomorphisms.

Let $PH^2_{\mu,1}(M)$ be the set of $C^2$ partially hyperbolic diffeomorphisms on the compact manifold $M$ which preserve a smooth volume $\mu$ and have one dimensional center bundle. The stable, central and unstable bundles will be H\" older continuous, with the H\" older exponent depending on the contraction/expansion rates of the diffeomorphism along the bundles (see subsection~\ref{ssec:regularidadVt}). Let $PH^{2,\frac 12+}_{\mu,1}(M)$ be the subset of diffeomorphisms in $PH^2_{\mu,1}(M)$ with the property that the contraction/expansion rates along the stable, center and unstable subbundles guarantee that the partially hyperbolic splitting is $C^{\alpha}$ for some $\alpha>\frac 12$. We assume also that all the subbundles are orientable and the orientations are preserved. This is a fairly large open set inside the $C^2$ diffeomorphisms of $M$, and it contains many known examples of partially hyperbolic diffeomorphisms and $C^1$ small perturbations in dimension three: time one maps of Anosov flows, skew products over Anosov, derived from Anosov maps with the center eigenvalue close to one. The examples also work in higher dimension with an additional pinching condition on the stable and unstable spectrum.

Let us consider the metric entropy function $h_{\mu}:\mathrm{Diff}^2(M)\rightarrow\mathbb R$. It is known that restricted to $PH^2_{\mu,1}(M)$ the map $h_{\mu}$ is upper semicontinuous (see \cite{VY2019} for example). From the Pesin formula one can easily see that $h_{\mu}$ is not continuous in general, and there is an exact description of the discontinuity points. For any $f\in PH^2_{\mu,1}(M)$ we consider the $\mu$-regular points of $M$ with strictly positive and strictly negative center exponent: $A_+(f)=\{x\in M:\lambda_c(x,f)>0\}$ and $A_-(f)=\{x\in M:\lambda_c(x,f)<0\}$. Let
$$
\mathcal A=\{f\in PH^2_{\mu,1}(M): \mu(A_+(f))>0, \mu(A_-(f))>0\}.
$$
Then the discontinuity points of $h_{\mu}$ on $PH^2_{\mu,1}(M)$ are exactly the diffeomorphisms in $\mathcal A$ (this is because of the Pesin formula and the fact that the stably ergodic diffeomorphisms are dense in $PH^2_{\mu,1}(M)$, see \cite{BW2010,HHU2008}).

In conclusion $h_{\mu}$ is continuous on $PH^2_{\mu,1}(M)\setminus\mathcal A$. Let us remark that $PH^2_{\mu,1}(M)\setminus\mathcal A$ is large, it contains the $C^1$ open and $C^2$ dense set of stably ergodic diffeomorphisms (see \cite{BW2010,HHU2008}). In particular $PH^{2,\frac 12+}_{\mu,1}(M)\setminus\mathcal A$ contains the $C^2$ maps within a fairly large $C^1$ neighborhood of the time-one map of a geodesic flow on a surface of negative curvature. We claim that if we further restrict our attention to $PH^{2,\frac 12+}_{\mu,1}(M)\setminus\mathcal A$, then $h_{\mu}$ is Lipschitz along any $C^2$ path.

\begin{maintheorem}\label{cor:gflow}
Let $M$ be a compact oriented Riemannian manifold with a smooth volume $\mu$. For any $C^{2}$ one-parameter family $f_t$ of $C^{2}$ diffeomorphisms in $PH^{2,\frac 12+}_{\mu,1}(M)$, the stable, central and unstable integrated Lyapunov exponents are differentiable with respect to $t$ at every point, and the metric entropy with respect to $\mu$ is Lipschitz with respect to $t$ on $\{ f_t:t\in I\}\setminus\mathcal A$.
\end{maintheorem}

In view of the discussion above, an immediate corollary is the following.

\begin{maincor}\label{cor:gflow2}
Let $\phi_1$ be the time-one map of the geodesic flow on a surface of negative curvature preserving the volume $\mu$. There exists a $C^1$ neighborhood $\mathcal U$ of $\phi_1$ such that for any $C^{2}$ one-parameter family of $C^{2}$ diffeomorphisms in $\mathcal U$ preserving $\mu$, the metric entropy with respect to $\mu$ is Lipschitz with respect to the parameter.
\end{maincor}

The above theorem shows that $h_{\mu}:PH^{2,\frac 12+}_{\mu,1}(M)\setminus\mathcal A$ is Lipschitz along any smooth path. Let us remark that in view of the Dacorogna-Moser Theorem (see \cite{DM1990}), it is reasonable to expect that if $f_0$ and $f_1$ are two $C^2$ close volume preserving diffeomorphisms, then the two maps can be joined by a $C^2$ curve of $C^2$ volume preserving diffeomorphisms $(f_t)_{t\in[0,1]}$, such that the tangent to the curve, which is a time-dependent vectorfield $X_t$, is bounded in the $C^1$ norm in terms of the initial $C^2$ distance between $f_0$ and $f_1$. This suggests that a reasonable conjecture is that restricted to $PH^{2,\frac 12+}_{\mu,1}(M)\setminus\mathcal A$, the metric entropy $h_{\mu}$ is locally Lipschitz, considering the $C^2$ topology on the set of diffeomorphisms.

Let us also comment that we do not expect that the map $h_{\mu}$ is differentiable along every smooth curve, we believe that $h_{\mu}$ may have corners when the integrated center exponent changes sign.


\subsection{Formulas for the derivatives}\label{ssec:formulas}

The formula \eqref{eq:derLyapH} gives a (not very explicit) formula for $\lambda'(0)$ when the sum of the regularities of $E$ and $F$ is greater than one. The next result gives explicit formulas for the first derivative of $\lambda$ in 0 when either $E$ or $F$ are $C^1$, and the second derivative of $\lambda$ in 0 when both $E$ and $F$ are $C^1$.

\begin{maintheorem}\label{teo:lyapunov12}
Assume that \ref{H} is satisfied for $r\geq 3$, and $X, Y, \omega_F, V_E, V'$ and $\tilde\eta$ are defined as above.
\begin{enumerate}
\item
If $F$ is $C^1$ on a neighborhood of $\mathrm{supp}(\mu)$, then $\lambda$ is differentiable in $0$ and
\begin{equation}\label{eq:derLyap1F}
\lambda'(0)=\int \LL_X\omega_F(V_E)d\mu.
\end{equation}
\item
If $E$ is $C^1$ on a neighborhood of $\mathrm{supp}(\mu)$, then $\lambda$ is differentiable in $0$ and
\begin{equation}\label{eq:derLyap1E}
\lambda'(0)=-\int \omega_F(\LL_XV_E)d\mu.
\end{equation}
\item If both $E$ and $F$ are $C^1$ on a neighborhood of $\mathrm{supp}(\mu)$, then $\lambda$ has expansion of order 2 at $t=0$:
$$
\lambda(t)=\lambda(0)+t\lambda'(0)+\frac {t^2}2\lambda''(0)+o(t^2),
$$
where
\begin{equation}\label{eq:derLyap2}
\lambda''(0)=\int_M-\LL_X\omega_F\left(\LL_X V_E\right)+\LL_Y\omega_F(V_E)-\left( \LL_X\omega_F(V_E)\right)^2+\frac 2{\tilde\eta}\LL_X\omega_F(f_*V')d\mu.
\end{equation}
\end{enumerate}
\end{maintheorem}

Here $\LL$ is the usual Lie derivative which is well defined on $C^1$ forms and multivector fields.

This result generalizes the formulas in \cite{SW2000,R2003,D2004,DP2002}. In \cite{SW2000} the first two derivatives are computed when $f$ is a linear partially hyperbolic map on the 3-torus, so all the subbundles are one dimensional and lineal, and the perturbation is within the center-unstable direction. In \cite{R2003} the first two derivatives are computed for the case when $f$ is a linear map on the $d$-torus (eventually multiplied by a rotation), so all the bundles are lineal and of any dimension, and the perturbation can be arbitrarily, however the formulas for the derivative seem more difficult to work with. In \cite{D2004} again the first two derivatives are computed when $f$ is the time one map of the geodesic flow on a surface of constant negative curvature, so all the bundles are one dimensional and $C^{\infty}$, and any perturbation is allowed, even not conservative (the measure $\mu_t$ may depend on $t$, we will talk more about this result in the next subsection). In \cite{DP2002} $f$ is the time one map of a modified Anosov flow product with another diffeomorphism, and the perturbation is a specific one inside the center-unstable direction (the center bundle is one dimensional and $C^{\infty}$, the others may be larger); there are no derivatives computed explicitly, however there are estimates of the higher order of the central exponent showing that it is not constant.

Our result is basically a refinement of the results mentioned above, however we think that it is interesting for several reasons. First of all it assumes low regularity of the subbundles, compared with the other previous results (one subbundle is $C^1$ for the first derivative, both subbundles are $C^1$ for the second derivative). We do acknowledge however that the assumptions are still strong and do not hold for a generic diffeomorphism with a dominated splitting, with the exception of some very specific examples. A second reason why our result is interesting is that we consider any possible dimensions of the subbundles, any invariant measure $\mu$ (or course, admitting that there are few examples other than the volume), and any family of perturbations. A third reason is that we obtain fairly simple formulas (at least for the first derivative) which hold in this general situations, and this allows one to work with them further in order to study the critical points of the Lyapunov exponent for example (see subsection \ref{ssec:critical} for some remarkable results in this direction).

It is remarkable that the first derivative does not depend explicitly on $f$ or $f_t$, it only depends on the both sub-bundles $E$ and $F$ of the invariant splitting for $f$, on the invariant measure $\mu$, and on the vector field $X$ tangent to the family $h_t$. It is not hard to see that $\lambda'(0)$ is independent on the choice of $\omega_F$ and $V_E$, and it is in fact linear in the vector field $X$ (and in $\omega_f$ and $V_E$), and bounded with respect to the $C^1$ topology. The second derivative is the sum of a bilinear form in $X$ and a linear form in $Y$, and the last term does depend on $f$.

\subsection{Variable measure}\label{ssec:vmeasure}

We remark that Theorem~\ref{teo:lyapunov12} can be formulated even in the setting of invariant measures $\mu_t$ varying with the parameter $t\in I$. Of course, we need to impose some conditions on the regularity of the family of measures, at least continuity in the weak* topology..

Let us consider now the corresponding hypothesis for variable measure. It is identical with the hypothesis \ref{H}, with the only difference that now the invariant measures $\mu$ depend on $t$.

{\bf Hypothesis \referencia{H1}{(H')}} {\it The conditions (i) and (ii) are satisfied. The condition (iii) is replaced by\\
(iii') Every map $f_t$ preserves a Borel probability measure $\mu_t$, $\lim_{t\rightarrow 0}\mu_t=\mu_0:=\mu$ in the weak* topology.}

We will also have a change in the notations, this time we have
$$\lambda(t)=\lambda(f_t,E_t,\mu_t).$$

We say that {\it the family $\mu_t$ has linear response ${\mathcal R}(\varphi)$ for the continuous function $\varphi:M\rightarrow \mathbb R$}, if the application $t\mapsto \int_M\varphi d\mu_t$ is differentiable in $t=0$ and the derivative is $\mathcal R(\varphi)\in\mathbb R$. Let us point out that this condition is much weaker than the differentiability of $\mu_t$, {\it it is the differentiability only for the observable $\varphi$}. In particular, if $\varphi$ is constant, then any family of measures has linear response, and $\mathcal R(\varphi)=0$.

We obtain the following result.

\begin{maintheorem}\label{teo:response12}
Assume that \ref{H1} is satisfied for $r\geq 3$. Then:

\begin{enumerate}
\item If $F$ is $C^1$, then $\lambda$ is differentiable in $0$ if and only if the family $\mu_t$ has linear response $\mathcal R(\log\eta)$ for the function $\log\eta:M\rightarrow\mathbb R$. In this case we have
$$
\lambda'(0)=\mathcal R(\log\eta)+\int_M \LL_X\omega_F(V_E)d\mu.
$$
\item If $E$  is $C^1$, then $\lambda$ is differentiable in $0$ if and only if the family $\mu_t$ has linear response $\mathcal R(\log\eta)$ for the function $\log\eta:M\rightarrow\mathbb R$. In this case we have
$$
\lambda'(0)=\mathcal R(\log\eta)-\int \omega_F(\LL_XV_E)d\mu.
$$
\item
Suppose that $E$ is $C^1$ and $F$ is $C^2$. In addition, suppose that $\eta$ is constant and the family $\mu_t$ has linear response $\mathcal R\left(\LL_X\omega_F(V_E)\circ f\right)$ for the function $\LL_X\omega_F(V_E)\circ f:M\rightarrow\mathbb R$. Then $\lambda$ has expansion of order two at $t=0$, and
\begin{align*}
\lambda''(0)&=2\mathcal R\left(\LL_X\omega_F(V_E)\circ f\right)+\\
&+\int_M \LL_X\LL_X\omega_F(V_E)+\LL_Y\omega_F(V_E)-\left(\LL_X\omega_F(V_E)\right)^2+\frac 2{\tilde\eta}\LL_X\omega_F(f_*V')d\mu.
\end{align*}
\end{enumerate}
\end{maintheorem}

Let us compare our result with \cite{D2004}. In \cite{D2004} Dolgopyat considers a partially hyperbolic $f$ which is an Anosov element of a rapidly mixing abelian Anosov action, the invariant measures $\mu_t$ are $u$-Gibbs measures of the diffeomorphisms $f_t$, and he obtains the extremely remarkable result that the family $\mu_t$ is differentiable with respect to $t$ at $t=0$. He applies this result in order to obtain two derivatives of the Lyapunov exponents with respect to $u$-Gibbs measures for perturbations of the time one map of a geodesic flow on a surface of constant negative curvature. In this specific example all the three bundles are $C^{\infty}$ and one dimensional, all the corresponding functions $\eta$ are constant, and of course the measures are chosen to be $u$-Gibbs.

Theorem \ref{teo:lyapunov12} is just an extension of Dolgopyat's example from the last part of \cite{D2004} to a more general setting, specially when considering the first derivative of the Lyapunov exponent. First of all we consider any dimensions of the subbundles, we assume only $C^1$ regularity of one subbundle, and the formula which we obtain for the first derivative is simple (for the second derivative more regularity is needed and the formula is more complicated). Second of all we see that we can consider more general families of measures $\mu_t$, as long as weak* continuity is satisfied. In order to obtain the first derivative of the Lyapunov exponent we do not need the full differentiability of the family of measures $\mu_t$ is general, but only for the observable $\eta$ (one can think of it as the projection on just one coordinate), and this is trivially satisfied in the algebraic examples when $\eta$ is constant. Actually this is a necessary and sufficient condition for the existence of the first derivative of the Lyapunov exponent. The computation of the second derivative requires much more conditions, we ask for $\eta$ to be constant and $\mu_t$ to be basically differentiable, like in Dolgopyat's example (although only some "partial" differentiability may be sufficient for some specific perturbations).

\subsection{A second application: Differentiability of topological entropy.} As we said before, the accomplishment of Theorem \ref{teo:response12} is to obtain first derivatives of Lyapunov exponents in more generals situations. As an application we can obtain a remarkable result regarding the regularity of the topological entropy with respect to the map.

In the Anosov setting it is known that the topological entropy is locally constant for diffeomorphisms and smooth for flows (with respect to the map). However outside of the hyperbolic setting the results about differentiability of the topological entropy are very rare, and usually considering only specific perturbations. We are able to obtain differentiability of the topological entropy, in any direction, at a specific partially hyperbolic map: the time one map of a geodesic flow on a manifold of constant negative curvature. We also get a formula for the derivative.

\begin{maintheorem}\label{teo:responsegeo}
Let $f$ be the time one map of a geodesic flow on a manifold of constant negative curvature, and let $f_t$ be a $C^3$ family of diffeomorphisms with $f_0=f$. Then the map $t\mapsto h_{\rm top}(f_t)$ is differentiable at $t=0$, and the derivative is
\begin{equation}\label{eq:responsegeo}
\left.\frac{\partial}{\partial t}h_{\rm top}(f_t)\right|_{t=0}=-\int_M\omega_{E^{cs}}(\LL_XV_{E^u})d\mu,
\end{equation}
where $\mu$ is the Liouville measure, $X$ is the vector field tangent to the perturbation $h_t=f_t\circ f_0^{-1}$ at $t=0$, and $\omega_{E^{cs}}$ and $V_{E^u}$ are chosen as in the subsection~\ref{ssec:formulas} for the splitting $E^{cs}\oplus E^u$.
\end{maintheorem}

The result also holds for the time-$t$ map of the suspension flow over a linear Anosov map, for $t$ irrational.

\subsection{Families of flows}\label{ssec:flows}

We can obtain similar results if we consider families of flows instead of diffeomorphisms. A splitting is dominated for a $C^1$ flow $\phi$ if it is dominated for the time-one map of the flow $\phi_1$. The integrated Lyapunov exponents associated to an invariant bundle and an invariant measure is equal to the exponent corresponding to the time-one map of the flow, and the same bundle and measure.

Adapting the hypothesis \ref{H} to families of flows we get the hypothesis 

{\bf Hypothesis \referencia{HF}{(HF)}}: {\it 
\begin{enumerate}
\item
Let $X$ be a $C^r$vector field on the compact orientable Riemannian manifold $M$ of dimension $n$, such that the corresponding flow $\phi^X$ has a dominated splitting $TM=E^1_t\oplus E^2_t\oplus E^3_t$, with all the sub-bundles oriented and the orientations preserved by $D\phi^X$, $E=E^2$ and $F=E^1\oplus E^2$ are nontrivial, $E$ has dimension $k$.
\item
Let $X_t$, $t\in I\subset\mathbb R$, be a $C^r$ family of $C^r$ vectorfields passing through $X$ ($0\in I$, $X_0=X$).
\item
All the flows $\phi^{X_t}$ preserve the same Borel probability $\mu$.
\end{enumerate}}

We will use the same notations as before, using $\phi_1^{X_t}$ instead of $f_t$: the dominated splittings $TM=E_t\oplus F_t$ invariant for $\phi^{X_t}$, the Lyapunov exponent $\lambda(t)=\lambda(\phi_1^{X_t},E_t,\mu)$, the multivector field $V_E$ and the continuous form $\omega_F$. However the vectorfield $X$ now has another meaning.

Let $X'=\left.\frac{\partial}{\partial t}X_t\right|_{t=0}$. The role of the vectorfield $X$ tangent to the family of diffeomorphisms $f_t$ will be played now by the vectorfield $X'$, the tangent to the family $X_t$.


We have the following result:

\begin{maintheorem}\label{teo:flows}
Assume that \ref{HF} is satisfied for $r\geq 3$, $F$ is $C^{\alpha}$ and $E$ is $C^{\beta}$ on a neighborhood of the support of $\mu$, for some $\alpha,\beta\in[0,1]$. Then
\begin{enumerate}
\item
If $\alpha+\beta<1$ then $\lambda(t)=\lambda(0)+O(t^{\alpha+\beta})$;
\item
If $\alpha+\beta=1$ then $\lambda(t)=\lambda(0)+O(t\log t)$;
\item
If $\alpha+\beta>1$ then $\lambda(t)=\lambda(0)+t\lambda'(0)+O(t^{\alpha+\beta-1})$. Furthermore
\begin{equation}\label{eq:derLyapHf}
\lambda'(0)=\left.\frac{\partial}{\partial t}\int_M(\phi_t^{X'})^*\omega_F(V_E)d\mu\right|_{t=0}=\left.\frac{\partial}{\partial t}\int_M\omega_F((\phi^{X'}_t)_*V_E)d\mu\right|_{t=0}.
\end{equation}
\item
If $\alpha=1$, then $\lambda(t)$ is differentiable in $t=0$ and
\begin{equation}\label{eq:derLyap1Ff}
\lambda'(0)=\int \LL_{X'}\omega_F(V_E)d\mu.
\end{equation}
\item
If $\beta=1$, then $\lambda(t)$ is differentiable in $t=0$ and
\begin{equation}\label{eq:derLyap1Ef}
\lambda'(0)=-\int \omega_F(\LL_{X'}V_E)d\mu.
\end{equation}
\item
If $\alpha=\beta=1$ then $\lambda$ has expansion of order 2 at $t=0$.
\end{enumerate}
\end{maintheorem}

\subsection{Critical points of the integrated Lyapunov exponent}\label{ssec:critical}

Once we obtain formulas for the derivative of the integrated Lyapunov exponent which can be applied for large sets of diffeomorphisms, a natural question is what can we say about the critical points of $\lambda$. We say that a diffeomorphism $f$ is {\it critical (for the bundle $E$ which is part of a dominated splitting and the measure $\mu$)} if for any smooth family $f_t$, preserving the measure $\mu$ and passing through $f_0=f$, we have $\lambda'(0)=0$.

The following are some examples, the invariant measure $\mu$ is the volume, and the proofs of the claims are left to the reader:
\begin{enumerate}
\item
Linear automorphism of the torus: is critical for any bundle which is part of a dominated splitting.
\item
Time-one map of a volume preserving hyperbolic flow: critical for the central bundle, not critical for the stable and unstable bundles.
\item
Skew product over a volume preserving Anosov diffeomorphism, with rotations on the center fibers which are circles: critical for the center bundle, may be critical or not for the stable and unstable bundles.
\end{enumerate}

On the other hand, if for example $\mu$ is the Dirac measure at a fixed point, then there are no critical diffeomorphisms. This is why the study of the critical diffeomorphisms is more interesting when the invariant measure is the volume, or at least the bundle $E$ is not transversal to the support of $\mu$.

\subsubsection{Non-flat critical points.}

Our next results says basically that if $\mu$ is the volume, and a critical diffeomorphism has a $C^1$ splitting, then the critical point is non-degenerate or non-flat (the second derivative is nonzero). Given a family of diffeomorphisms $f_t$, satisfying the hypothesis \ref{H}, we denote by $\lambda_i(t)$ the integrated Lyapunov exponent of $f_t$ corresponding to $E^i_t$ with respect to the volume:
$$
\lambda_i(t)=\lambda(f_t,E_t^i,\mu)=\int_M\log\|Df_t^{\wedge \dim(E^i)}|_{E_t^i}\|d\mu.
$$

\begin{maintheorem}\label{teo:l2nonzero}
Assume  that $f$ is a $C^3$ volume preserving diffeomorphism on the compact manifold $M$. Assume that $f$ has a dominated splitting $TM=E^1\oplus E^2\oplus E^3$ which is $C^1$, with $E^2$ and $E^3$ nontrivial.

Then there exists a family of $C^{\infty}$ diffeomorphisms $h_t$, with $f_t=h_t\circ f$ satisfying the hypothesis \ref{H} for $r=3$ and $\mu$ equal to the volume, such that
\begin{equation}
\lambda_3''(0)<0\qquad \mbox{ and }\qquad\lambda_2''(0)>0.
\end{equation}
\end{maintheorem}

Theorem~\ref{teo:l2nonzero} generalizes classical results obtained previously by Shub-Wilkinson, Ruelle, Dolgopyat and Dolgopyat-Pesin (see \cite{SW2000,R2003,D2004,DP2002}). Like in the papers mentioned above, it can be used in order to remove zero exponents for volume preserving partially hyperbolic diffeomorphisms by arbitrarily small $C^{\infty}$ perturbations. In particular one can obtain nonuniform hyperbolicity, as well as pathological center and intermediate foliations, by doing arbitrarily small $C^{\infty}$ perturbations of diffeomorphisms with $C^1$ dominated splittings. For example this is the case of partially hyperbolic automorphisms on nilmanifolds, or products of volume preserving codimension one Anosov maps with rotations.

It seems very probable that the result can be adapted to more general situations, for example if we assume that only $E^2$ and $E^3$ are smooth, and uses a perturbation in the direction of $E^2\oplus E^3$. For example this could be the case of skew products over volume preserving codimension one Anosov maps, where the fibers are circles and the fiber maps are rotations. It is worth mentioning that this construction was already known in the $C^1$ topology from \cite{BB2003}, so the novelty here is the use of $C^{\infty}$ small perturbations.

\subsubsection{Critical points and rigidity.}

The last result says that if again $\mu$ is the volume, and the critical diffeomorphism has a sufficiently smooth splitting forming two transversal foliations, then the critical point is rigid in the sense that the volume disintegrates as a true product along the two complimentary foliations.

\begin{maintheorem}\label{teo:rigidity}
Assume  that $f$ is a $C^3$ volume preserving diffeomorphism on the compact manifold $M$. Assume that $f$ has a dominated splitting $TM=E^1\oplus E^2\oplus E^3$, $E:=E^2$ and $F:=E^1\oplus E^3$  integrate to complimentary foliations $\mathcal W^E$ and $\mathcal W^F$. Assume also that $F$ is $C^1$, $\dim E=1$, and the foliation $\mathcal W^E$ has $C^2$ leaves, it is absolutely continuous, and the densities of the disintegrations of the volume along the $\mathcal W^E$-leaves are $C^1$ along the $\mathcal W^E$-leaves.

If $f$ is a critical diffeomorphism for $E$ and the volume, then the disintegrations of the volume along $\mathcal W^E$ are invariant under $\mathcal W^F$-holonomy.
\end{maintheorem}

Les us make a few remarks on this result.

\begin{rem}
The condition required on the bundle $E$ is satisfied if $E$ is $C^1$, or more generally if $E$ is the unstable (or stable) bundle of a $C^2$ diffeomorphism.
\end{rem}
\begin{rem}
The disintegrations of the volume along $\mathcal W^E$ are invariant under $\mathcal W^F$-holonomy if and only if the disintegrations of the volume along $\mathcal W^F$ are invariant under $\mathcal W^E$-holonomy, or we say that the volume is a "true product".
\end{rem}
\begin{rem}
The conclusion of Theorem~\ref{teo:rigidity} is similar in some sense to the "Invariance Principle"-type results, see for example \cite{AV2010,ASV2013}, etc. In these results zero center exponents would imply that the disintegrations along the center foliations are invariant under the stable and unstable holonomies. It is interesting that the criticality of an exponent will also imply a similar conclusion (of course we require stronger regularity assumptions).
 
\end{rem}

\subsection{A third application: Rigidity of critical points for Lyapunov exponents for conservative Anosov surface diffeomorphisms.} The stable and unstable bundles of area preserving Anosov diffeomorphisms in dimension 2 are $C^{2-}$ (that means diffeomorphisms of class $C^r$, for all $0\leq r<2$), so the stable and unstable Lyapunov exponents $\lambda_s$ and $\lambda_u$ (with respect to the area) are differentiable with respect to the parameter along one-parameter families. We obtain the following corollary.

\begin{maincor}\label{cor:rigidityA}
The critical diffeomorphisms for the unstable (stable) Lyapunov exponent with respect to the area, in the space of $C^{\infty}$ area preserving Anosov diffeomorphisms of the two-torus homotopic to the linear map $L$, are $C^{\infty}$ conjugated to $L$ (in particular they are the global maximum).
\end{maincor}

This corollary answers a conjecture from \cite{GT2014}. In fact the question posed in \cite{GT2014} is much weaker, they asked wether local maximality of the unstable exponent implies rigidity. A local maximum of the unstable exponent is automatically a critical point.

\subsection{Some further questions}

In this subsection we will mention some further questions which we consider interesting.

\begin{enumerate}
{\it 
\item
How optimal are our results? The results on the regularity of invariant bundles in terms of the contraction and expansion rates are in general optimal, so the $\alpha$ and $\beta$ from our hypothesis are finite, and in general small. Our method seems to be limited in the sense that the maximum regularity of the Lyapunov exponent which we can obtain is $\alpha+\beta$. But is this indeed optimal? We don't know in fact any example of a $C^{\infty}$ family of $C^{\infty}$ diffeomorphisms with a dominated splitting such that the integrated Lyapunov exponent corresponding to a sub-bundle is not $C^{\infty}$. Does such an example exist?
\item
Our formulas for the first derivative of the Lyapunov involve only the two bundles of the dominated splitting (and the measure). Thus the problem of finding and understanding the critical points translates into a purely geometric/analytic question. If the two bundles are $C^1$ and integrable, and the measure is the volume, criticality means that the volume decomposes as a "true product" along the 2 foliations. Does a similar statement hold if the two bundles are only H\"{o}lder, with the sum of the exponents bigger than 1? What about if one is H\"{o}lder and one is smooth? What happens if the bundles are not integrable?
\item
If $f$ is partially hyperbolic, critical for the unstable bundle and the volume, and the splitting is $C^1$, is it true that $f$ is a (local) maximum for the unstable Lyapunov exponent? The proof of Theorem~\ref{teo:l2nonzero} suggests that this is the case. For many perturbations supported on small enough neighborhoods of non-periodic points the second derivative of the unstable Lyapunov exponent is negative (recall that the second derivative is bilinear in $X$).
\item
One could definitely obtain better regularity results for the Lyapunov exponents if one considers special families of perturbations. Is it possible to apply this remark in order to remove zero exponents in new and interesting situations?
\item
It seems possible to obtain further results in the case of variable measures, assuming only H\"{o}lder regularity of the bundles, but assuming in exchange better regularity of the invariant measures with respect to the parameters. Can one obtain results in this direction for relevant dynamical measures, like the SRB measures, Gibbs u-states, or measures of maximal entropy?}

\end{enumerate}

\subsection{Manuscript organization} In the next section we provide  some definitions and some preparative results. We begin giving more details about the definition of the $k$-form $\omega_F$ and the $k$-multivector field $V_E$  and the dependence of the regularity from the smoothness  of $F$ and $E$ respectively (see subsection~\ref{ssec:omegaV}). Then (subsection~\ref{ssec:XY})  we describe the family of ``tangent'' vectors fields $X$ and $Y$ for the family $h_t=f_t\circ f^{-1}$ and their the rol in the first (and second) order ``Taylor expansion'' (on  local charts) for the family $f_t$. We prove also that the flow $\phi^X_t$ (resp. $\phi^Y_t$) generated by $X_t$ (resp. by $Y_t$) preserves $\mu$. In this proof appears for the first time, the connection with the Lie derivative. Using this information we obtain a ``Taylor expansion'' for the maps $t\to (f_t)^*\omega_F$ and $t\to (f_t)_*V_E$ (see subsection~\ref{ssec:taylor}) and finally we investigate the regularity of of the map $t\mapsto V_t$ which is equivalent to finding the regularity of the map $t\mapsto E_t$ and we obtain some formulas for their derivatives (subsection~\ref{ssec:regularidadVt}).

Theorem~\ref{teo:flow} is proved in Section~\ref{sec:regularidad}. In Section~\ref{sec:holder} we prove Theorem~\ref{teo:lyapunovH} and Theorem~\ref{cor:gflow}. The obtention of the formulas for the derivatives and the proof of Theorem~\ref{teo:lyapunov12} are presented in Section~\ref{sec:demteoD}. In Section~\ref{sec:varmeasure} we deal with the case when the invariant measure $\mu_t$ depends on the map $f_t$ and we prove Theorem~\ref{teo:response12} and Theorem~\ref{teo:responsegeo}. Section~\ref{sec:flows} is devoted to flows  and we present the proof for Theorem~\ref{teo:flows}. In section~\ref{sec:nonvanish} we study the case of non-flat critical points, that means when non-vanishing of the second derivative, and we  prove Theorem~\ref{teo:l2nonzero}. Section~\ref{sec:critandrig} is dedicated to the study of critical points and rigidity. There we prove Theorem~\ref{teo:rigidity} and Corollary~\ref{cor:rigidityA}.


\section{Definitions and preliminary results}\label{sec:setting}

\subsection{The form $\omega_F$ and the multi--vector fields $V_t$. }\label{ssec:omegaV}
	
Let $TM=E_t\oplus F_t$ be continuous splittings (on $M$ and the parameter $t$) such that $E=E_0$ is $C^{\beta}$ and $F=F_0$ is $C\sp {\alpha}$ with $\alpha, \beta\geq0$. Let $k:=\dim E$. We assume that $M$, $E_t$ and $F_t$ are orientable. We claim that there exist a continuous $k$--form $\omega_F$ on $M$, and continuous $k$-multivector fields $V_t$ such that:
\begin{itemize}
\item[1.] $\omega_F$ is $C\sp {\alpha}$ and $\ker\omega_F(p)=F(p)\wedge T_pM\sp{\wedge (k-1)}$,
\item[2.] $V_t(p)\in E(p)^{\wedge k}$, and $V_E=V_0$ is $C\sp {\beta}$,
\item[3.] $\omega_F(V_t)=1$ for every $t\in I$ (eventually for a smaller interval $I$).
\end{itemize}

A sketch of the proof is the following.

Given any smooth chart $U\subset M$ where $E$ and $F$ are parallelizable, one can choose a $C^{\beta}$ positively oriented base of $E$ to be $\{V_1,V_2,\dots V_k\}$ and a $C^{\alpha}$ positively oriented base of $F$ to be $\{V_{k+1},\dots V_d\}$. Let $V^U:=V_1\wedge V_2\wedge\dots\wedge V_k$ be a nonzero $C^{\beta}$ $k$-multivector field in $ E^{\wedge k}$ inside the chart $U$. Using a finite covering of $M$ with such charts, and a smooth partition of unity, one can construct a nonzero $C^{\beta}$ $k$-multivector field $\tilde V_E$ in $ E^{\wedge k}$ on the entire $M$.

In a similar way we can construct continuous nonzero $k$-multivector field $\tilde V_t$ in $ E_t^{\wedge k}$ on the entire $M$. Since $t\mapsto E_t$ is continuous, we can choose $\tilde V_t$ such that $t\mapsto\tilde V_t$ is also continuous.

Let $\mu$ be a nonzero smooth $d$-form on $M$. Let $\omega_F^U:=i_{V_{k+1}}\dots i_{V_d}\mu$ be a nonzero $C^{\alpha}$ $k$-form inside the chart $U$ with the kernel $F\wedge TM\sp{\wedge (k-1)}$ ($i_V\mu$ is the interior product of $\mu$ with $V$). Using again a finite covering of $M$ with such charts, and a smooth partition of unity, one can construct a nonzero $C^{\alpha}$ $k$-form $\tilde \omega_F$ with the kernel $F\wedge TM\sp{\wedge (k-1)}$ on the entire $M$.

The transversality of $E$ and $F$ guarantees that $\tilde\omega_F(\tilde V_E)$ is nonzero. If $\alpha\geq\beta$, we just let $\omega_F=\tilde\omega_F$ and $V_E=\frac 1{\tilde\omega_F(\tilde V_E)}\tilde V_E$. Otherwise we let $V_E=\tilde V_E$ and $\omega_F=\frac 1{\tilde\omega_F(\tilde V_E)}\tilde \omega_F$. Thus we get $\omega_F$ to be $C^{\alpha}$, $V_E$ to be $C^{\beta}$, and $\omega_F(V_E)=1$.

Since $t\mapsto\tilde V_t$ is continuous, eventually after restricting $I$ we can assume that $\omega_F(\tilde V_t)$ is nonzero. Let $V_t=\frac 1{\omega_F(\tilde V_t)}\tilde V_t$, so $\omega_F(V_t)=1$. Also from construction we have $t\mapsto V_t$ is in fact continuous.

\begin{rem}
Let us remark that the choice of $\omega_F$ and $V_E$ is not unique. Given any function $h:M\rightarrow (0,\infty)$ of class $C^{\max\{\alpha,\beta\}}$, we can replace $\omega_F$ and $V_t$ by $h\omega_F$ and $\frac 1hV_t$.
\end{rem}

{\bf Notations:} For simplicity in the rest of the paper we will use the notations $\omega:=\omega_F$ and $V:=V_E$, if no confusion can be made.

\subsection{The ``tangent'' vector fields $X$ and $Y$ for the family $h_t:=f_t\circ f_0^{-1}$.}\label{ssec:XY}

Suppose that we have a $C^r$--family of diffeomorphism $(h_t)_{t\in I}$ on $M$ such that $h=h_0=Id$ is the identity on $M$. We are interested in approximating $h_t$ by flows.

Define the $C^{r-1}$ vector field $X$ on $M$ tangent to the family $h_t$ in $t=0$ by
\begin{equation}
X(p)=\left.\frac{\partial}{\partial t}h_t(p)\right|_{t=0}.
\end{equation}

If $r\geq 2$ then $X$ is $C^1$ and will generate a flow which we denote $\phi^X_t$. The flow $\phi_t^X$ is a good approximation of first order for the family $h_t$. The following lemma is straightforward. 

\begin{lema}\label{1flow}
Under the above conditions, the following relations hold uniformly in any charts:
\begin{enumerate}
\item
If $r\geq 2$ then
\begin{equation}\label{eq:ht1}
h_t(p)=\phi_t^X(p)+O(t^2)\quad(=p+tX(p)+O(t^2));
\end{equation}
and
\begin{equation}\label{eq:Dht1}
Dh_t(p)=D\phi_t^X(p)+o(t)\quad(=Id+tDX(p)+o(t));
\end{equation}
\item
If $2<r\leq 3$ then
\begin{equation}\label{eq:Dht12}
Dh_t(p)=D\phi_t^X(p)+O(t^{r-1})\quad(=Id+tDX(p)+O(t^{r-1})).
\end{equation}
\end{enumerate}
\end{lema}

In order to obtain a better approximation of $h_t$ (up to order two), we need to introduce the vector field $Y$, which can be seen as a ``second order correction of the flow''. An intrinsic way of defining $Y$ is the following.

For $r\geq 1$, define the $C^{r-1}$ vector fields $X_t$ ``tangent'' to each $h_t$: 
\begin{equation}\label{eq:Xt}
X_{t}(p)=\left.\frac{\partial}{\partial s}h_t^{-1}(h_{t+s}(p))\right|_{s=0}=Dh_t^{-1}(h_t(p))\cdot\frac{\partial}{\partial t}h_t(p)=\left[Dh_t(p)\right]^{-1}\cdot\frac{\partial}{\partial t}h_t(p).
\end{equation}
Clearly we have that $X=X_0$. If $r\geq 2$, then we can differentiate $X_t$ with respect to $t$ and we obtain the vector fields $Y_t$:
\begin{equation}\label{eq:Yt}
Y_{t}(p)=\lim_{s\to0}\frac{X_{s+t}(p)-X_{t}(p)}{s}=\frac{\partial}{\partial t}X_t(p).
\end{equation}

Let $\quad Y:=Y_0$. We can give a formula for $Y$ in local charts. Suppose that in some chart we have
$$
h(p)=p,\quad \left.\frac \partial{\partial t}h_t(p)\right|_{t=0}=X(p),\quad \mbox{and}\, \, \,  \left.\frac {\partial^2}{\partial t^2}h_t(p)\right|_{t=0}=Z(p)
$$
where $X,Z:\mathbb R^n\to\mathbb R^n$ are $C^{r-1}$ respectively $C^{r-2}$. This means that we can write
\begin{equation}\label{eq:locchart}
h_t(p)=p+tX(p)+\frac{t\sp2}{2}Z(p)+R(t,p)
\end{equation}
where $R(t,p)=o(t\sp2)$ uniformly on $p$.

The vector field $X$ is independent of the choice of the chart, however $Z$ is not (this is why we use $Y$ and not $Z$). We claim that
\begin{equation}\label{eq:Y_eq}
Y=Z-DX\cdot X.
\end{equation}
In order to see this, we compute $Y$:
\begin{align*}
Y(p)&=\left.\frac{\partial}{\partial t}X_t(p)\right|_{t=0}=\frac{\partial}{\partial t}\left.\left[\left[Dh_t(p)\right]^{-1}\cdot\frac{\partial}{\partial t}h_t(p)\right]\right|_{t=0}\\
&=\left[Dh_0(p)\right]^{-1}Z(p)+\left.\frac{\partial}{\partial t}\left[Dh_t(p)^{-1}\right]\right|_{t=0}\cdot X(p)\\
&=Z(p)-\left.\left[ Dh_t(p)^{-1}\cdot\frac{\partial}{\partial t}Dh_t(p)\cdot Dh_t(p)^{-1}\right]\right|_{t=0}\cdot X(p)\\
&=Z(p)-D\left(\left.\frac{\partial}{\partial t}h_t(p)\right|_{t=0}\right)\cdot X(p)\\
&=Z(p)-DX(p)\cdot X(p),
\end{align*}
since $Dh(p)=Id$, the derivative of the inverse of a matrix function satisfies $(A(t)^{-1})'=A(t)^{-1}\cdot A'(t)\cdot A(t)^{-1}$, and the partial derivatives commute, $\frac{\partial}{\partial t}Dh_t(p)|_{t=0}=D\left(\frac{\partial}{\partial t}h_t|_{t=0}\right)$ (remember that $r\geq2$).

\begin{rem}
We remark that the vector fields $X$ and $Y$ allow to approximate the parametric family $h_t$ with a composition of flows. In fact, if $r\geq 3$, the flows $\phi^X$ and $\phi^Y$ generated by $X$ and $Y$ are well defined. Then we have
$$
h_t(p)=\phi_t^X(\phi_{\frac{t\sp2}2}^Y(p))+o(t^2)\quad\hbox{if } r\geq 3,
$$
$$
Dh_t(p)=D\left[\phi_t^X(\phi_{\frac{t\sp2}2}^Y(p))\right]+o(t^2)\quad\hbox{if } r\geq 4,
$$
in any chart and uniformly in $p$.

The proof is straightforward, one just has to check that the first two derivatives (with respect to $t$) of both sides of the equations coincide in $t=0$. One can also approximate the family $h_t$ with $\phi^Y_{\frac{t^2}2}\circ\phi^X_t$, for $r$ sufficiently large.
\end{rem}

An important observation is the following.

\begin{lema}\label{lem:XYpresmu}
Suppose that $h_t$ preserves the Borel probability $\mu$ for all $t\in I$. If $r\geq 2$ then $\phi^{X_t}$ preserves $\mu$, for any $t\in I$. If $r\geq 3$, then $\phi^{Y_t}$ also preserves $\mu$, for any $t\in I$. In particular, if $\mu$ is the volume on $M$ and $r\geq 3$, then the vector fields $X$ and $Y$ are divergence--free.
\end{lema}

\begin{proof}
Recall that $f$ preserves a measure $\mu$ if and only if $\int_M gd\mu=\int_Mg\circ fd\mu$ for any $C^0$ function $g:M\rightarrow\mathbb R$. Since the $C^1$ functions are dense in the space of $C^0$ functions, this is equivalent to $\int_M gd\mu=\int_Mg\circ fd\mu$ for any $C^1$ function $g:M\rightarrow\mathbb R$.

If a vector field $\upchi$ is differentiable and generates the flow $\phi^\upchi$, then $\phi^\upchi$ preserves $\mu$ if and only if $\int_M gd\mu=\int_Mg\circ \phi^\upchi_sd\mu$ for any $C^1$ function $g:M\rightarrow\mathbb R$ an any $s\in\mathbb R$. This in turn is equivalent to 
\begin{equation}\label{eq:Zpresmu}
\left.\frac{\partial}{\partial s}\int_Mg(\phi^\upchi_s(p))d\mu\right|_{s=0}=\int_M\left(\left.\frac{\partial}{\partial s}g(\phi^\upchi_s(p))\right|_{s=0}\right)d\mu=\int_M\LL_\upchi g\,d\mu=0,
\end{equation}
for any $C^1$ function $g:M\rightarrow\mathbb R$ ($\LL_\upchi$ is the Lie derivative).

Take some $g:M\rightarrow\mathbb R$ of class $C^1$. If every $h_t$ preserves $\mu$, then $\int_Mg\circ h_td\mu=\int_Mgd\mu$ is constant. Recall that by \eqref{eq:Xt} we have that $\frac{\partial}{\partial t}h_t(p)=Dh_t(p)X_t(p)$. We have
\[
\begin{aligned}
0&=\frac{\partial}{\partial t}\int_Mg(h_t(p))d\mu=\int_MDg(h_t(p))\frac{\partial}{\partial t}h_t(p)d\mu\\
&=\int_MDg(h_t(p))Dh_t(p)X_t(p)d\mu=\int_MD(g\circ h_t)(p)\cdot X_t(p)\\
&=\int_M\LL_{X_t}(g\circ h_t)d\mu.
\end{aligned}
\]
Since $h_t$ is a diffeomorphism, this means that for each $t$ and for any $C^1$ function $\tilde g=g\circ h_t$ we have
$$
\int_M\LL_{X_t}\tilde gd\mu=0,
$$
so by \eqref{eq:Zpresmu} we have that $\phi^{X_t}$ preserves $\mu$ for every $t$.

The Lie derivative is linear with respect to the vector fields, so by \eqref{eq:Zpresmu}, the flows generated by $(X_t-X_s)/(t-s)$ preserve $\mu$ for all $t,s\in I$. The Lie derivative is also continuous with respect to the vector field, so from the definition of $Y_t$ (recall (\ref{eq:Yt})) we get that the flows generated by $Y_t$ also preserve $\mu$.
\end{proof}


\subsection{Expansions for $t\mapsto h_t^*\omega$ and $t\mapsto h_{t*}V$ at $t=0$.}\label{ssec:taylor}

We refere again to figure Figure~\ref{dib:perturb} for an intuitive represententation of $\omega_F$ and $V_t$, and the action induced on them by the derivative of $f$ and $h_t$. 

Let $\Omega^k(M)$ be the Banach space of continuous $k$-forms on $M$. If $\omega\in\Omega^k(M)$, its norm is defined by
$$\|\omega\|=\sup\{\omega_p(v_1,\dots,v_k)\::\: p\in M,v_i\in T_pM,\|v_i\|=1, i\in\{1,\dots, k\}\}.$$
Also let $\mathcal X^{k}(M)$ be the Banach space of continuous $k$-multivector fields on $M$. If $V\in\mathcal X^{k}(M)$ then its norm is
$$\|V\|=\sup_{p\in M}\|V(p)\|,$$
where $\|V(p)\|$ is the usual norm on the exterior product of $T_pM$. Let us remark that the pairing $(\omega,V)\mapsto\omega(V)$ is bilinear and continuous with values in $C^0(M)$.

We are interested in the Frechet differentiability of the maps $h_t^*\omega:I\subset\mathbb R\rightarrow\Omega^k(M)$ and $h_{t*}V:I\rightarrow\mathcal X^{k}(M)$. Using smooth partitions of unity, one can see that it is sufficient to check the regularity of the maps in local charts.

We have the following lemma which is fundamental to our future considerations. Recall that $r$ is the regularity of the family $h_t$, and let $X$ and $Y$ the vector fields tangent to the family $h_t$ defined in the previous subsection.

\begin{lema}\label{le:h} Let $\omega\in \Omega^k(M)$ be a continuous $k$--form.
\begin{enumerate}
\item If $\omega$ is $C^{\alpha}$ and $r\geq\alpha+1$ then $t\mapsto h_t^*\omega$ is $C^{\alpha}$.
\item If $\omega$ is $C^1$ and $r\geq 2$, then $t\mapsto h_t^*\omega$ is Frechet differentiable and the derivative in zero is $\LL_X\omega$:
\begin{equation}\label{eq:derivadahw}
h_t^*\omega=\omega+t\LL_X\omega+o(t).
\end{equation}
\item If $\omega$ is $C^2$ and $r\geq 3$, then $t\mapsto h_t^*\omega$ is twice Frechet differentiable and the second derivative in zero is $\LL_X\LL_X\omega+\LL_Y\omega$:
\begin{equation}\label{eq:derivadahw2}
h_t^*\omega=\omega+t\LL_X\omega+\frac{t^2}2(\LL_X\LL_X\omega+\LL_Y\omega)+o(t^2).
\end{equation}
\end{enumerate}
\end{lema}

\begin{proof}
The part (i) and the differentiability claims follow directly from the formulas of the pullback of a form in local coordinates.

For the parts (ii) and (iii) we just have to check that if $\omega$ is $C^1$ then $\left.\frac{\partial}{\partial t}h_t^*\omega\right|_{t=0}=\LL_X\omega$, and if $\omega$ is $C^2$ then $\left.\frac{\partial^2}{\partial t^2}h_t^*\omega\right|_{t=0}=\LL_X\LL_X\omega+\LL_Y\omega$. Let us make first the following remarks.

Observe first that if (ii), (iii) are true for the forms $\omega_1$ and $\omega_2$ then they are also true for the form $\omega_1+\omega_2$, because the Lie derivative is linear:
\begin{align*}
\left.\frac{\partial}{\partial t}h_t^*(\omega_1+\omega_2)\right|_{t=0}&=\left.\frac{\partial}{\partial t}(h_t^*\omega_1)\right|_{t=0}+\left.\frac{\partial}{\partial t}(h_t^*\omega_2)\right|_{t=0}\\ \\
&=\LL_X\omega_1+\LL_X\omega_2\\
&=\LL_X(\omega_1+\omega_2),
\end{align*}
\begin{eqnarray*}
\left.\frac{\partial^2}{\partial t^2}h_t^*(\omega_1+\omega_2)\right|_{t=0}&=&\left.\frac{\partial^2}{\partial t^2}(h_t^*\omega_1)\right|_{t=0}+\left.\frac{\partial^2}{\partial t^2}(h_t^*\omega_2)\right|_{t=0}\\ \\
&=&\Big(\LL_X\LL_X\omega_1+\LL_Y\omega_1\Big)+\Big(\LL_X\LL_X\omega_2+\LL_Y\omega_2\Big)\\
&=&\LL_X\LL_X(\omega_1+\omega_2)+\LL_Y(\omega_1+\omega_2).
\end{eqnarray*}

Observe also that if (ii), (iii) are true for the forms $\omega_1$ and $\omega_2$ then they are also true for the form $\omega_1\wedge\omega_2$, because the derivatives obey the Leibniz rule.
\begin{align*}
\left.\frac{\partial}{\partial t}h_t^*(\omega_1\wedge\omega_2)\right|_{t=0}&=\left.\frac{\partial}{\partial t}h_t^*\omega_1\wedge h_t^*\omega_2\right|_{t=0}+\left.h_t^*\omega_1\wedge\frac{\partial}{\partial t}h_t^*\omega_2\right|_{t=0}\\
&=\LL_X\omega_1\wedge\omega_2+\omega_1\wedge\LL_X\omega_2=\LL_X(\omega_1\wedge\omega_2),
\end{align*}
\begin{align*}
\left.\frac{\partial^2}{\partial t^2}h_t^*(\omega_1\wedge\omega_2)\right|_{t=0}&=\left.\left[\frac{\partial^2}{\partial t^2}h_t^*\omega_1\wedge h_t^*\omega_2+2\frac{\partial}{\partial t}h_t^*\omega_1\wedge\frac{\partial}{\partial t}h_t^*\omega_2+h_t^*\omega_1\wedge\frac{\partial^2}{\partial t^2}h_t^*\omega_2\right]\right|_{t=0}\\
&=(\LL_X\LL_X\omega_1+\LL_Y\omega_1)\wedge\omega_2+2\LL_X\omega_1\wedge\LL_X\omega_2+\omega_1\wedge (\LL_X\LL_X\omega_2+\LL_Y\omega_2)\\
&=\LL_X\LL_X(\omega_1\wedge\omega_2)+\LL_Y(\omega_1\wedge\omega_2).
\end{align*}

The formulas in (ii), (iii) are local, and it is sufficient to verify them in a chart $U\subset\mathbb R^n$, where any form can be decomposed into a sum of forms $gdx_{i_1}\wedge\dots\wedge dx_{i_k}$. The two remarks above show that we only need to verify (ii) and respectively (iii) for a zero form $g$ of class $C^1$ respectively $C^2$, and for the one-forms $dx_i$, or more generally for a one-form $dg$ with $g$ of class $C^{\infty}$.

Let us prove first (ii) for a map $g:U\rightarrow\mathbb R$ of class $C^1$, and $h_t$ of class $C^2$, meaning that $X$ is $C^1$. Then $h_t^*g(p)=g(h_t(p))$ and
$$
\left.\frac\partial{\partial t}h_t^*g(p)\right|_{t=0}=\left.\frac\partial{\partial t}g(h_t(p))\right|_{t=0}=Dg(p)\cdot X(p)=\LL_Xg(p),
$$
so indeed $\left.\frac\partial{\partial t}h_t^*g\right|_{t=0}=\LL_Xg$.

Now let us prove (iii) for a map $g:U\rightarrow\mathbb R$ is of class $C^2$, and $h_t$ of class $C^3$, meaning that $X$ is $C^2$ and $Y$ is $C^1$. We have that
\begin{align*}
\frac{\partial^2}{\partial t^2}h_t^*g(p)&=\frac{\partial^2}{\partial t^2}g(h_t(p))=\frac\partial{\partial t}\left[Dg(h_t(p))\cdot\frac\partial{\partial t}h_t(p)\right]\\
&=D^2g(h_t(p))\left( \frac\partial{\partial t}h_t(p),\frac\partial{\partial t}h_t(p)\right)+Dg(h_t(p))\cdot\frac{\partial^2}{\partial t^2}h_t(p),
\end{align*}
and in $t=0$, we get
$$
\left.\frac{\partial^2}{\partial t^2}h_t^*g(p)\right|_{t=0}=D^2g(p)\left( X(p), X(p)\right)+Dg(p)\cdot Z(p).
$$
On the other hand
\begin{align*}
\LL_X\LL_Xg(p)&=\LL_X\left (Dg(p)\cdot X(p)\right)=D\left( Dg(p)\cdot X(p)\right)\cdot X(p)\\
&=D^2g(p)\left( X(p), X(p)\right)+Dg(p)\cdot DX(p)\cdot X(p)
\end{align*}
and, using \eqref{eq:Y_eq},
$$
\LL_Yg(p)=Dg(p)\cdot Y=Dg(p)\cdot Z(p)-Dg(p)\cdot DX(p)\cdot X(p).
$$
Combining the last 3 equalities we get that
$$
\left.\frac{\partial^2}{\partial t^2}h_t^*g(p)\right|_{t=0}=\LL_X\LL_Xg(p)+\LL_Yg(p).
$$

Now consider a zero form $g$ of class $C^{\infty}$. Recall that the exterior derivative commutes with the pullback ($h_t^*(dg)=d(h_t^*g)$), and with the Lie derivative ($\LL_X(dg)=d\left(\LL_Xg\right)$).

Let us prove (ii) for $dg$ given $h_t$ of class $C^2$. The map $(t,p)\mapsto g(h_t((p))$ is $C^2$ in both $t$ and $p$, so the partial derivatives commute: $d\frac{\partial}{\partial t}g(h_t(p))=\frac{\partial}{\partial t}dg(h_t(p))$. Then

$$
\left.\frac{\partial}{\partial t}h_t^*(dg)\right|_{t=0}=\left.\frac{\partial}{\partial t}d(g\circ h_t)\right|_{t=0}=\left.d\frac{\partial}{\partial t}(g\circ h_t)\right|_{t=0}=d\LL_Xg=\LL_Xdg.
$$

Now let us prove (iii) for $dg$, given $h_t$ of class $C^3$. The map $(t,p)\mapsto g(h_t((p))$ is $C^3$ in $t$ and $p$, so the following partial derivatives commute: $d\frac{\partial^2}{\partial t^2}g(h_t(p))=\frac{\partial^2}{\partial t^2}dg(h_t(p))$. Then
$$
\frac{\partial^2}{\partial t^2}h_t^*(dg)=\frac{\partial^2}{\partial t^2}d(g\circ h_t)=d\frac{\partial^2}{\partial t^2}(g\circ h_t),
$$
and
$$
\left.\frac{\partial^2}{\partial t^2}h_t^*(dg)\right|_{t=0}=d(\LL_X\LL_Xg+\LL_Yg)=\LL_X\LL_Xdg+\LL_Ydg.
$$

This finishes the proof of the lemma.
\end{proof}

\begin{rem}\label{rem:derivativet}
The formulas \eqref{eq:derivadahw} and \eqref{eq:derivadahw2} give us the derivatives of $t\mapsto h_t^*\omega$ in $t=0$. One can use these formulas in order to obtain formulas at any $t_0$, using the observation that we now use the map $t\mapsto h_{t_0+t}^*\omega$. Then $\omega$ is replaced by $h_{t_0}^*\omega$, $X$ is replaced by $X_{t_0}$, and $Y$ is replaced by $Y_{t_0}$, and the formulas of the derivatives are:
$$
\left.\frac{\partial}{\partial t}h_t^*\omega\right|_{t=t_0}=\LL_{X_{t_0}}h_{t_0}^*\omega,
$$
$$
\left.\frac{\partial^2}{\partial t^2}h_t^*\omega\right|_{t=t_0}=\LL_{X_{t_0}}\LL_{X_{t_0}}h_{t_0}^*\omega+\LL_{Y_{t_0}}h_{t_0}^*\omega.
$$
\end{rem}

One can obtain similar formulas for multivector fields instead of differential forms.

\begin{lema}\label{lem:taylorhV}
Let $V$ be a continuous $k$-multivector field $V$ on $M$.
\begin{enumerate}
\item If $V$ is $C^{\beta}$ and $r\geq\beta+1$ then $t\mapsto h_{t*}V$ is $C^{\beta}$.
\item If $V$ is $C^1$ and $r\geq 2$, then $t\mapsto h_{t*}V$ is Frechet differentiable and the derivative in zero is $-\LL_XV$:
\begin{equation}\label{eq:taylorV1}
h_{t*}V=V-t\LL_XV+o(t).
\end{equation}
\item If $V$ is $C^2$ and $r\geq 3$, then $t\mapsto h_{t*}V$ is twice Frechet differentiable and the second derivative in zero is $\LL_X\LL_XV-\LL_YV$:
\begin{equation}\label{eq:taylorV2}
h_{t*}V=V-t\LL_X(V)+\frac{t^2}2\left( \LL_X\LL_X(V)-\LL_Y(V)\right)+o(t^2).
\end{equation}
\end{enumerate}
\end{lema}

\begin{proof}

Again, like in the case of forms, the part (i) and the differentiability claims are immediate. For the parts (ii) and (iii) we have to check again that $\left.\frac{\partial}{\partial t}h_{t*}V\right|_{t=0}=-\LL_XV$ if $V$ is $C^1$, and $\left.\frac{\partial^2}{\partial t^2}h_{t*}V\right|_{t=0}=\LL_X\LL_XV-\LL_YV$ if $V$ is $C^2$.

One can prove the claims directly for vector fields, and using the Leibniz rule extend the result for multivector fields, similar to the proof of Lemma~\ref{le:h}. We will give a proof using Lemma~\ref{le:h} and the duality between forms and multivector fields. Let us remark that $\omega(h_{t*}V)=h_t^*\omega(V)\circ h_t^{-1}$ (where $\omega(h_{t*}V)$ and $h_t^*\omega(V)$ are seen as maps from $M$ to $\mathbb R$).

Assume first that $V$ is $C^1$. It is easy to see that $\left.\frac{\partial}{\partial t}h_t^{-1}\right|_{t=0}=-X$. For any $C^1$ form $\omega$ we have:
\begin{align*}
\left.\omega\left(\frac{\partial}{\partial t}h_{t*}V\right)\right|_{t=0}&=\left.\frac{\partial}{\partial t}\omega(h_{t*}V)\right|_{t=0}=\left.\frac{\partial}{\partial t}\left[h_t^*\omega(V)\circ h_t^{-1}\right]\right|_{t=0}\\
&=\left.\frac{\partial}{\partial t}\left[h_t^*\omega(V)\right]\circ h_t^{-1}\right|_{t=0}+\left.d\left[h_t^*\omega(V)\right]\left(\frac{\partial}{\partial t}h_t^{-1}\right)\right|_{t=0}\\
&=\LL_X\omega(V)+d(\omega(V))(-X)=\LL_X[\omega(V)]-\omega(\LL_XV)-\LL_X[\omega(V)]\\
&=-\omega(\LL_XV).
\end{align*}

This clearly implies that $\left.\left(\frac{\partial}{\partial t}h_{t*}V\right)\right|_{t=0}=-\LL_XV$.

The proof of the formula \eqref{eq:taylorV2} is similar, and since we do not need it in our future considerations, we omit the proof.

\end{proof}

We also have a result estimating the approximation of $h_{t*}V$ by $\phi_{t*}^XV$, and of $h_t^*\omega$ by $\phi_t^{X*}\omega$.

\begin{lema}\label{le:htphit}
Let $\omega$ be a $k$-form on $M$, and $V$ a $k$-multivector field on $M$, and $r\geq 2$.
\begin{enumerate}
\item
If $\omega$ is $C^{\alpha}$, $\alpha\in[0,1]$ then
\begin{equation}\label{eq:htphitomega}
h_t^*\omega=\phi_t^{X*}\omega+O(t^{\min\{2\alpha,r-1\}}).
\end{equation}
\item
If $V$ is $C^{\beta}$, $\beta\in[0,1]$, then
\begin{equation}\label{eq:htphitV}
h_{t*}V=\phi_{t*}^XV+O(t^{\min\{2\beta,r-1\}}).
\end{equation}
\end{enumerate}
\end{lema}

\begin{proof}

{\bf Part (i).} The formula can be verified locally in charts, and applying an argument similar to the one from Lemma~\ref{le:h}, it is sufficient to verify the formula for a $C^{\alpha}$ 0-form $g$, and a $C^{\infty}$ 1-form $dg$.

So let $g:M\rightarrow\mathbb R$ be $C^{\alpha}$. Then applying \eqref{eq:ht1} we get
$$
h_t^*g(p)-\phi_t^{X*}g(p)=g(h_t(p))-g(\phi_t^X(p))\leq Cd(h_t(p),\phi_t^X(p))^{\alpha}=O(t^{2\alpha}).
$$

Now let $g:M\rightarrow\mathbb R$ be $C^{\infty}$. Applying \eqref{eq:Dht12} and \eqref{eq:ht1} we get
\begin{align*}
(h_t^*-\phi_t^{X*})dg(p)&=d(g\circ h_t)(p)-d(g\circ\phi_t^X)(p)\\
&=dg(h_t(p))Dh_t(p)-dg(\phi_t^X(p)D\phi_t^X(p)\\
&=dg(h_t(p))[Dh_t(p)-D\phi_t^X(p)]+[dg(h_t(p))-dg(\phi_t^X(p)]D\phi_t^X(p)\\
&=O(t^{r-1})+O(t^2)=O(t^{r-1}).
\end{align*}
This finishes the proof of the first part.

{\bf Part (ii).} Since locally every multivector field is a combination of exterior products of vector fields, it is sufficient to verify the formula just for vector fields. So let $V$ be a $C^{\beta}$ vector field on $M$. Then
\begin{eqnarray*}
\|(h_{t*}-\phi_{t*}^X)V(p)\|&=&\|Dh_t(h_t^{-1}(p))V(h_t^{-1}(p))-D\phi^X_t(\phi_{-t}^X(p))V(\phi^X_{-t}(p))\|\\
&\leq&\|Dh_t(h_t^{-1}(p))[V(h_t^{-1}(p))-V(\phi_{-t}^X(p))]\|\\
& &+\|[Dh_t(h_t^{-1}(p))-Dh_t(\phi_{-t}^X(p))]V(\phi_{-t}^X(p))\|\\
& &+\|[Dh_t(\phi_{-t}^X(p))-D\phi_t^X(\phi_{-t}^X(p))]V(\phi_{-t}^X(p))\|\\
&\leq& Cd(h_t^{-1}(p)-\phi_{-t}^X(p))^{\beta}+Cd(h_t^{-1}(p),\phi_{-t}^X(p)))+C\|Dh_t-D\phi_t\|\\
&\leq& Ct^{2\beta}+Ct^2+Ct^{r-1}=O(t^{\min\{2\beta,r-1\}}),
\end{eqnarray*}
where we used again \eqref{eq:Dht12} and \eqref{eq:ht1} (which implies that also $d(h_t^{-1}(p)-\phi_{-t}^X(p))=O(t^2)$). This finishes the proof.

\end{proof}


\subsection{Regularity of $t\mapsto V_t$ and a formula for $V'=\left.\frac{\partial}{\partial t}V_t\right|_{t=0}$.}\label{ssec:regularidadVt}

In this section we will investigate the regularity of the map $t\mapsto V_t$ which is equivalent to finding the regularity of the map $t\mapsto E_t$.

So let us assume that $f_t$ has a dominated splitting $TM=E^1_t\oplus E^2_t\oplus E^3_t$, and denote $E_t:=E^2_t$ and $F_t:=E^1_t\oplus E^3_t$. Let $\lambda^1_{E^1}<\lambda^2_{E^1}<\lambda^1_{E^2}<\lambda^2_{E^2}<\lambda^1_{E^3}<\lambda^2_{E^3}$ be expansion bounds along the three sub-bundles for $f=f_0$:
$$
\lambda^1_{E^1}<m(Df|_{E^1})\leq\| Df|_{E^1}\|<\lambda^2_{E^1},
$$
$$
\lambda^1_{E^2}<m(Df|_{E^2})\leq\| Df|_{E^2}\|<\lambda^2_{E^2},
$$
$$
\lambda^1_{E^3}<m(Df|_{E^3})\leq\| Df|_{E^3}\|<\lambda^2_{E^3},
$$
Then the same relations will hold for $f_t$ and the corresponding decomposition $TM=E^1_t\oplus E^2_t\oplus E^3_t$ for $t\in I$, where $I$ is a small interval around zero. Let $g:M\times I\rightarrow M\times I$ be the $C^r$ diffeomorphism defined by
$$
g(x,t)=(f_t(x),t).
$$
The standard Invariant Section Theorem (\cite{HPS1977}, see also \cite{PSW2004}) tells us that $(p,t)\mapsto F_t(p)$ is of class $C^{\alpha}$ in both $t\in I$ and $p\in M$, for
\begin{equation}\label{eq:alpha}
\alpha=\min\left\{\frac{\log\lambda^1_{E^2}-\log\lambda^2_{E^1}}{\log\lambda^2_{E^3}},\frac{\log\lambda^1_{E^3}-\log\lambda^2_{E^2}}{-\log\lambda^1_{E^1}}\right\}.
\end{equation}
In a similar way, one obtains that $(p,t)\mapsto E_t(p)$ is of class $C^{\beta}$ in both $t\in I$ and $p\in M$, for
\begin{equation}\label{eq:beta}
\beta=\min\left\{\frac{\log\lambda^1_{E^3}-\log\lambda^2_{E^2}}{\log\lambda^2_{E^3}}, \frac{\log\lambda^1_{E^2}-\log\lambda^2_{E^1}}{-\log\lambda^1_{E^1}}\right\},
\end{equation}

If we assume some further regularity of the bundles for $t=0$ (with respect to the point on the manifold $M$), then we can also obtain better regularity with respect to the parameter $t$ at $t=0$. More specifically we have the following result.

\begin{prop}\label{p:Vt}

Assume that $f_t$ has a dominated splitting $TM=E^1_t\oplus E^2_t\oplus E^3_t$ (here $E^1$ or $E^3$ can be trivial). If $E=E^2_0$ is of class $C^{\beta}$ for some $\beta\in[0,1]$, and $r\geq 2$, then the map $t\mapsto V_t$ has expansion of order $\beta$ at $t=0$.

\end{prop}

For $\beta=1$ the result was obtain by Dolgopyat in \cite{D2004} (see also \cite{PSW2004} for an alternative proof), and we use their method for $\beta\in(0,Lip]$. The difference is that we will use the action induced on multivector fields instead of the action induced on the Grassmannian. Let us comment that it appears that the result could be improved up to $\beta=1+\beta_0$ where $\beta_0$ is given by the formula \eqref{eq:beta}, and it seems improbable to obtain a similar result for larger values of $\beta$ without further restrictions on the family $f_t$.

\begin{proof}

Suppose that $\beta\in(0, Lip]$ (for $\beta=0$ there is nothing to prove). Recall that since the bundle $E_t$ is invariant under $f_{t*}$, there exists $\eta_t:M\rightarrow(0,\infty)$ such that $(f_t)_*V_t(p)=\eta_t(p)\cdot V_t(f_t(p))$. This means that if we denote $\tilde\eta:=\eta\circ f_t^{-1}$, then we have
\begin{equation}\label{eq:ftvt}
f_{t*}V_t=h_{t*}f_*V_t=\tilde\eta_tV_t.
\end{equation}
In fact $\tilde\eta_t=\omega(h_{t*}f_*V_t)$. Furthermore
\begin{align*}
0&=h_{t*}f_*V_t-\tilde\eta_tV_t\\
&=(h_{t*}f_*-\tilde\eta_tId)(V_t-V)+h_{t*}f_*V-\tilde\eta_tV\\
&=(f_*-\tilde\eta Id)(V_t-V)+[h_{t*}f_*-f_*-(\tilde\eta_t-\tilde\eta)Id](V_t-V)+[h_{t*}f_*V-f_*V]\\
&\quad -(\tilde\eta_t-\tilde\eta)V.
\end{align*}

Observe that $\lim_{t\rightarrow 0}h_{t*}f_*-f_*+(\tilde\eta_t-\tilde\eta)Id=0$ so
$$
[h_{t*}f_*-f_*+(\tilde\eta_t-\tilde\eta)Id](V_t-V)=o(\|V_t-V\|).
$$
Also from Lemma~\ref{lem:taylorhV} we know that $t\mapsto h_{t*}f_*V$ is $C^{\beta}$ in $t=0$ so $h_{t*}f_*V-f_*V=O(t^{\beta})$. Then
\begin{equation}\label{eq:exp}
(f_*-\tilde\eta Id)(V_t-V)-(\tilde\eta_t-\tilde\eta)V=o(\|V_t-V\|)+O(t^{\beta}).
\end{equation}

Let us remark that $V_t-V$ is in the kernel of $\omega$, while $V$ is in the complimentary space $E^{\wedge k}$. Let $\mathcal P:TM\sp{\wedge k}\to \ker(\omega)=F\wedge TM\sp{\wedge(k-1)}$ be the canonical projection parallel to $E^{\wedge k}$, which is given by the formula
\begin{equation}\label{eq:P}
\mathcal P(W)=W-\omega(W)V,\quad (\forall \, W\in TM\sp{\wedge k}).
\end{equation}
Applying the projection $\mathcal P$ to the formula \eqref{eq:exp} we get
\begin{equation}\label{eq:VtV}
\mathcal P[(f_*-\tilde\eta Id)(V_t-V)-(\tilde\eta_t-\tilde\eta)V]=(f_*-\tilde\eta Id)(V_t-V)=o(\|V_t-V\|)+O(t^{\beta})
\end{equation}
{\bf Claim:} $\|(f_*-\tilde\eta Id)(V_t-V)\|\geq C\|V_t-V\|$ for some $C>0$ and small $t$.

If the claim is true, then combined with \eqref{eq:VtV} it gives immediately that $V_t-V=O(t^{\beta})$ as needed.
\begin{proof}[Proof of the claim]
In general there is no need that the operator $f_*-\tilde\eta Id$ is invertible, not even if we restrict it to the kernel of $\omega$. However we will see that if we restrict it to $F\wedge E^{\wedge (k-1)}$ then it is indeed invertible, and this is good enough in order to obtain the claim.

Let $\mathcal T=\frac {f_*}{\tilde\eta}|_{F\wedge E^{\wedge (k-1)}}$. Then $\mathcal T$ can be decomposed into the direct sum $\mathcal T=\mathcal T_1\oplus\mathcal T_3$, where $\mathcal T_i=\mathcal T|_{E^i\wedge E^{\wedge (k-1)}}$, this is because the dominated splitting is invariant under $f_*$. Because of the domination property, one can see that $\mathcal T_1$ is a contraction, while $\mathcal T_3$ is an expansion, in other words the operator $\mathcal T$ is hyperbolic, so $\mathcal T-Id$ is invertible. This in turn implies that $\left.(f_*-\tilde\eta Id)\right|_{F\wedge E^{\wedge (k-1)}}=\tilde\eta(\mathcal T-Id)$ is also invertible, so there exists $C>0$ such that
\begin{equation}\label{eq:boundW}
\|(f_*-\tilde\eta Id)(W)\|> C,\quad\forall\, W\in F\wedge E^{\wedge (k-1)},\quad \|W\|=1.
\end{equation}
From the continuity of the operator $f_*-\tilde\eta Id$, there exists a neighborhood $\mathcal U$ of the set $\{W\in F\wedge E^{\wedge (k-1)}\::\quad \|W\|=1\}$ inside $TM^{\wedge k}$ such that the relation \eqref{eq:boundW} holds for every $W\in\mathcal U$. Then what is left to prove is that $\frac{V_t-V}{\|V_t-V\|}$ is inside $\mathcal U$ for small values of $t$ and $V_t-V\neq0$ (if $V_t-V=0$ then there is nothing to prove).

The fact that $\frac{V_t-V}{\|V_t-V\|}$ is close to $F\wedge E^{\wedge (k-1)}$ follows from the fact that $V\in E^{\wedge k}$, and $V$ and $V_t$ are in fact simple multivectors. Using a partition of unity one can see that it is enough to show this fact locally, and in this case we have
$$
V=V_1\wedge V_2\wedge\dots\wedge V_k,\quad V_i\in E,\quad i\in\{1,2,\dots,k\},
$$
$$
V_t=(V_1+W_{t1})\wedge(V_2+W_{t2})\wedge\dots\wedge (V_k+W_{tk}),\quad W_{ti}\in F,\quad i\in\{1,2,\dots,k\}.
$$
We can suppose that $\|V_{i_1}\wedge\dots\wedge V_{i_l}\|$ and $\|V_{i_1}\|\cdot\dots\cdot\|V_{i_l}\|$ are comparable, for any $\{i_1,\dots i_l\}$ subset of $\{1,2,\dots k\}$ (the vector fields $V_i$ can be chosen to form locally an orthogonal base of $E$ and have all constant size for example). We can decompose $V_t-V=W_t+\tilde W_t$, where
$$
W_t=W_{t1}\wedge V_2\wedge\dots\wedge V_k+\dots+V_1\wedge V_2\wedge\dots\wedge V_{k-1}\wedge W_{tk}\in F\wedge E^{\wedge (k-1)}
$$
and
$$
\tilde W_t=W_{t1}\wedge W_{t2}\wedge V_3\wedge\dots\wedge V_k+\dots+W_{t1}\wedge W_{t2}\wedge\dots\wedge W_{tk}\in F^{\wedge 2}\wedge TM^{\wedge (k-2)}.
$$
Since $F$ and $E$ have the angle uniformly bounded away from zero, we have that $\|W_t\|$ is comparable with $\max\{\|W_{t1}\|,\dots\|W_{tk}\|\}$ uniformly in $t$, and then each term from the formula of $\tilde W_t$ is bounded from above by $D\|W_t\|^2$ for some $D>0$. Then for some $\tilde D>0$ we have that
$$
\|\tilde W_t\|\leq \tilde D\|W_t\|^2.
$$
Estimating the distance between $\frac{W_t}{\|W_t\|}$ and $\frac{V_t-V}{\|V_t-V\|}=\frac{W_t+\tilde W_t}{\|W_t+\tilde W_t\|}$ we get
\begin{align*}
\left\|\frac{W_t}{\|W_t\|}-\frac{W_t+\tilde W_t}{\|W_t+\tilde W_t\|}\right\|&\leq\|W_t\|\cdot\left|\left(\frac 1{\|W_t\|}-\frac 1{\|W_t+\tilde W_t\|}\right)\right| +\frac{\|\tilde W_t\|}{\|W_t+\tilde W_t\|}\\
&\leq \frac{2\|\tilde W_t\|}{\|W_t+\tilde W_t\|}\leq \frac{2\|\tilde W_t\|}{\|W_t\|-\|\tilde W_t\|}\\
&\leq\frac{2\tilde D\|W_t\|}{1-\tilde D\|W_t\|}\leq 4\tilde D\|W_t\|
\end{align*}
which converges uniformly to zero as $t$ goes to zero, so indeed $\frac{V_t-V}{\|V_t-V\|}$ is inside $\mathcal U$ for small enough $t$ and this finishes the proof of the claim.
\end{proof}

We obtained that $V_t-V=O(t^{\beta})$ for $\beta \in (0,Lip]$. If $\beta=1$ then from Lemma~\ref{lem:taylorhV} we know that $t\mapsto h_{t*}f_*V$ is $C^1$ in $t=0$ and
\begin{align*}
h_{t*}f_*V-f_*V&=-t\LL_X(f_*V)+o(t)=-t\LL_X(\tilde\eta V)+o(t)\\
&=-t\tilde\eta\LL_XV+-t\LL_X\tilde\eta V+o(t).
\end{align*}
We also have that $V_t-V=O(t)$, so the relation \eqref{eq:exp} becomes
$$
(f_*-\tilde\eta Id)(V_t-V)-(\tilde\eta_t-\tilde\eta)V=-t\tilde\eta\LL_XV+-t\LL_X\tilde\eta V+o(t).
$$
projecting by $\mathcal P$ on the kernel of $\omega$ we get
$$
(f_*-\tilde\eta Id)(V_t-V)=-t\tilde\eta\mathcal P\LL_XV+o(t).
$$
Dividing by $t$ and taking the limit when $t$ goes to zero we get
$$
(f_*-\tilde\eta Id)\left(\lim_{t\rightarrow 0}\frac{V_t-V}t\right)=-\tilde\eta\mathcal P\LL_X(V),
$$
so $V':=\lim_{t\rightarrow 0}\frac{V_t-V}t$ exists and
\begin{equation}\label{eq:V'}
V'=\left(Id-\frac{f_*}{\tilde\eta}\right)^{-1}\mathcal P\LL_X(V).
\end{equation}
Let us remark that it is easy to see in charts that in fact
\begin{align*}
\LL_X(V)&=\LL_X(V_1\wedge V_2\wedge\dots\wedge V_k)\\
&=(\LL_XV_1)\wedge V_2\wedge\dots\wedge V_k+\dots+V_1\wedge\dots\wedge(\LL_XV_k)\in TM\wedge E^{\wedge (k-1)},
\end{align*}
so $\mathcal P\LL_X(V)\in F\wedge E^{\wedge (k-1)}$ and the inverse of $Id-\mathcal T$ is well defined since $\mathcal T=\frac {f_*}{\tilde\eta}$ is hyperbolic.

\end{proof}

We also obtain a formula for $V'$ if $\beta\geq 1$. Since $F=E^1\oplus E^3$ and $\mathcal P\LL_XV\in F\wedge E^{\wedge (k-1)}$, we can decompose it as
$$
\mathcal P\LL_XV=\mathcal P_1(\LL_XV)+\mathcal P_3(\LL_XV),
$$
where $\mathcal P_i(\LL_XV)\in E^i\wedge E^{\wedge (k-1)}$, $i=1,3$.

\begin{prop}\label{le:formulaV'}
Assume that $f_t$ has a dominated splitting $TM=E^1_t\oplus E^2_t\oplus E^3_t$. If $E=E^2_0$ is of class $C^1$ and $r\geq 2$, then the derivative of the map $t\mapsto V_t$ in $t=0$ is
\begin{equation}\label{eq:cve1}
\begin{aligned}
V'&=\left[\left.\left[Id-\left(\frac{f_*}{\tilde\eta}\right)\right]\right|_{F\wedge E^{\wedge(k-1)}}\right]^{-1}\mathcal P(\LL_XV)\\
&=\sum_{n\geq 0}\left(\frac{f_*}{\tilde\eta}\right)^n\mathcal P_1(\LL_XV)-\sum_{n\geq 1}\left(\frac{f_*}{\tilde\eta}\right)^{-n}\mathcal P_3(\LL_XV).
\end{aligned}
\end{equation}
\end{prop}

\begin{proof}
Recall that the formula \eqref{eq:V'} gives us that $V'=(Id-\mathcal T)^{-1}\mathcal P\LL_XV$. Also the operator $\mathcal T$ is hyperbolic, and it can be decomposed into the direct sum $\mathcal T=\mathcal T_1\oplus\mathcal T_3$, where $\mathcal T_i=\mathcal T|_{E^i\wedge E^{\wedge (k-1)}}$, $i=1,3$. Then we also have $(Id-\mathcal T)^{-1}=(Id-\mathcal T_1)^{-1}\oplus(Id-\mathcal T_3)^{-1}$. Since $\mathcal T_1$ is a contraction, we have that
\begin{equation}\label{eq:L1}
(Id-\mathcal T_1)^{-1}=\sum_{n\geq 0}\mathcal T_1^n,
\end{equation}
and since $\mathcal T_3$ is an expansion, we have
\begin{equation}\label{eq:L3}
(Id-\mathcal T_3)^{-1}=-\sum_{n\geq 1}\mathcal T_3^{-n}.
\end{equation}
Putting the formulas \eqref{eq:L1} and \eqref{eq:L3} together we obtain that indeed
$$
V'=(Id-\mathcal T)^{-1}\mathcal P\LL_XV=(Id-\mathcal T_1)^{-1}\mathcal P_1(\LL_XV)+(Id-\mathcal T_3)^{-1}\mathcal P_3(\LL_XV)
$$
satisfies the desired formula.

\end{proof}


\section{A result on regularity of averaged observables for flows.}\label{sec:regularidad}

In this section we will prove Theorem~\ref{teo:flow}. The proof is based on the following lemma:

\begin{lema}\label{le:Fourier}
Let $\mathbb T$ be the circle $[0,2\pi]|_{0=2\pi}$, and let $f,g:\mathbb T\rightarrow\mathbb R$ be continuous functions. Suppose that $f$ is $C^{\alpha}$ and $g$ is $C^{\beta}$, $\alpha,\beta\geq 0$, and let $h$ be the convolution $f\star g$, i.e. $h(t)=\int_{\mathbb T}f(x)g(t-x)dx$. If either $\alpha+\beta$ is not an integer, or at least one of $\alpha$ or $\beta$ is an integer, then $h$ is $C^{\alpha+\beta}$ and $\| h\|_{C^{\alpha+\beta}}\leq C_{\alpha,\beta}\| f\|_{C^{\alpha}}\| g\|_{C^{\beta}}$. If $\alpha,\beta\notin\mathbb N$ and $\alpha+\beta\in\mathbb N$ then $h$ is $C^{\alpha+\beta-1+Zygmund}$ with the modulus of continuity $C\|f\|_{C^{\alpha}}\|g\|_{C^{\beta}}|t\log t|$.
\end{lema}

This result seems to be known in the more general context of Besov spaces, but since we didn't find a reference, we need the exact bounds, and the proof is fairly simple, we will include it here.

\begin{proof}
{\bf Step 1: Reduction to the case $\alpha,\beta\in (0,1)$.} Let us remark first that the problem can be easily reduced to the case when $\alpha,\beta\in[0,1)$. Indeed, let $\alpha=a+\alpha',\ \beta=b+\beta',\ a,b\in\mathbb N,\ \alpha',\beta'\in[0,1)$. Then differentiating inside the integral and eventually changing the variable, we get $h^{(a+b)}(t)=\int_{\mathbb T}f^{(a)}(x)g^{(b)}(t-x)dx$, with $f^{(a)}$ of class $C^{\alpha'}$ and $g^{(b)}$ of class $C^{\beta'}$. Furthermore $\| h\|_{C^{a+b}}\leq 2\pi\| f\|_{C^{a}}\| g\|_{C^{b}}$.

So we will assume that $\alpha,\beta\in[0,1)$. If either $\alpha$ or $\beta$ are 0 then the result is trivial, so we consider only the case when $\alpha,\beta\in(0,1)$.

{\bf Step 2: Estimate on Fourier coefficients of $h$.} Let $\hat f(n),\ \hat g(n)$ be the Fourier coefficients of $f$ and $g$. Then we know that $\hat h(n)=\widehat{(f\star g)}(n)=2\pi\hat f(n)\hat g(n)$, and the Fourier series of $f,g,h$ are uniformly convergent since the functions are H\"older.

The Fourier coefficients of $f_t(x)=f(x-t)$ for some fixed $t$ are $e^{-int}\hat f(n)$, so the Fourier coefficients of $f_t-f$ are $(e^{-int}-1)\hat f(n)$. If $C_f>0$ is minimal such that $|f(x)-f(y)|\leq C_f|x-y|^{\alpha}$, then a bound for the $L^2$ norm of $f_t-f$ is $\sqrt{2\pi}C_ft^{\alpha}$. Choosing $t_k=2^{-k}\cdot 2\pi/3$, we observe that for $2^k\leq|n|\leq 2^{k+1}$ then $|nt_k|\in [2\pi/3,4\pi/3]$, so $|e^{-int_k}-1|>\sqrt 3$. Applying Parseval identity for $f_{t_k}-f$, we get
\begin{eqnarray*}
\sum_{|n|=2^k}^{2^{k+1}}|\hat f(n)|^2&\leq&\frac 13\sum_{|n|=2^k}^{2^{k+1}}|\hat f(n)|^2|e^{-int_k}-1|^2
\leq\frac 13\sum_{n\in\mathbb Z}|\widehat{(f_{t_k}-f)}|^2\\
&\leq& \frac{2\pi}3 C_f^2t_k^{2\alpha}\leq CC_f^22^{-2k\alpha},
\end{eqnarray*}
where $C$ denotes some universal constant. A similar computation will give that
$$
\sum_{|n|=2^k}^{2^{k+1}}|\hat g(n)|^2\leq CC_g^22^{-2k\beta},
$$
and then by Cauchy-Schwartz
$$
\sum_{|n|=2^k}^{2^{k+1}}|\hat h(n)|=2\pi\sum_{|n|=2^k}^{2^{k+1}}|\hat f(n)\hat g(n)|\leq CC_fC_g2^{-k(\alpha+\beta)}.
$$

{\bf Step 3: The case $\alpha+\beta<1$.} Consider first the case $\alpha+\beta<1$. Let $t,s\in\mathbb T^1$ and $k_0\in\mathbb N$ such that $2^{-(k_0+1)}\leq |t-s|<2^{-k_0}$. Let $c>0$ be such that if $|x|<1$ then $|1-e^x|\leq c|x|$. We obtain
\begin{align*}
|h(t)-h(s)|&=2\pi\sum_{n\in\mathbb Z}\hat f(n)\hat g(n)e^{ins}(1-e^{in(t-s)})\\
&\leq\sum_{k<k_0}\sum_{|n|=2^k}^{2^{k+1}-1}|\hat h(n)|\cdot cn|t-s|+\sum_{k\geq k_0}\sum_{|n|=2^k}^{2^{k+1}-1}2|\hat h(n)|\\
&\leq\sum_{k<k_0}CC_fC_g2^{-k(\alpha+\beta)}2^{k+1}2^{-k_0}+\sum_{k\geq k_0}CC_fC_g2^{-k(\alpha+\beta)}\\
&=CC_fC_g2^{-k_0+1}\sum_{k<k_0}2^{k(1-\alpha-\beta)}+CC_fC_g\sum_{k\geq k_0}2^{-k(\alpha+\beta)}\\
&=CC_fC_g2^{-k_0+1}\frac{2^{k_0(1-\alpha-\beta)}-1}{1-2^{1-\alpha-\beta}}+CC_fC_g2^{-k_0(\alpha +\beta)}\frac 1{1-2^{-(\alpha+\beta)}}\\
&\leq C_{\alpha,\beta}C_fC_g2^{-(k_0+1)(\alpha+\beta)}\leq  C_{\alpha,\beta}C_fC_g|t-s|^{\alpha+\beta}.
\end{align*}
This shows that $h$ is $C^{\alpha +\beta}$ and $\|h\|_{C^{\alpha+\beta}}\leq C_{\alpha,\beta}\|f\|_{C^{\alpha}}\|g\|_{C^{\beta}}$ and completes the case $\alpha+\beta<1$.

\begin{rem}\label{rem:HZ}
One can see from the computation above that for $\alpha+\beta\in(1/2,1)$, one can take $C_{\alpha,\beta}=\frac C{1-\alpha-\beta}$ for some universal constant $C$.
\end{rem}

{\bf Step 4: The case $\alpha+\beta=1$.} If $\alpha+\beta=1$ then we get
\begin{eqnarray*}
|h(t)-h(s)|&=&CC_fC_g2^{-k_0+1}\sum_{k<k_0}2^0+CC_fC_g\sum_{k\geq k_0}2^{-k}\\
&=&CC_fC_gk_02^{-k_0}+CC_fC_g2^{-k_0}\\
&\leq& CC_fC_gk_02^{-(k_0+1)}\leq  CC_fC_g|(t-s)\log(t-s)|.
\end{eqnarray*}
Then the modulus of continuity of $h$ is $C\|f\|_{C^{\alpha}}\|g\|_{C^{\beta}}|t\log t|$ so $h$ is indeed Zygmund.

{\bf Step 5: The case $\alpha+\beta>1$.} Now consider $\alpha+\beta>1$. We remark first that
$$
\sum_{|n|=2^k}^{2^{k+1}}|n\hat h(n)|\leq 2^{k+1}\sum_{|n|=2^k}^{2^{k+1}}|\hat h(n)|=CC_fC_g2^{-k(\alpha+\beta-1)}.
$$
This implies that
$$
\sum_{n\in\mathbb Z}n|\hat h(n)|=\sum_{k\in\mathbb N}\sum_{|n|=2^k}^{2^{k+1}}|n\hat h(n)|\leq\sum_{k\in\mathbb N}CC_fC_g2^{-k(\alpha+\beta-1)}=\frac C{\alpha+\beta-1}C_fC_g
$$
is absolutely convergent. This implies that $h$ is $C^1$ and gives the bound on the derivative of $h$, while the Fourier coefficients of $h'$ will be $\widehat{h'}(n)=in\hat h(n)$.

The proof that $h'$ is $C^{\alpha+\beta-1}$ is similar to the proof of the H\"older continuity of $h$ in the case $\alpha+\beta<1$, one just uses the relation
$$
\sum_{|n|=2^k}^{2^{k+1}}|\hat {h'}(n)|=\sum_{|n|=2^k}^{2^{k+1}}|n\hat h(n)|\leq CC_fC_g2^{-k(\alpha+\beta-1)}.
$$
\end{proof}

\begin{rem}
We remark that the case $\alpha+\beta\in\mathbb N$ is indeed special. One can see this by taking $f(x)=g(x)=\sum_{n=0}^{\infty}\frac 1{2^n}\sin 2^{2n}x$. Then $f$ and $g$ are $C^{\frac 12}$, while $h(t)=\pi\sum_{n=0}^{\infty}\frac1{2^{2n}}\cos2^{2n}t$ is Zygmund but it is not Lipschitz because it has the derivative infinite in 0.
\end{rem}


Now we will prove Theorem~\ref{teo:flow}.

\begin{proof}
We start with the remark that again we can reduce to the case when $\alpha,\beta\in[0,1)$. If $\alpha=a+\alpha',\ \beta=b+\beta',\ a,b\in\mathbb N,\ \alpha',\beta'\in[0,1)$, denote $f_a(x)=\frac{\partial^a}{\partial s^a}f(\phi_s(x))|_{s=0}$ and $g_b(x)=\frac{\partial^b}{\partial s^b}g(\phi_s(x))|_{s=0}$. Differentiating inside the integral and eventually changing the variable we get again that
$$
h^{(a+b)}(t)=\int_Mf_a(x)g_b(\phi_t(x))d\mu,
$$
so it is enough to show the result for $f_a$ being $C^{\alpha'}$ and $g_b$ being $C^{\beta'}$.

If $\alpha=0$ then clearly
$$
|h(t)-h(t')|\leq\int_M|f(x)|\cdot|g(\phi_t(x))-g(\phi_{t'}(x))|d\mu\leq\| f\|_{C^0}C_g\| X\|_{C^0}|t-t'|^{\beta},
$$
so $h$ is clearly $C^{\beta}$ with the required bound. The case $\beta=0$ is similar.

Consequently we can assume from now on that $\alpha,\beta\in(0,1)$.

{\bf The case of ergodic $\mu$.} We assume first that $\mu$ is ergodic. If $x\in M$ is a generic point then from Birkhoff Ergodic Theorem we have
$$
h(t)=\int_Mf(x)g(\phi_t(x))d\mu=\lim_{T\rightarrow\infty}\frac 1T\int_0^Tf(\phi_s(x))g(\phi_{s+t}(x))ds.
$$

By Poincar\'e recurrence, there exist a sequence $t_n\rightarrow\infty$, $n\in\mathbb N$, such that $\lim_{n\rightarrow\infty}\phi_{t_n}(x)=x$. Let $T_N=\frac{t_n+1}{2\pi}$. We can construct smooth closed curves $s\in[0,2\pi T_n]\mapsto \phi^n_s(x)\in M$ obtained by keeping $\phi^n_s(x)=\phi_s(x)$ for $s\in[0,t_n]$, and completing with the curve $\phi^n_s(x)$, $s\in[t_n,t_n+1]$. For $n$ sufficiently large we can assume that the curve $\phi^n_s(x)$ is in the neighborhood of the support of $\mu$ where the regularity of $f$ and $g$ is satisfied, and $\left\|\frac{\partial}{\partial s}\phi^n_s(x)\right\|\leq \|X\|_{C^0}$. Then let
$$
h_n(t)=\frac 1{2\pi T_n}\int_0^{2\pi T_n}f(\phi^n_s(x))g(\phi^n_{s+t}(x))ds=\frac 1{2\pi}\int_0^{2\pi}f(\phi^n_{T_nr}(x))g(\phi^n_{T_nr+t}(x))dr.
$$

Clearly $\lim_{n\rightarrow\infty}h_n(t)=h(t)$ for any $t$, so $h_n$ converges pointwise to $h$. Let $f_n(r)=f(\phi^n_{T_nr}(x))$ and $g_n(r)=g(\phi^n_{T_nr}(x))$. Then
\begin{align*}
|f_n(r)-f_n(r')|&=|f(\phi^n_{T_nr}(x))-f(\phi^n_{T_nr'}(x))|\leq \| f\|_{C^{\alpha}}d(\phi^n_{T_nr}(x),\phi^n_{T_nr'}(x))^{\alpha}\\
&\leq \|X\|_{C^0}^{\alpha}\| f\|_{C^{\alpha}}T_n^{\alpha}|r-r'|^{\alpha},
\end{align*}
so $f_n$ is $C^{\alpha}$ and
$$
\| f_n\|_{C^{\alpha}}\leq\|X\|_{C^0}^{\alpha}\| f\|_{C^{\alpha}}T_n^{\alpha}.
$$
Similarly $g_n$ is $C^{\beta}$ with and 
$$
\| g_n\|_{C^{\beta}}\leq\|X\|_{C^0}^{\beta}\| g\|_{C^{\beta}}T_n^{\beta}.
$$
We also have that
$$
h_n(t)=\frac 1{2\pi}\int_0^{2\pi}f_n(r)g_n(r+t/T_n)dr.
$$
Observe that if $\sigma$ is the involution $\sigma(r)=-r$, and $H_n:=\sigma\circ (f_n\star(g_n\circ\sigma))$, then
$$
h_n(t)=\sigma\circ (f_n\star(g_n\circ\sigma))\left(\frac t{T_n}\right)=H_n\left(\frac t{T_n}\right).
$$
Since $\sigma$ is an isometry, Lemma~\ref{le:Fourier} says that $H_n$ must be $C^{\alpha+\beta}$ (if $\alpha+\beta\neq 1$) and
$$
\|H_n\|_{C^{\alpha+\beta}}\leq C_{\alpha,\beta}\|X\|_{C^0}^{\alpha+\beta}T_n^{\alpha+\beta}\| f\|_{C^{\alpha}}\| g\|_{C^{\beta}}.
$$

{\bf Sub-case $\alpha+\beta<1$.} If $\alpha+\beta<1$ then
$$
|h_n(t)-h_n(t')|=\left|H_n\left(\frac t{T_n}\right)-H_n\left(\frac{t'}{T_n}\right)\right|\leq C_{\alpha,\beta}\|X\|_{C^0}^{\alpha+\beta}\| f\|_{C^{\alpha}}\| g\|_{C^{\beta}}|t-t'|^{\alpha+\beta},
$$
so $h_n$ are uniformly $C^{\alpha+\beta}$, which implies that $h=\lim_{n\rightarrow\infty}h_n$ must be also $C^{\alpha+\beta}$ with the same upper bound on the $C^{\alpha+\beta}$ norm. It also implies the the limit $h_n\rightarrow h$ is uniform on compact sets.

{\bf Sub-case $\alpha+\beta=1$.} In this case Lemma~\ref{le:Fourier} says that $H_n$ are Zygmund, with the modulus of continuity $C\|X\|_{C^0}\|f\|_{C^{\alpha}}\|g\|_{C^{\beta}}T_n|t\log t|$. Then we get that $h_n$ is also Zygmund, however the modulus of continuity may not be uniform with respect to $n$ so we cannot pass to the limit.

We will use instead the previous step and the Remark~\ref{rem:HZ}. Since for $\alpha'<\alpha$, the $C^{\alpha'}$ norm is bounded from above by the $C^{\alpha}$ norm (eventually multiplied by a fixed constant), we get that for every $s\in(1/2,1)$ and any $n$, $h_n$ is $C^s$ and
$$
|h_n(t)-h_n(t')|\leq  \frac C{1-s}\|X\|_{C^0}^s\| f\|_{C^{\alpha}}\| g\|_{C^{\beta}}|t-t'|^s.
$$
If $|t-t'|$ is sufficiently small, we can take $s=1+\frac 1{\log\left(\|X\|_{C^0}|t-t'|\right)}$  and we get
\begin{align*}
|h_n(t)-h_n(t')|&\leq  -C\| f\|_{C^{\alpha}}\| g\|_{C^{\beta}}\log\left(\|X\|_{C^0}|t-t'|\right)\left(\|X\|_{C^0}|t-t'|\right)^{1+\frac 1{\log\left(\|X\|_{C^0}|t-t'|\right)}}\\
&=-eC\| f\|_{C^{\alpha}}\| g\|_{C^{\beta}}\|X\|_{C^0}|t-t'|\left(\log|t-t'|+\log\|X\|_{C^0}\right).
\end{align*}
In conclusion, $h_n$ is indeed uniformly Zygmund, so we can pass to the limit and conclude that $h$ is also Zygmund.

{\bf Sub-case $\alpha+\beta>1$.} Again we will have that $H_n$ must be $C^{\alpha+\beta}$ with the norm $C_{\alpha,\beta}\|X\|_{C^0}^{\alpha+\beta}T_n^{\alpha+\beta}\| f\|_{C^{\alpha}}\| g\|_{C^{\beta}}$. In particular $H_n$ is $C^1$ and the derivative $H_n'$ is $C^{\alpha+\beta-1}$ with the constant $C_{\alpha,\beta}\|X\|_{C^0}^{\alpha+\beta}T_n^{\alpha+\beta}\| f\|_{C^{\alpha}}\| g\|_{C^{\beta}}$. Then $h_n$ must be also $C^1$ and
\begin{align*}
|h'_n(t)-h'_n(t')|&=\frac 1{T_n}\left|H'_n\left(\frac t{T_n}\right)-H'_n\left(\frac{t'}{T_n}\right)\right|\\
&\leq C_{\alpha,\beta}\|X\|_{C^0}^{\alpha+\beta}\| f\|_{C^{\alpha}}\| g\|_{C^{\beta}}|t-t'|^{\alpha+\beta-1},
\end{align*}
so $h_n'$ verifies the $(\alpha+\beta-1)$-H\"older condition uniformly with respect to $n$.

We claim that $h_n'$ are also uniformly bounded (this does not follow directly from the bounds above).

Again we know that for any $s\in(0,1)$, the maps $h_n$ are uniformly $C^s$, the $C^s$ norms of $h_n$ are uniformly bounded by $C_s\|X\|_{C^0}^r\| f\|_{C^{\alpha}}\| g\|_{C^{\beta}}$, and $h_n$ converges uniformly on compact sets to $h$.

In particular for $s=\frac 12$ the $C^{\frac 12}$ property of $h_n$ gives
$$
|h_n(t+a)-h_n(t)|\leq\| h_n\|_{C^{\frac 12}}a^{1/2},
$$
while the $C^{\alpha+\beta-1}$ condition on $h_n'$ gives
$$
|h_n'(t)-h_n'(s)|\leq\| h_n'\|_{C^{\alpha+\beta-1}}|t-s|^{\alpha+\beta-1}\leq\| h_n'\|_{C^{\alpha+\beta-1}}a^{\alpha+\beta-1},\ \ \forall s\in[t,t+a],
$$
or
$$
h_n'(t)\leq \| h_n'\|_{C^{\alpha+\beta-1}}a^{\alpha+\beta-1}+h_n'(s).
$$
Then we get
\begin{eqnarray*}
h_n'(t)&\leq &\| h_n'\|_{C^{\alpha+\beta-1}}a^{\alpha+\beta-1}+\frac 1a\int_t^{t+a}h_n'(s)ds\\
&=&\| h_n'\|_{C^{\alpha+\beta-1}}a^{\alpha+\beta-1}+\frac 1a(h_n(t+a)-h_n(t))\\
&\leq&\| h_n'\|_{C^{\alpha+\beta-1}}a^{\alpha+\beta-1}+\| h_n\|_{C^{\frac 12}}a^{-1/2}\\
&\leq&C_{\alpha,\beta}\| f\|_{C^{\alpha}}\| g\|_{C^{\beta}}\left[\|X\|_{C^0}^{\alpha+\beta}a^{\alpha+\beta-1}+\|X\|_{C^0}^{1/2}a^{-1/2}\right].
\end{eqnarray*}
Choosing $a=\|X\|_{C^0}^{-1}$ we get
$$
h_n'(t)\leq C_{\alpha,\beta}\| f\|_{C^{\alpha}}\| g\|_{C^{\beta}}\|X\|_{C^0}.
$$

A similar argument works for $-h_n'(t)$, so we have the uniform bounds for $| h_n'|$. Using also the uniform H\"older conditions on $h_n'$, we can apply Arzela-Ascoli in order to obtain a subsequence $h_{n_k}'$ convergent (uniformly on compact sets) to some $h'$, and this will imply that $h'$ must be equal to the derivative of $h$. The uniform bounds on $h_n'$ transfer to $h'$, so
$$
\|h\|_{C^1}\leq C_{\alpha,\beta}\| f\|_{C^{\alpha}}\| g\|_{C^{\beta}}\|X\|_{C^0}
$$

This concludes the proof for the case when $\mu$ is ergodic.

{\bf The case of general $\mu$.} Now suppose that $\mu$ is not ergodic, then it must have an ergodic decomposition
$$
\mu=\int_{\mathcal M_e}\nu \,dm_{\mu}(\nu),
$$
where $\mathcal M_e$ are the ergodic invariant probabilities of $\phi$ and $m_{\mu}$ is a Borel probability measure on $\mathcal M_e$. Then
$$
h(t)=\int_Mf(x)g(\phi_t(x))d\mu=\int_{\mathcal M_e}\int_Mf(x)g(\phi_t(x))d\nu dm_{\mu}(\nu):=\int_{\mathcal M_e}h_{\nu}(t)dm_{\mu}(\nu).
$$
Now since $h_{\nu}(t)$ are $C^{\alpha+\beta}$ in $t$ with uniform bounds independent of $\nu$, then the same must be true for $h(t)$, and the bounds are preserved. This concludes the proof of the theorem.

\end{proof}


\section{Regularity of the averaged Lyapunov exponents}\label{sec:holder}

Now we will prove Theorem~\ref{teo:lyapunovH}. We will invoke frequently the following lemma of calculus, we will omit its proof.

\begin{lema}\label{le:lemalog}
Let $\alpha:I\times M\to\mathbb{R}$ continuous in $t$ and $r\geq 0$. If $0<c<\alpha(t)$, then
\begin{equation}\label{eq:logbas1}
\log\left(\alpha(t)+o,O(t^r,t^r\log t)\right)=\log\alpha(t)+o,O(t^r,t^r\log t).
\end{equation}
In particular, if $\alpha:I\times M\to\mathbb{R}$ continuous with respect to  $t\in I$ uniformly with respect to $p\in M$  and $\nu$ is a  probability Borel measure on $M$, then
\begin{equation}\label{eq:logbas2}
\int\log\left(\alpha(t,p)+o,O(t^r,t^r\log t)\right)d\nu(p)=\int\log\alpha(t,p)d\nu(p)+o,O(t^r,t^r\log t).
\end{equation}
\end{lema}

We will use the notations introduced previously in the paper. Recall that from \eqref{eq:Lyapunov} we have
$$
\lambda(t)=\int_M\log\eta_td\mu=\int_M\log\tilde\eta_td\mu,
$$
where $f_{t*}V_t=\tilde\eta_tV_t$, $\tilde\eta_t=\eta_t\circ f_t^{-1}$. Since $\omega(V_t)=1$, applying $\omega$ we get
$$
\tilde\eta_t=\omega(f_{t*}V_t)=\omega(h_{t*}f_*V_t)=h_t^*\omega(f_*V_t)\circ h_t^{-1}.
$$

\begin{proof}[Proof of Theorem~\ref{teo:lyapunovH}]

The strategy of the proof is to approximate $\lambda(t)$ up to order $t^{\alpha+\beta}$ by a simple formula involving the action of the flow $\phi_t^X$ on $\omega$ (or $V$), and then to use Theorem~\ref{teo:flow} in order to obtain the regularity of this new expression.

{\bf Step 1: The following approximations hold:
\begin{equation}\label{eq:tetat}
\tilde\eta_t=\omega(h_{t*}f_*V_t)=\tilde\eta\circ h_t^{-1}\cdot \omega(\phi_{t*}^XV)+O(t^{\min\{\alpha+\beta,2\beta\}}),
\end{equation}
\begin{equation}\label{eq:etat}
\eta_t\circ f^{-1}=h_t^*\omega(f_*V_t)=\tilde\eta\cdot \phi_t^{X*}\omega(V)+O(t^{\min\{\alpha+\beta,2\alpha\}}),
\end{equation}
\begin{equation}\label{eq:lambdat}
\lambda(t)=\lambda(0)+\int_M\log\omega(\phi^X_{t*}V)d\mu+O(t^{\alpha+\beta})=\lambda(0)+\int_M\log\phi_t^{X*}\omega(V)d\mu+O(t^{\alpha+\beta}).
\end{equation}}

In fact, we have:
\begin{align*}
\tilde\eta_t&=\omega(h_{t*}f_*V_t)=\omega(h_{t*}f_*V)+\omega(h_{t*}f_*(V_t-V))\\
&=\tilde\eta\circ h_t^{-1}\cdot\omega(h_{t*}V)+\omega(h_{t*}f_*(V_t-V))\\
&=\tilde\eta\circ h_t^{-1}\cdot\omega(\phi^X_{t*}V)+\tilde\eta\circ h_t^{-1}\cdot\omega((h_{t*}-\phi^X_{t*})V)+\omega(h_{t*}f_*(V_t-V))\\
&:=\tilde\eta\circ h_t^{-1}\cdot\omega(\phi^X_{t*}V)+E_1(t,p)+E_2(t,p)
\end{align*}
We will evaluate $E_1$ and $E_2$ separately. Since $V$ is $C^{\beta}$ we can apply Lemma~\ref{le:htphit} and since $r\geq\alpha+\beta+1$ we get
$$
E_1=\tilde\eta\circ h_t^{-1}\cdot\omega((h_{t*}-\phi_{t*}^X)V)=\tilde\eta\circ h_t^{-1}\cdot\omega(O(t^{\min\{2\beta,r-1\}}))=O(t^{\min\{2\beta,\alpha+\beta\}}).
$$

Recall that $V_t-V$ is in the kernel of $\omega$, which is invariant by $f_*$, so we have $\omega(f_*(V_t-V))=0$. We obtain
\begin{eqnarray*}
E_2&=&\omega(h_{t*}f_*(V_t-V))=h_t^*\omega(f_*(V_t-V))\circ h_t^{-1}\\
&=&(h_t^*\omega-\omega)(f_*(V_t-V))\circ h_t^{-1}=O(t^{\alpha+\beta})
\end{eqnarray*}
since from Lemma~\ref{le:h} we know that $h_t^*\omega-\omega=O(t^{\alpha})$, and from Proposition~\ref{p:Vt} we have that $V_t-V=O(t^{\beta})$.

Putting the above estimates together we obtain the formula \eqref{eq:tetat}. The proof of \eqref{eq:etat} is similar:
\begin{align*}
\eta_t\circ f^{-1}&=h_t^*\omega(f_*V_t)=h_t^*\omega(f_*V)+h_t^*\omega(f_*(V_t-V))\\
&=\tilde\eta\cdot h_t^*\omega(V)+h_t^*\omega(f_*(V_t-V))\\
&=\tilde\eta\cdot \phi_t^{X*}\omega(V)+\tilde\eta\cdot (h_t^*-\phi_t^{X*})\omega(V)+h_t^*\omega(f_*(V_t-V))\\
&:=\tilde\eta\cdot \phi_t^{X*}\omega(V)+E_3(t,p)+E_4(t,p)
\end{align*}
Then applying again Lemma~\ref{le:htphit} we obtain
$$
E_3(t,p)=(h_t^*-\phi_t^{X*})\omega(f_*V)=O(t^{\min\{2\alpha,r-1\}})(V)=O(t^{\min\{2\alpha,\alpha+\beta\}}).
$$
Furthermore
$$
E_4(t,p)=h_t^*\omega(f_*(V_t-V))=(h_t^*\omega-\omega)(f_*(V_t-V))=O(t^{\alpha+\beta})
$$
is similar to the estimation of $E_2$, and the proof of \eqref{eq:etat} follows.

In order to obtain the approximation \eqref{eq:lambdat}, we use \eqref{eq:tetat} or \eqref{eq:etat}, depending whether $\beta\geq\alpha$ or not. For example if $\beta\geq\alpha$ we use \eqref{eq:tetat} and we get
\begin{eqnarray*}
\lambda(t)&=&\int_M\log\tilde\eta_t=\int_M\log\left[\tilde\eta\circ h_t^{-1}\cdot \omega(\phi_{t*}^XV)+O(t^{\alpha+\beta})\right]d\mu\\
&=&\int_M\log\omega(\phi^X_{t*}V)d\mu+\int_M\log(\tilde\eta\circ h_t^{-1}) d\mu+O(t^{\alpha+\beta})\\
&=&\int_M\log\omega(\phi^X_{t*}V)d\mu+\lambda(0)+O(t^{\alpha+\beta}).
\end{eqnarray*}
We used the fact that $\phi_{-t}^X$ preserves $\mu$, so $\int_M\log(\tilde\eta\circ\phi_{-t}^X) d\mu=\int_M\log\tilde\eta d\mu=\lambda(0)$.

If $\alpha>\beta$ we use \eqref{eq:etat} and we get
\begin{align*}
\lambda(t)&=\int_M\log\eta_t=\int_M\log\eta_t\circ f^{-1}=\int_M\log\left[\tilde\eta\cdot \phi_t^{X*}\omega(V)+O(t^{\alpha+\beta})\right]d\mu\\
&=\int_M\log\phi_t^{X*}\omega(V)d\mu+\int_M\log\tilde\eta d\mu+O(t^{\alpha+\beta})\\
&=\int_M\log\left[\omega(\phi^X_{t*}V)\circ\phi_t^X\right]d\mu+\lambda(0)+O(t^{\alpha+\beta})\\
&=\int_M\log\omega(\phi^X_{t*}V)d\mu+\lambda(0)+O(t^{\alpha+\beta}).
\end{align*}
We used again the fact that $f$ and $\phi_t^X$ preserve $\mu$.

{\bf Step 2: The map $t\mapsto \int_M\omega(\phi^X_{t*}V)d\mu$ is $C^{\alpha+\beta}$ (or is $C^{\alpha+\beta-1+Zygmund}$).}

This is an application of Theorem~\ref{teo:flow}, and it can be done in general for any $\alpha,\beta>0$, as long as $X$ is $C^r$ with $r\geq \alpha+\beta+1$.

Choose a finite open cover of $M$ with charts and a smooth partition of unity associated to it, $(U_i,\rho_i),\ 1\leq i\leq N$. Since
$$
\int_M\omega(\phi^X_{t*}V)d\mu=\sum_{i=1}^N\int_{M}\rho_i\omega(\phi^X_{t*}V)d\mu,
$$
it is sufficient to study the regularity of $t\mapsto\int_{U_i}\rho_i\omega(\phi^X_{t*}V)d\mu$. So we can assume that $\omega$ is $C^{\alpha}$ and supported in a small chart $U$.

We know that $D\phi^X_t(p)$ is $C^{r-1}$, with $r-1\geq\alpha+\beta$, so
\begin{equation}\label{eq:Dphit}
D\phi^X_t(p)=Id+tD_1(p)+\frac{t^2}2D_2(p)+\dots+\frac{t^a}{a!}D_a(p)+O(t^{\alpha+\beta}),
\end{equation}
where $a=[\alpha+\beta]$, $D_i:U_i\rightarrow \mathcal M_{d\times d}(\mathbb R)$ is $C^{r-i-1}$, for all $1\leq i\leq a$. Denote $D_0=Id$, and observe that $D_1=DX$.

We can assume that in the chart $U$ we have $V=V_1\wedge V_2\wedge\dots\wedge V_k$, where $V_1,V_2,\dots V_k$ are $C^{\beta}$ vector fields. Then
\begin{align*}
\phi_{t*}^XV\circ\phi_t^X&=(\phi_{t*}^XV_1\wedge\phi_{t*}^XV_2\wedge\dots\wedge\phi_{t*}^XV_k)\circ\phi_t^X=\\
&=D\phi_t^X(p)V_1(p)\wedge D\phi_t^X(p)V_2(p)\wedge\dots\wedge D\phi_t^X(p)V_k(p).
\end{align*}
Using the expansion \eqref{eq:Dphit} of $D\phi_t^X(p)$ we get an expansion
$$
\phi_{t*}^XV\circ\phi_t^X=V+t\mathcal D_1(V)+\frac t2\mathcal D_2(V)+\dots+\frac{t^a}{a!}\mathcal D_a(V)+O(t^{\alpha+\beta}),
$$
where each $\mathcal D_i(V)$ is an expression involving $D_1,D_2,\dots D_i$ and $V$, so it is of class $C^{\min\{r-i-1,\beta\}}$. In particular
$$
\mathcal D_1(V)=DX\cdot V_1\wedge V_2\wedge\dots\wedge V_k+V_1\wedge DX\cdot V_2\wedge\dots\wedge V_k+\dots+V_1\wedge V_2\wedge\dots\wedge DX\cdot V_k
$$
and $\|\mathcal D_1(V)\|\leq k\|X\|_{C^1}\|V\|_{C^0}$.

Then, for $t$ small, we have
$$
\int_U\omega(\phi_{t*}V)d\mu=\int_U\omega(\phi_{t*}^XV)\circ\phi_t^Xd\mu=\sum_{i=0}^a\frac{t^i}{i!}\int_U\omega(\phi_t^X(p))\mathcal D_i(V)(p)d\mu(p)+O(t^{\alpha+\beta})
$$

Now remember that $\omega$ is $C^{\alpha}$, while $\mathcal D_i(V)$ is $C^{\min\{r-i-1,\beta\}}$.
Furthermore, in the chart $U$ we have $\omega=\sum_{I\in\mathcal I}a_Idx_I$ and $\mathcal D_i(V)=\sum_{I\in\mathcal I}b_I\frac\partial{\partial x_I}$, with $\mathcal I$ being the set of multi-indices of size $k$, and all $a_I$ are of class $C^{\alpha}$ and $b_I$ of class $C^{\min\{r-i-1,\beta\}}$. Then by Theorem~\ref{teo:flow} we have that
$$
A_i(t):=\int_U\omega(\phi_t^X(p))\mathcal D_i(V)(p)d\mu(p)=\sum_{I\in\mathcal I}\int_Ua_I(\phi_t^X(p))b_I(p)d\mu(p)
$$
is of class $C^{\alpha+\min\{r-i-1,\beta\}}=C^{\min\{\alpha+\beta,r-i-1+\alpha\}}$ as a function of $t$ (or of class  $C^{\min\{\alpha+\beta-1,r-i-2+\alpha\}+Zygmund}$). If $\alpha+\beta\leq r-i-1+\alpha$ then $A_i(t)$ is $C^{\alpha+\beta}$ (or $C^{\alpha+\beta-1+Zygmund}$). Otherwise we have $t\mapsto t^iA_i(t)$ has expansion in $t=0$ of order $i+(r-i-1+\alpha) -\epsilon\geq2\alpha+\beta-\epsilon>\alpha+\beta$ for small enough $\epsilon>0$.

We conclude that for all $0\leq i\leq a$, we have that $t^iA_i$ has expansion of order $(\alpha+\beta)$ (or $\alpha+\beta-1+Zygmund$) in $t=0$, so also $t\mapsto\int_U\omega(\phi^X_{t*}V)d\mu$ has expansion of order $(\alpha+\beta)$ (or $\alpha+\beta-1+Zygmund$) in $t=0$.

Since in the above argument one can replace $V$ by $\phi_{s*}^XV$ for any $s$, we obtain that $t\mapsto\int_U\omega(\phi_{t*}^XV)d\mu$ is $C^{\alpha+\beta}$ (or is $C^{\alpha+\beta-1+Zygmund}$) for all $t$.

If in particular $\alpha,\beta\in(0,1)$, $\alpha+\beta>1$, $r\geq \alpha+\beta+1$, using the formula of $\mathcal D_1(V)$ we get
\begin{align*}
\int_U\omega(\phi_{t*}V)d\mu&=\int_U\omega(\phi_{-t}^X(p))V(p)d\mu(p)+t\int_U\omega(\phi_{-t}^X(p))\mathcal D_1(V)(p)d\mu(p)+O(t^{\alpha+\beta})\\
&=\int_U\omega(\phi_{-t}^X(p))V(p)d\mu(p)+t\int_U\omega(p)\mathcal D_1(V)(p)d\mu(p)+O(t^{\alpha+\beta}).
\end{align*}
Using the estimate on the derivative in Theorem~\ref{teo:flow} and putting the charts together we get
\begin{equation}\label{eq:boundderivative}
\left|\frac{\partial}{\partial t}\int_M\omega(\phi_{t*}^XV)d\mu\right|_{t=0}\leq C_{\alpha,\beta,M}\|X\|_{C^0}\|\omega\|_{C^{\alpha}}\|V\|_{C^{\beta}}+C_M\|X\|_{C^1}\|\omega\|_{C^0}\|V\|_{C^0}.
\end{equation}

{\bf Step 3: The map $t\mapsto \int_M\log\omega(\phi^X_{t*}V)d\mu$ is $C^{\alpha+\beta}$ (or Zygmund if $\alpha+\beta=1$) in $t=0$.}

Even if the map $t\mapsto\int_M\omega(\phi_{t*}^XV)d\mu=\int_M\phi_t^{X*}\omega(V)$ is $C^{\alpha+\beta}$ (or $C^{\alpha+\beta-1+Zygmund}$) for all $t$, the map $t\mapsto \int_M\log\omega(\phi^X_{t*}V)d\mu$ could be $C^{\alpha+\beta}$ (or Zygmund if $\alpha+\beta=1$) only in $t=0$. Let us remind first that $\phi_{t*}^XV$ is uniformly $C^{\beta}$ in $t$ and $\phi_t^{X*}\omega$ is uniformly $C^{\alpha}$ in $t$. Then either $\omega(\phi_{t*}^XV)$ or $\phi_t^{X*}\omega(V)$ will be uniformly $C^{\max\{\alpha,\beta\}}$ in $t$. Assume that $g(t,p):=\omega(\phi_{t*}^XV)$ is uniformly $C^{\max\{\alpha,\beta\}}$ in $t$, the other case can be treated similarly.

\begin{lema}\label{lem:derlog}
Given $g(t,p)$ uniformly $C^s$ in $t$, $s\in(0,1)$, with $g(0,p)=1$ for all $p\in M$, and $\int_Mg(t,p)d\mu$ of class $C^u$ in $t$, $u\leq 2s$, then $\int_M\log g(t,p)d\mu$ has expansion of order $u$ in $t$ at $t=0$.
\end{lema}

\begin{proof}
We know that $\log A=A-1+O\left((A-1)^2\right)$. Then
\begin{eqnarray*}
\int_M\log g(t,p)d\mu&=&\int_Mg(t,p)-1+O\left((g(t,p)-1)^2\right)d\mu\\
&=&-1+\int_Mg(t,p)d\mu+O(t^{2s})\\
&=&-1+\int_Mg(t,p)d\mu+O(t^u)
\end{eqnarray*}
and the result follows. The result also works for $u=Zygmund$, and if $u>1$ then
$$
\left.\frac\partial{\partial t}\int_M\log g(t,p)d\mu\right|_{t=0}=\left.\frac\partial{\partial t}\int_Mg(t,p)d\mu\right|_{t=0}.
$$
\end{proof}

Applying the above lemma for $s=\max\{\alpha,\beta\}$ and $u=\alpha+\beta$ or $u=Zygmund$ is $\alpha+\beta=1$, we get that indeed $t\mapsto\int_M\log\omega(\phi_{t*}V)d\mu$ has an expansion of order $\alpha+\beta$ or Zygmund in $t=0$.

Now putting Step 1 and Step 3 together we obtain the desired regularity for $\lambda(t)$ in $t=0$.

Furthermore, if $\alpha+\beta>1$ then
\begin{equation}\label{eq:derlog}
\lambda'(0)=\left.\frac{\partial}{\partial t}\int_M\omega(\phi_{t*}^XV)d\mu\right|_{t=0}=\left.\frac{\partial}{\partial t}\int_M\phi_t^{X*}\omega(V)d\mu\right|_{t=0}=\left.\frac{\partial}{\partial t}\int_M\log[\omega(\phi_{t*}^XV)]d\mu\right|_{t=0}
\end{equation}
and
\begin{equation}\label{eq:boundderivative2}
|\lambda'(0)|\leq C_{\alpha,\beta,M}\|X\|_{C^0}\|\omega\|_{C^{\alpha}}\|V\|_{C^{\beta}}+C_M\|X\|_{C^1}\|\omega\|_{C^0}\|V\|_{C^0}.
\end{equation}
This finishes the proof of the theorem.
\end{proof}

\begin{rem}\label{rem:regularity}
Inspecting the proof one can see that in fact the condition $r\geq\max\{\alpha+\beta+1,2\}$ is sufficient for the proof above. In particular, if $\alpha+\beta>1$ and $r=2$ then the derivative $\lambda'(0)$ exists and satisfies the bound \ref{eq:boundderivative2}.
\end{rem}

Next we prove Theorem~\ref{cor:gflow}.

\begin{proof}[Proof of Theorem~\ref{cor:gflow}]

Let $f_t$ be a $C^2$ family of $C^2$ diffeomorphisms in $PH^{2,\frac 12+}_{\mu,1}(M)$, and let $TM=E_t^s\oplus E_t^c\oplus E_t^u$ be the corresponding splitting for every $f_t$. Let $X_t$ be the vector field tangent to the family at each $f_t$. We can assume that all the bundles $E_t^i$ are $C^{\alpha}$ for some $\alpha>\frac 12$, and by compactness we have that $\|E_t^i\|_{C^{\alpha}}$ for $i\in\{s,c,u\}$, and $\|X_t\|_{C^1}$ are uniformly bounded. By Theorem~\ref{teo:lyapunovH}, for every $i\in\{s,c,u\}$, the map $t\mapsto\lambda(f_t,E_t^i, \mu)$ is differentiable everywhere, and the derivative is uniformly bounded, so the map is Lipschitz.

On the other hand, if we apply the Pesin formula we obtain that the metric entropy of $f_t$ with respect to $\mu$ is the integral of the sum of the positive Lyapunov exponents of $f_t$ with respect to $\mu$. Since every $f_t$ is not in $\mathcal A$, the center exponent will have the same sign $\mu$-almost everywhere. If the center exponent is positive then 
$$
h_{\mu}(f_t)=\lambda(f_t,E_t^c,\mu)+\lambda(f_t,E_t^u,\mu)=-\lambda(f_t,E_t^s,\mu),
$$
and if the center exponent is negative then
$$
h_{\mu}(f_t)=\lambda(f_t,E_t^u,\mu).
$$
Then we have
$$
h_{\mu}(f_t)=\max\{\lambda(f_t,E_t^u,\mu),-\lambda(f_t,E_t^s,\mu)\}
$$
is the maximum of two Lipschitz functions, so it is also Lipschitz.

\end{proof}



\section{Formulas for the derivatives (Proof of Theorem~\ref{teo:lyapunov12})}\label{sec:demteoD}

In this section we will prove Theorem~\ref{teo:lyapunov12}. The computations are based on the estimates from Lemmas~\ref{le:h} and~\ref{lem:taylorhV}. The first two parts of the theorem could also be obtained from the formulas \eqref{eq:tetat} and \eqref{eq:etat}.

\begin{proof}
{\bf Part (i).} We have that $\omega$ is $C^1$, so from \eqref{eq:derivadahw} we know that $h_t^*\omega=\omega+t\LL_X\omega+o(t)$. We also have that $V_t-V=o(1)$ and is in the kernel of $\omega$. Then, using the facts that $f_*V=\tilde\eta V$ and $\omega(V)=1$, we get
\begin{align*}
\eta_t\circ f^{-1}&=h_t^*\omega(f_*V_t)=(\omega+t\LL_X\omega)[f_*V+f_*(V_t-V)] +o(t)\\
&=\omega(\tilde\eta V)+t\LL_X\omega(\tilde\eta V)+\omega(f_*(V_t-V))+t\LL_X\omega(f_*(V_t-V))]+o(t)\\
&=\tilde\eta(1+t\LL_X\omega(V))+o(t).
\end{align*}
We compose with $f$ and we take the logarithm, and using the fact that $\log(1+a)=a+o(a)$, we get that the following holds if $\omega$ is $C^1$:
\begin{equation}\label{eq:etatomega}
\log\eta_t=\log\eta+\log[1+t\LL_X\omega(V)\circ f]+o(t)=\log\eta+t\LL_X\omega(V)\circ f+o(t).
\end{equation}

Integrating with respect to $\mu$ we obtain
\begin{align*}
\lambda(t)&=\int_M\log\eta_td\mu=\int_M\log\eta+t\LL_X\omega(V)\circ f\,d\mu+o(t)\\
&=\lambda(0)+t\int_M\LL_X\omega(V)d\mu+o(t).
\end{align*}

This shows that $\lambda$ is differentiable in zero and
$$
\lambda'(0)=\int_M\LL_X\omega(V)d\mu.
$$

{\bf Part (ii).} Now assume that $V$ is $C^1$. By Proposition~\ref{p:Vt} we know that $V_t$ is differentiable in $t=0$, so $V_t=V+tV'+o(t)$, with $V'$ in the kernel of $\omega$. We also know from Lemma\ref{lem:taylorhV} that $h_{t*}V=V-t\LL_XV+o(t)$. Using that $f_*$ preserves the kernel of $\omega$, $h_{t*}-Id=o(1)$, we get
\begin{align*}
\tilde\eta_t&=\omega(h_{t*}f_*V_t)=\omega(h_{t*}f_*V)+t\omega(h_{t*}f_*V')]+o(t)\\
&=\omega(h_{t*}\tilde\eta V)+t\omega(h_{t*}f_*V')+o(t)\\
&=\tilde\eta\circ h_t^{-1}[\omega(V)-t\omega(\LL_XV)]+t\omega(f_*V')+t\omega(h_{t*}f_*V'-f_*V')+o(t)\\
&=\tilde\eta\circ h_t^{-1}[1-t\omega(\LL_XV)]+o(t).
\end{align*}
Composing with $f_t=h_t\circ f$ and using that $\omega(\LL_XV)\circ f_t=\omega(\LL_XV)\circ f+o(1)$ we get
$$
\eta_t=\tilde\eta_t\circ f_t=\eta[1-t\omega(\LL_XV)\circ f_t]+o(t)=\eta[1-t\omega(\LL_XV)\circ f]+o(t).
$$
Taking the logarithm we get that if $V$ is $C^1$ then the following holds:
\begin{equation}\label{eq:etatV}
\log\eta_t=\log\eta+\log[1-t\omega(\LL_XV)\circ f]+o(t)=\log\eta-t\omega(\LL_XV)\circ f+o(t).
\end{equation}

Integrating with respect to $\mu$ we obtain
\begin{align*}
\lambda(t)&=\int_M\log\eta_td\mu=\int_M\log\eta-t\omega(\LL_XV)\circ fd\mu+o(t)\\
&=\lambda(0)-t\int_M\omega(\LL_XV)d\mu+o(t).
\end{align*}

This shows again that $\lambda$ is differentiable in zero and
$$
\lambda'(0)=-\int_M\omega(\LL_XV)d\mu.
$$

{\bf Part (iii).} Recall that $V_t=V+tV'+o(t)$ and $h_t^*\omega=\omega+t\LL_X\omega +o(t)$. Let $\tilde\omega_t=h_t^*\omega-\omega-t\LL_X\omega=o(t)$.

We first evaluate $\eta_t\circ f^{-1}$. Using that $V_t-V$ is in the kernel of $\omega$ which is preserved by $f_*$, $V_t-V=tV'+o(t)$, and $\tilde\omega_t=o(t)$, we get
\begin{align*}
\eta_t\circ f^{-1}&=h_t^*\omega(f_*V_t)=h_t^*\omega(f_*V)+h_t^*\omega[f_*(V_t-V)]\\
&=h_t^*\omega(\tilde\eta V)+\omega[f_*(V_t-V)]+t\LL_X\omega[tV'+o(t)]+\tilde\omega_t[tV'+o(t)]\\
&=\tilde\eta h_t^*\omega(V)+t^2\LL_X\omega(f_*V')+o(t^2).
\end{align*}

Now we compose with $f$ and we apply the logarithm. Invoking Lemma~\ref{le:lemalog} and recalling the relation $\log(1+s)=s-\frac{1}{2}s^2 +o(s^2)$ we obtain

\begin{align*}
\log\eta_t&=\log\left[\eta\left(h_t^*\omega(V)+\frac{t^2}{\tilde\eta}\LL_X\omega(f_*V')\right)\circ f+o(t^2)\right]\\
&=\log\eta+\log\left[1+t\LL_X\omega(V)+\tilde{\omega}_t(V)+\frac{t^2}{\tilde\eta}\LL_X\omega(f_*V')\right]\circ f+o(t^2)
\end{align*}
so
\begin{equation}\label{eq:forsec3407}
\log\eta_t=\log\eta+\left[t\LL_X\omega(V)+\tilde{\omega}_t(V)+\frac{t^2}{\tilde\eta}\LL_X\omega(f_*V')-\frac{t^2}2[\LL_X\omega(V)]^2\right]\circ f+o(t^2)
\end{equation}
We used the fact that
$$
[t\LL_X\omega(V)+\tilde{\omega}_t(V)+\frac{t^2}{\tilde\eta}\LL_X\omega(f_*V')]^2=t^2[\LL_X\omega(V)]^2+o(t^2).
$$

Integrating with respect to $\mu$ and using the fact that $f$ preserves $\mu$ we get the following estimation on $\lambda_t$:
\begin{eqnarray}
\lambda(t)&=&\lambda(0)+t\int_M \LL_X\omega(V)d\mu+ \nonumber\\
& &+\frac{t^2}2\int_M \left(\frac{2}{\tilde\eta}\LL_X\omega(f_*V')-[\LL_X\omega(V)]^2\right) d\mu +\int_M  \tilde{\omega}_t(V)d\mu+o(t^2). \label{eq:forsec3404}
\end{eqnarray}

We are left with the estimation of $\int_M \tilde{\omega}_t(V)d\mu$. We will do this using approximations with $C^2$ forms.

Remember that since $X$ preserves the measure $\mu$, then for any $C^1$ function $g:M\rightarrow\mathbb R$ we have $\int_M\LL_Xgd\mu=0$.

\begin{lema}\label{le:3403} The following estimation holds:

\begin{equation}\label{eq:forsec3408}
\int_M \tilde{\omega}_t(V)d\mu=\frac{t^2}2\int_M -\LL_X\omega(\LL_XV)+\LL_Y\omega(V) \,d\mu+o(t^2)
\end{equation}

\end{lema}

\begin{proof}

We have to show basically that the function $A(t):=\int_M \tilde{\omega}_t(V)d\mu$ is twice differentiable in $t=0$, $A(0)=A'(0)=0$, and the second derivative is
$$
A''(0)=\int_M -\LL_X\omega(\LL_XV)+\LL_Y\omega(V) d\mu.
$$

Since $\tilde\omega_t=h_t^*\omega-\omega-t\LL_X\omega$, Lemma~\ref{le:h} tells us that $A$ is $C^1$, and also $A(0)=A'(0)=0$, because $\tilde\omega_0=0$ and $\left.\frac{\partial}{\partial t}\tilde\omega_t\right|_{t=0}=0$.

Also we have that $t\mapsto \omega+t\LL_X\omega$ is $C^{\infty}$ in $t$, and the second derivative vanishes, so it is enough to show that the map
$$
B(t):=A(t)+\int_M\left(\omega+t\LL_X\omega\right)(V)d\mu=\int_Mh_t^*\omega(V)d\mu
$$
is twice differentiable and 
$$
B''(0)=\int_M -\LL_X\omega(\LL_XV)+\LL_Y\omega(V) d\mu.
$$

Consider a sequence of $C^2$ forms $\omega_n$, $n\geq 1$, that converges to $\omega$ in the $C^1$ topology. The $B_n(t):=\int_Mh_t^*\omega_n(V)d\mu$ clearly converges uniformly to $B(t)$.

From Lemma~\ref{le:h}, \eqref{eq:derivadahw} and Remark~\ref{rem:derivativet}, we have that $h_t^*\omega_n$ is differentiable with respect to $t$, and
$$
\frac{\partial}{\partial t}h_t^*\omega_n=\LL_{X_t}\left(h_t^*\omega_n\right),
$$
which converges uniformly to 
$$
\frac{\partial}{\partial t}h_t^*\omega=\LL_{X_t}\left(h_t^*\omega\right).
$$
This implies that $B_n'$ converges uniformly to $B'$.

Also from Lemma~\ref{le:h}, \eqref{eq:derivadahw2} and Remark~\ref{rem:derivativet}, we have that $h_t^*\omega_n$ is twice differentiable with respect to $t$, and
$$
\frac{\partial^2}{\partial t^2}h_t^*\omega_n=\LL_{X_t}\LL_{X_t}\left(h_t^*\omega_n\right)+\LL_{Y_t}\left(h_t^*\omega_n\right).
$$
Then
$$
B_n''(t)=\int_M\LL_{X_t}\LL_{X_t}\left(h_t^*\omega_n\right)(V)+\LL_{Y_t}\left(h_t^*\omega_n\right)(V)d\mu.
$$
Since we have that
$$
\LL_{X_t}\LL_{X_t}\left(h_t^*\omega_n\right)(V)=\LL_{X_t}\left[\LL_{X_t}\left(h_t^*\omega_n\right)(V)\right]-\LL_{X_t}\left(h_t^*\omega_n\right)(\LL_{X_t}V),
$$
and $X_t$ preserves $\mu$, we get that
\begin{equation}\label{eq:omegan}
B_n''(t)=\int_M-\LL_{X_t}\left(h_t^*\omega_n\right)(\LL_{X_t}V)+\LL_{Y_t}\left(h_t^*\omega_n\right)(V)d\mu.
\end{equation}

Since $\omega_n$ converges to $\omega$ in the $C^1$ topology, the right hand side of \eqref{eq:omegan} converges  uniformly to
$$
\int_M-\LL_{X_t}\left(h_t^*\omega\right)(\LL_{X_t}V)+\LL_{Y_t}\left(h_t^*\omega\right)(V)d\mu,
$$
so $B$ must be also twice differentiable with the second derivative given above. Replacing $t=0$ we obtain the claim.

\end{proof}

The proof of the part (iii) of the theorem follows now from Lemma~\ref{le:3403} and \eqref{eq:forsec3404}.

\end{proof}


\section{The case of variable measure (Proof of Theorem~\ref{teo:response12})}\label{sec:varmeasure}

In this section we will treat the case when the invariant measure $\mu_t$ depends on the map $f_t$.

\begin{proof}[Proof of Theorem~\ref{teo:response12}]

{\bf Part (i).} 
Since $\omega$ is $C^1$, the relation \eqref{eq:etatomega} holds:
$$
\log\eta_t=\log\eta+t\LL_X\omega(V)\circ f+o(t).
$$
Integrating with respect to $\mu_t$ we have,
\begin{eqnarray*}
\lambda(t)&=&\int_M\log\eta_t\,d\mu_t=\int_M\log\eta \,d\mu_t+t\int_M\LL_X\omega(V)\circ f\,d\mu_t+o(t)\\
&=&\int_M\log\eta \,d\mu_t+t\int_M \LL_X\omega(V)\circ f\,d\mu+o(t)\\
&=&\int_M\log\eta \,d\mu_t+t\int_M \LL_X\omega(V)\,d\mu+o(t).
\end{eqnarray*}
We used the facts that $f$ preserves $\mu$, and for a continuous map $g:M\rightarrow \mathbb R$ we have $\int_Mg\,d\mu_t=\int_Mgd\mu+o(1)$.

The formula above shows that $\lambda$ is differentiable in 0 if and only if the family $\mu_t$ has linear response $\mathcal R(\log\eta)$ for the function $\log\eta:M\rightarrow\mathbb R$,
and in this case we obtain
$$
\lambda'(0)=\mathcal R(\log\eta)+\int_M \LL_X\omega(V)d\mu.
$$

{\bf Part (ii).}
Since now $V$ is $C^1$, the relation \eqref{eq:etatV} holds:
$$
\log\eta_t=\log\eta-t\omega(\LL_XV)\circ f+o(t).
$$
Again, integrating with respect to $\mu_t$ we obtain
\begin{eqnarray*}
\lambda(t)&=&\int_M\log\eta_t\,d\mu_t=\int_M\log\eta \,d\mu_t-t\int_M \omega(\LL_XV)\circ f\,d\mu_t+o(t)\\
&=&\int_M\log\eta \,d\mu_t-t\int_M \omega(\LL_XV)\circ f\,d\mu+o(t)\\
&=&\int_M\log\eta \,d\mu_t-t\int_M\omega(\LL_XV)\,d\mu+o(t).
\end{eqnarray*}

The formula above shows again that $\lambda$ is differentiable in 0 if and only if the family $\mu_t$ has linear response $\mathcal R(\log\eta)$ for the function $\log\eta:M\rightarrow\mathbb R$,
and in this case
$$
\lambda'(0)=\mathcal R(\log\eta)-\int_M \omega(\LL_XV)d\mu.
$$

{\bf Part (iii).}
Since both $\omega$ and $V$ are $C^1$, then formula \eqref{eq:forsec3407} holds:
$$
\log\eta_t=\log\eta+\left[t\LL_X\omega(V)+\tilde{\omega}_t(V)+\frac{t^2}{\tilde\eta}\LL_X\omega(f_*V')-\frac{t^2}2[\LL_X\omega(V)]^2\right]\circ f+o(t^2).
$$
Since $\omega$ is in fact $C^2$ then by Lemma~\ref{le:h} we have
$$
\tilde\omega_t=\frac{t^2}2(\LL_X\LL_X\omega+\LL_Y\omega)+o(t^2).
$$
Combining the two relations above we get
\begin{align*}
\log\eta_t&=\log\eta+t\LL_X\omega(V)\circ f+\\
& +\frac{t^2}2\left[\LL_X\LL_X\omega(V)+\LL_Y\omega(V)-(\LL_X\omega(V))^2+\frac 2{\tilde\eta}\LL_X\omega(f_*V')\right]\circ f+o(t^2).
\end{align*}

Integrating with respect to $\mu_t$, we have
\begin{eqnarray*}
\lambda(t)&=&\int_M\log\eta \,d\mu_t+t\int_M \LL_X\omega(V)\circ f\,d\mu_t+\\
& &+\frac{t^2}2\int_M\left[\LL_X\LL_X\omega(V)+\LL_Y\omega(V)-(\LL_X\omega(V))^2+\frac 2{\tilde\eta}\LL_X\omega(f_*V')\right]\circ f\,d\mu_t+o(t^2)\\
&=&\lambda(0)+t\int_M \LL_X\omega(V)d\mu+t^2\mathcal R(\LL_X\omega(V)\circ f)\,d\mu+\\
& &+\frac{t^2}2\int_M\left[\LL_X\LL_X\omega(V)+\LL_Y\omega(V)-(\LL_X\omega(V))^2+\frac 2{\tilde\eta}\LL_X\omega(f_*V')\right]\,d\mu+o(t^2)
\end{eqnarray*}
We used that $\log\eta=\lambda(0)$ is constant, $\mu_t$ has linear response for $\LL_X\omega(V)\circ f$, and $\mu_t=\mu+o(1)$ in the weak* topology. This finishes the proof of the theorem.

\end{proof}

Next we will prove Theorem~\ref{teo:responsegeo}. The strategy is to apply Theorem~\ref{teo:response12} for the family $f_t$ and two families of corresponding invariant measures: measures of maximal entropy and Gibbs u-states. Since the Liouville measure for $f$ is both the unique measure of maximal entropy and the unique Gibbs u-state, the hypothesis of Theorem~\ref{teo:response12} will be satisfied, and we obtain differentiability of the corresponding stable and unstable Lyapunov exponents. On the other hand, the topological entropy will be bounded above and below by the two Lyapunov exponents, so the conclusion follows. Part of the argument is the fact that the derivative of the center exponents vanishes.

\begin{proof}[Proof of Theorem~\ref{teo:responsegeo}]

Assume that $f$ is the time-one map of the geodesic flow on the unit tangent bundle of a manifold with constant negative curvature, denoted $M$. Let $\mu$ be the Liouville measure on $M$, this means that $\mu$ is invariant under $f$, and it is the unique measure of maximal entropy, and the unique Gibbs u-state for $f$.

There exists a $C^1$ neighborhood $\mathcal U$ of $f$ such that all the diffeomorphisms $g\in\mathcal U$ are partially hyperbolic with one-dimensional center, and are $\epsilon$-entropy expansive, for the same $\epsilon>0$ (see for example \cite{LVY2013}). In particular the entropy function is upper-semicontinuous, and there exist measures of maximal entropy.

Assume now that the smooth family $f_t$ is in $\mathcal U$, so each $f_t$ has a measure of maximal entropy $\mu_t$, which we can choose to be ergodic (see \cite{W1982} for example). If, for some sequence $t_n\rightarrow 0$, we have
$$
\mu_0=\lim_{n\rightarrow\infty}\mu_{t_n}
$$
in the weak* topology, then $\mu_0$ must be an invariant measure for $f_0=f$. 

\begin{lema}
The entropy function is upper semicontinuous in both $f$ and $\mu$, in the sense that
$$
\limsup_{n\rightarrow\infty}h_{\mu_{t_n}}(f_{t_n})\leq h_{\mu_0}(f).
$$
\end{lema}

\begin{proof}

This follows basically from Bowen (see \cite{B1972}). Since all the maps $f_{t_n}$ are $\epsilon$-entropy expansive for the same $\epsilon>0$, then for any partition $\mathcal A$ with size smaller that $\epsilon$, we have that
\begin{equation}\label{eq:usc1}
h_{\mu_{t_n}}(f_{t_n})=h_{\mu_{t_n}}(f_{t_n},\mathcal A)=\lim_{k\rightarrow\infty}H_{\mu_{t_n}}\left(\mathcal A\vee f_{t_n}^{-1}\mathcal A\vee\dots\vee f_{t_n}^{-k+1}\mathcal A\right).
\end{equation}

Let $\mathcal A$ be such that the boundaries of the elements have zero measure with respect to $\mu_0$. Since the right hand side limit is decreasing, then given $\delta>0$, there exists $k_0$ such that
\begin{equation}\label{eq:usc2}
H_{\mu_0}\left(\vee_{i=0}^{k_0}f^{-i}\mathcal A\right)<h_{\mu_0}(f)+\delta.
\end{equation}

For every element $A\in\vee_{i=0}^{k_0}f^{-i}\mathcal A$, we can choose $C=\overline C\subset\hbox{in}t(A)\subset A\subset\overline A\subset U=\hbox{int}(U)$, with $\mu_0(U\setminus C)>\delta'$, and $\mu_0(\partial C)=\mu_0(\partial U)=0$. This implies that $\lim_{n\rightarrow\infty}\mu_{t_n}(C)=\mu_0(C)$ and $\lim_{n\rightarrow\infty}\mu_{t_n}(U)=\mu_0(U)$. For large enough $n$, the element $A_n$ from the partition $\vee_{i=0}^{k_0}f_{t_n}^{-i}\mathcal A$ corresponding to $A$ will satisfy $C\subset A_n\subset U$ (this is because $f_{t_n}$ converges to $f$), so we get that
$$
\left|\mu_{t_n}(A_n)-\mu_0(A)\right|<\delta'.
$$
This will hold for all the elements of the partition and for all large enough $n$, so by taking $\delta'$ sufficiently small we obtain that there exists some $n_{\delta}>0$ such that, for any $n>n_{\delta}$ we have
\begin{equation}\label{eq:usc3}
\left|H_{\mu_{t_n}}\left(\vee_{i=0}^{k_0}f_{t_n}^{-i}\mathcal A\right)-H_{\mu_0}\left(\vee_{i=0}^{k_0}f^{-i}\mathcal A\right)\right|<\delta.
\end{equation}

Now putting together \eqref{eq:usc1}, \eqref{eq:usc2} and \eqref{eq:usc3}, and remembering that the sequence which gives the entropy is decreasing, we get that
$$
h_{\mu_{t_n}}(f_{t_n})<h_{\mu_0}(f)+2\delta
$$
for any $n>n_{\delta}$. This finishes the proof of the upper semicontinuity.

\end{proof}

Let us suppose by contradiction that $\mu_0$ is different from the Liouville measure $\mu$, then $h_{\mu_0}(f)<h_{\rm top}(f)$, because $\mu$ is the unique measure of maximal entropy for $f$. This implies that
$$
\limsup_{n\rightarrow\infty}h_{\rm top}(f_{t_n})=\limsup_{n\rightarrow\infty}h_{\mu_{t_n}}(f_{t_n})\leq h_{\mu_0}(f)< h_{\rm top}(f),
$$
so $h_{\rm top}$ is not continuous at $f$, and this contradicts the continuity of the topological entropy obtained in \cite{HZ2014, SY2016}.

Then we have that $\mu_t$ converges in the weak* topology to $\mu$ when $t$ goes to $0$. We consider first the splitting $TM=E^{cs}\oplus E^u$, with $F:=E^{cs}$ and $E:=E^u$. We have that $\eta^u$, the expansion along $E^u$, is constant for $f$, since the curvature is constant, so $\mathcal R(\log\eta)=0$. All the sub-bundles $E^s$, $E^c$ and $E^u$ are smooth.

Then from Theorem~\ref{teo:response12} part ii, we obtain that $t\mapsto \lambda^u(f_t,\mu_t)$ is differentiable in $t=0$, and the derivative is
$$
\left.\frac{\partial}{\partial t}\lambda^u(f_t,\mu_t)\right|_{t=0}=-\int_M\omega_{E^{cs}}(\LL_XV_{E^u})d\mu.
$$

A similar argument gives the differentiability in $0$ of the center Lyapunov exponent, and furthermore:
\begin{eqnarray*}
\left.\frac{\partial}{\partial t}\lambda^c(f_t,\mu_t)\right|_{t=0}&=&-\int_M\omega_{E^{su}}(\LL_XV_{E^c})d\mu\\
&=&\int_M\omega_{E^{su}}(\LL_{V_{E^c}}X)d\mu\\
&=&\int_M\LL_{V_{E^c}}\left[\omega_{E^{su}}(X)\right]d\mu-\int_M\LL_{V_{E^c}}\omega_{E^{su}}(X)d\mu\\
&=&0.
\end{eqnarray*}

We used the fact that we can choose $V_{E^c}$ to be the vector field generating the geodesic flow, so on one hand it preserves the Liouville measure $\mu$, and $\int_M\LL_{V_{E^c}}gd\mu=0$ for any $C^1$ map $g:M\rightarrow\mathbb R$, and on another hand $V_{E^c}$ preserves the form $\omega_{E^{su}}$, so $\LL_{V_{E^c}}\omega_{E^{su}}=0$.

Combining both derivatives above we also get
$$
\left.\frac{\partial}{\partial t}\lambda^{cu}(f_t,\mu_t)\right|_{t=0}=\left.\frac{\partial}{\partial t}\lambda^c(f_t,\mu_t)\right|_{t=0}+\left.\frac{\partial}{\partial t}\lambda^u(f_t,\mu_t)\right|_{t=0}=-\int_M\omega_{E^{cs}}(\LL_XV_{E^u})d\mu.
$$

Since the maps $f_t$ are partially hyperbolic with one dimensional center, and each $\mu_t$ is ergodic, from the Ruelle formula we obtain
$$
h_{\rm top}(f_t)=h_{\mu_t}(f_t)\leq\max\{\lambda^u(f_t,\mu_t),\lambda^{cu}(f_t,\mu_t)\}=:A(t).
$$

Because both $\lambda^u(f_t,\mu_t)$ and $\lambda^{cu}(f_t,\mu_t)$ are differentiable in $t=0$ with the same derivative, we have that also $A(t)$ is differentiable in $t=0$ with the same derivative:
$$
A'(0)=-\int_M\omega_{E^{cs}}(\LL_XV_{E^u})d\mu.
$$

Also it is easy to see that
$$
h_{\rm top}(f)=\lambda^u(f,\mu)=\lambda^{cu}(f,\mu)=A(0).
$$

Thus we obtained that the topological entropy is bounded from above by the differentiable function $A(t)$. The next step is to obtain a bound from below by another differentiable function with the same derivative.

Since the maps $f_t$ are partially hyperbolic and smooth, there exist ergodic Gibbs u-states $m_t$ for each $f_t$ (see \cite{PS1982} or \cite{BDV2005}). This means that $m_t$ is an invariant probability measure for $f_t$, and the disintegrations of $m_t$ along the unstable foliations of $f_t$ are absolutely continuous with respect to the Lebesgue measure on the leaves. This in particular implies that $h_{m_t}(f_t)\geq\lambda^u(f_t,m_t)$ (see for example Ledrappier-Young \cite{LY1985}).

Since $f_t$ converges to $f$ in the $C^2$ topology, any weak limit of Gibbs u-states is a Gibbs u-state (see \cite{PS1982} for example), and since $\mu$ is the unique Gibbs u-state for $f$, we get that the measures $m_t$ converge in the weak* topology to $\mu$ when $t$ goes to $0$. Applying Theorem~\ref{teo:response12} again we obtain that $t\mapsto \lambda^u(f_t,m_t)$ is differentiable in $t=0$, and the derivative is
$$
\left.\frac{\partial}{\partial t}\lambda^u(f_t,m_t)\right|_{t=0}=-\int_M\omega_{E^{cs}}(\LL_XV_{E^u})d\mu.
$$

Using the variational principle we get
$$
h_{\rm top}(f_t)\geq h_{m_t}(f_t)\geq\lambda^u(f_t,m_t)=:B(t),
$$
where $B(t)$ is differentiable in $t=0$ and
$$
B'(0)=-\int_M\omega_{E^{cs}}(\LL_XV_{E^u})d\mu=A'(0).
$$

Evaluating in $t=0$ we obtain again
$$
h_{\rm top}(f)=\lambda^u(f,\mu)=\lambda^{cu}(f,\mu)=B(0).
$$

In conclusion, we have that
$$
B(t)\leq h_{\rm top}(f_t)\leq A(t),\ \ \ B(0)=h_{\rm top}(f_0)=A(0),
$$
and both $A$ and $B$ are differentiable in $0$ with the same derivative. This clearly implies that $t\mapsto h_{\rm top}(f_t)$ is differentiable in $t=0$ and it has the same derivative,
$$
\left.\frac{\partial}{\partial t}h_{\rm top}(f_t)\right|_{t=0}=-\int_M\omega_{E^{cs}}(\LL_XV_{E^u})d\mu.
$$

\end{proof}


\section{The case of flows}\label{sec:flows}

In this section we will prove Theorem~\ref{teo:flows}. We will apply Theorem~\ref{teo:lyapunovH} and Theorem~\ref{teo:lyapunov12} to the family of time-one maps $f_t=\phi_1^t$ of the flows generated by $X_t$. The first step is to find $\overline X$, the vector field tangent to the family $h_t=f_t\circ f^{-1}=\phi_1^t\circ\phi_{-1}$ in $t=0$.

\begin{lema}
Assume hypothesis \ref{HF}is satisfied, $r\geq 2$. Let $\phi$ be the flow $\phi^0$ generated by $X^0$. Then the following holds:
\begin{equation}\label{eq:Xbar}
\overline X=\left.\frac{\partial}{\partial t}\phi_1^t\circ\phi_{-1}\right|_{t=0}=\int_0^1\phi_{s*}X'ds.
\end{equation}
\end{lema}

\begin{proof}
Since it is enough to verify the formula locally in charts, we can assume that we are in $\mathbb R^n$. We have
$$
\frac{\partial}{\partial s}\phi_s^t(\phi_{-1}(p))=X^t(\phi_s^t(\phi_{-1}(p))).
$$
Furthermore
$$
\frac{\partial}{\partial t}\frac{\partial}{\partial s}\phi_s^t(\phi_{-1}(p))=\left(\frac{\partial}{\partial t}X^t\right)(\phi_s^t(\phi_{-1}(p)))+DX^t(\phi_s^t(\phi_{-1}(p)))\cdot\frac{\partial}{\partial t}\phi_s^t(\phi_{-1}(p))
$$
Evaluating in $t=0$ and using the notation
\begin{equation}\label{eq:As}
A(s):=\left.\frac{\partial}{\partial t}\phi_s^t(\phi_{-1}(p))\right|_{t=0}
\end{equation}
we obtain
\begin{equation}\label{eq:A's}
A'(s)=X'(\phi_{s-1}(p))+DX(\phi_{s-1}(p))A(s).
\end{equation}
For $s=0$ we have
\begin{equation}\label{eq:A0}
A(0)=\left.\frac{\partial}{\partial t}\phi_0^t(\phi_{-1}(p))\right|_{t=0}=\left.\frac{\partial}{\partial t}\phi_{-1}(p)\right|_{t=0}=0.
\end{equation}
We have to solve the differential equation \eqref{eq:A's} with the initial condition \eqref{eq:A0}. The homogeneous matrix equation
$$
\Phi'(s)=DX(\phi_{s-1}(p))\Phi(s)
$$
has the fundamental solution $\Phi_0(s)=D\phi_{s-1}(p)$ (this can be seen differentiating the formula $\frac{\partial}{\partial s}\phi_{s-1}(p)=X(\phi_{s-1}(p))$). Then the solution of our equation is
\begin{align*}
\overline X(p)&=A(1)=\Phi_0(1)\Phi_0^{-1}(0)A(0)+\Phi_0(1)\int_0^1\Phi_0^{-1}(s)X'(\phi_{s-1}(p))ds\\
&=0+Id\int_0^1D\phi_{s-1}(p)^{-1}X'(\phi_{s-1}(p))ds\\
&=\int_0^1D\phi_{1-s}(\phi_{s-1}(p))X'(\phi_{s-1}(p))ds=\int_0^1[\phi_{(1-s)*}X'](p)ds\\
&=\int_0^1[\phi_{s*}X'](p)ds.
\end{align*}
This proves the formula \eqref{eq:Xbar} pointwise. The formula also holds if we see the right hand side as a Bochner integral in the Banach space $\mathcal X^1(M)$ of $C^1$ vector fields (with the $C^1$ norm). This is because the map $s\mapsto \phi_{s*}X'$ is continuous in the $C^1$ topology on $\mathcal X^1(M)$ since $\phi_s$ is $C^2$ and $X'$ is $C^1$ (in particular $\overline X$ is $C^1$).

\end{proof}

\begin{proof}[Proof of Theorem~\ref{teo:flows}]

All the claims about the regularity of $\lambda$ follow from the previous results on families of diffeomorphisms, applied to the family $f_t=\phi^t_1$, which will satisfy the hypothesis \ref{H} for $r\geq 3$. We only have to check that the formulas \eqref{eq:derLyapHf}, \eqref{eq:derLyap1Ff} and \eqref{eq:derLyap1Ef} hold.

From \eqref{eq:derLyapH} we have
$$
\lambda'(0)=\left.\frac{\partial}{\partial t}\int_M(\phi_t^{\overline X})^*\omega(V)d\mu\right|_{t=0}=\left.\frac{\partial}{\partial t}\int_M\omega((\phi^{\overline X}_t)_*V)d\mu\right|_{t=0}.
$$
Let
$$
\mathcal D(X):=\left.\frac{\partial}{\partial t}\int_M(\phi_t^X)^*\omega(V)d\mu\right|_{t=0}=\left.\frac{\partial}{\partial t}\int_M\omega((\phi^X_t)_*V)d\mu\right|_{t=0}.
$$
Then $\mathcal D$ is well defined for every vector field
$$
X\in\mathcal X^1_{\mu}(M):=\{X\in\mathcal X^1(M):\ \phi^X_*\mu=\mu\},
$$
i.e. the $C^1$ vector fields which preserve $\mu$.

It is easy to see that $X\in\mathcal X^1_{\mu}(M)$ if and only if $X\in\mathcal X^1(M)$ and for every $C^1$ map $g:M\rightarrow \mathbb R$ we have $\int_Mdg(X)d\mu=0$ (see the proof of Lemma~\ref{lem:XYpresmu}). This in turn implies that $\mathcal X^1_{\mu}(M)$ si a closed linear subspace of $\mathcal X^1(M)$. The formula \eqref{eq:boundderivative} shows then that $\mathcal D:\mathcal X^1_{\mu}(M)\rightarrow\mathbb R$ is a bounded operator.

Although this does not follow directly from the definition, $\mathcal D$ is also a linear operator on $\mathcal X^1_{\mu}(M)$. Let $\beta'<\beta$ such that $\alpha+\beta'>1$, and let $V_n$ be a sequence of $C^{\infty}$ multivector fields converging to $V$ in the topology $C^{\beta'}$. Let
\begin{align*}
\mathcal D_n(X)&:=\left.\frac{\partial}{\partial t}\int_M\omega((\phi^X_t)_*V_n)d\mu\right|_{t=0}=\int_M\omega\left(\left.\frac{\partial}{\partial t}(\phi^X_{t*}V_n)\right|_{t=0}\right)d\mu\\
&=-\int_M\omega(\LL_XV_n)d\mu,
\end{align*}
so the operator $\mathcal D_n$ is linear in $X$.

On the other hand from \eqref{eq:boundderivative} we have
\begin{align*}
|\mathcal D_n(X)-\mathcal D(X)|&=\left|\left.\frac{\partial}{\partial t}\int_M\omega(\phi^X_{t*}(V_n-V))d\mu\right|_{t=0}\right|\\
&\leq C\|\omega\|_{C^{\alpha}}\|V_n-V\|_{C^{\beta'}}\|X\|_{C^1},
\end{align*}
so $\mathcal D_n$ converges to $\mathcal D$, which means that $\mathcal D$ will be also linear.

We have that $X_t\in\mathcal X^2_{\mu}(M)$ for all $t$, so $X'=\lim_{t\rightarrow 0}\frac{X_t-X}t\in\mathcal X^1_{\mu}(M)$. Also $\phi_{s*}X'\in\mathcal X^1_{\mu}(M)$ for all $s\in[0,1]$, because for any $C^1$ map $g:M\rightarrow\mathbb R$ we have
$$
\int_Mdg(\phi_{s*}X')d\mu=\int_M\phi_s^*dg(X')d\mu=\int_Md(g\circ\phi_s)(X')d\mu=0.
$$

Since a bounded linear operator commutes with the Bochner integral we get
\begin{equation}\label{eq:lambdaint}
\lambda'(0)=\mathcal D(\overline X)=\mathcal D\left(\int_0^1\phi_{s*}X'ds\right)=\int_0^1\mathcal D(\phi_{s*}X')ds.
\end{equation}

Now we have to compute $\mathcal D(\phi_{s*}X')$. Observe that $\phi_t^{\phi_{s*}X'}=\phi_s\circ\phi_t^{X'}\circ\phi_{-s}$. Let $\eta^s:M\rightarrow\mathbb R$ be such that $D\phi_s(p)V(p)=\eta^s(p)V(\phi_s(p))$. Then
$$
\phi_{-s*}V(p)=D\phi_{-s}(\phi_s(p))V(\phi_s(p))=[D\phi_s(p)]^{-1}V(\phi_s(p))=\frac 1{\eta^s(p)}V(p),
$$
or $\phi_{-s*}V=\frac 1{\eta^s}V$. Also
$$
\phi_s^*\omega(V)(p)=\omega(\phi_{s*}V(p))=\omega(D\phi_s(p)V(p))=\omega(\eta^s(p)V(\phi_s(p)))=\eta^s(p),
$$
so $\phi_s^*\omega=\eta^s\omega$. Then
$$
\phi_s^*\omega(\phi_{t*}^{X'}\circ\phi_{-s*}V)=\eta^s\omega\left(\phi_{t*}^{X'}\frac 1{\eta^s}V\right)=\frac{\eta^s}{\eta^s\circ\phi_{-t}^{X'}}\omega(\phi_{t*}^{X'}V)
$$
and
\begin{align*}
\omega(\phi^{\phi_{s*}X'}_{t*}V)&=\omega(\phi_{s*}\circ\phi_{t*}^{X'}\circ\phi_{-s*}V)=\left[\phi_s^*\omega(\phi_{t*}^{X'}\circ\phi_{-s*}V)\right]\circ\phi_{-s}\\
&=\left[\frac{\eta^s}{\eta^s\circ\phi_{-t}^{X'}}\right]\circ\phi_{-s}\cdot\left[\omega(\phi_{t*}^{X'}V)\right]\circ\phi_{-s}.
\end{align*}
If $\beta\geq\alpha$ then, using that $\phi$ and $\phi^{X'}$ preserve $\mu$, and applying twice Lemma~\ref{lem:derlog}, we obtain
\begin{align*}
\mathcal D(\phi_{s*}X')&=\left.\frac{\partial}{\partial t}\int_M\omega(\phi^{\phi_{s*}X'}_{t*}V)d\mu\right|_{t=0}=\left.\frac{\partial}{\partial t}\int_M\log\left[\omega(\phi^{\phi_{s*}X'}_{t*}V)\right]d\mu\right|_{t=0}\\
&=\frac{\partial}{\partial t}\left[\int_M\log(\eta^s\circ\phi_{-s})d\mu-\int_M\log(\eta^s\circ\phi_{-t}^{X'}\circ\phi_{-s})d\mu\right.+\\
&+\left.\left.\int_M\log\left[\left(\omega(\phi_{t*}^{X'}V))\right)\circ\phi_{-s}\right]d\mu\right]\right|_{t=0}\\
&=\left.\frac{\partial}{\partial t}\int_M\log\left[\omega(\phi_{t*}^{X'}V))\right]d\mu\right|_{t=0}=\mathcal D(X').
\end{align*}

If $\alpha>\beta$ then one obtains in a similar way that $\mathcal D(\phi_{s*}X')=\mathcal D(X')$ again, this time using the estimation
$$
\phi_t^{X'*}\omega(V)=\left[\omega(\phi^{\phi_{s*}X'}_{t*}V)\right]\circ\phi^{\phi_{s*}X'}_t=\frac{\eta^s\circ\phi_t^{X'}\circ\phi_{-s}}{\eta^s\circ\phi_{-s}}\left[\phi_t^{X'*}\omega(V)\right]\circ\phi_{-s}.
$$

Then from \eqref{eq:lambdaint} we have
$$
\lambda'(0)=\int_0^1\mathcal D(\phi_{s*}X')ds=\int_0^1\mathcal D(X')ds=\mathcal D(X'),
$$
and the formula \eqref{eq:derLyapHf} follows.

For the proof of the formula \eqref{eq:derLyap1Ff} let us assume that $\alpha=1$, then
$$
\lambda'(0)=\int_M\LL_{\overline X}\omega(V)d\mu.
$$
since $X\mapsto\LL_X\omega(V)$ is linear and bounded (in the $C^1$ topology), we can apply again the commutativity of the linear bounded operator with the Bochner integral and we obtain
\begin{equation}\label{eq:derFc1}
\lambda'(0)=\int_M\int_0^1\LL_{(\phi_{s*}X')}\omega(V)dsd\mu=\int_0^1\int_M\LL_{(\phi_{s*}X')}\omega(V)d\mu ds.
\end{equation}
(we changed the order of integration). We further have
\begin{align*}
\int_M\LL_{(\phi_{s*}X')}\omega(V)d\mu&=\int_M\left[\phi_{-s}^*\LL_{X'}(\phi_s^*\omega)\right](V)d\mu=\int_M\left[\phi_{-s}^*\LL_{X'}(\eta^s\omega)\right](V)d\mu\\
&=\int_M\left[\phi_{-s}^*\left(\LL_{X'}(\eta^s)\omega+\eta^s\LL_{X'}\omega\right)\right](V)d\mu\\
&=\int_M\left(\LL_{X'}\eta^s\right)\circ\phi_{-s}d\mu+\int_M\eta^s\LL_{X'}\omega(\phi_{-s*}V)d\mu\\
&=\int_M\LL_{X'}\eta^sd\mu+\int_M\eta^s\LL_{X'}\omega\left(\frac 1{\eta^s}V\right)d\mu\\
&=\int_M\LL_{X'}\omega(V)d\mu.
\end{align*}
Substituting in \eqref{eq:derFc1} we obtain \eqref{eq:derLyap1Ff}.

The proof of \eqref{eq:derLyap1Ef} is similar and we omit it here.

\end{proof}


\section{Non-vanishing of the second derivative}\label{sec:nonvanish}

In this section we will prove Theorem~\ref{teo:l2nonzero}. The idea of the proof is to make an explicit perturbation supported on a small neighborhood of a non-periodic point, mixing the directions of $E^2$ and $E^3$. In this small neighborhood we approximate the bundles by linear ones, and we apply the formula \eqref{eq:derLyap2}. For the linear part of the bundles, the first three terms in \eqref{eq:derLyap2} can be computed explicitly, and the error term is small. For the last term in \eqref{eq:derLyap2}, $\frac 2{\tilde\eta}\LL_X\omega(f_*V')$, we use the expansion \eqref{eq:cve1} of $V'$, and if the return time of the support of the perturbation is large, then only the first term is significant, and this can be again computed explicitly for the linear part.

\begin{proof}

We have that $TM=E^1\oplus E^2\oplus E^3$ is a dominated splitting of class $C^1$ for $f$ ($E^1$ can be trivial) and the measure $\mu$ is the volume on $M$. Let $p_0\in M$ be a non-periodic point for $f_0$ and consider a smooth chart $(U,\phi)$ from a neighborhood $U$ of $p_0$ to $B(0,1)\subset\mathbb R^d$ such that:
\begin{itemize}
\item $\phi(p_0)=0$;
\item $\phi_*E^3(0)=\hbox{span}\{\frac\partial{\partial x_1},\frac\partial{\partial x_2},\dots\frac\partial{\partial x_l}\}$;
\item $\phi_*E^2(0)=\hbox{span}\{\frac\partial{\partial x_{l+1}},\frac\partial{\partial x_{l+2}},\dots\frac\partial{\partial x_m}\}$;
\item $\phi_*E^1(0)=\hbox{span}\{\frac\partial{\partial x_{m+1}},\frac\partial{\partial x_{m+2}},\dots\frac\partial{\partial x_d}\}$;
\item $\phi_*\mu=Leb$.
\end{itemize}

We will work mostly in this chart, and we will make some abuse using the same notations for the objects in $M$ and their push forward in the chart. Let $H:B(0,1)\rightarrow\mathbb R$ be a $C^{\infty}$ function with compact support. Let $H_r(p)=r^2H(\frac pr)$ be the rescaling of $H$ to the ball $B_r:=B(0,r)$. Consider the $C^{\infty}$ vector field $X_r$ with support in $B_r$:
$$
X_r(p)=-\frac{\partial H_r}{\partial x_{l+1}}(p)\frac\partial{\partial x_1}+\frac{\partial H_r}{\partial x_1}(p)\frac\partial{\partial x_{l+1}}\ \ \ (\hbox{here}\ p=(x_1,x_2,\dots x_d)\in\mathbb R^d),
$$
or
$$
X_r(p)=-r\frac{\partial H}{\partial x_{l+1}}\left(\frac pr\right)\frac\partial{\partial x_1}+r\frac{\partial H}{\partial x_1}\left(\frac pr\right)\frac\partial{\partial x_{l+1}}
$$

We will consider the family $h_t^r=\phi^{X_r}_t$, i.e. the flow generated by $X_r$. The flow preserves the volume, in fact it preserves the $x_1x_{l+1}$-planes, and is Hamiltonian on each plane. Observe that we have $X_r=O(r)$ and $DX_r=O(1)$. 

We will show that for $r$ sufficiently small, the family $h_t^r$ satisfies the conclusion of the theorem. We will apply the formula \eqref{eq:derLyap2} from Theorem~\ref{teo:lyapunov12} to the Lyapunov exponents of the bundles $E^3$ and $E^2$.

{\bf Estimation of $\lambda_{r,E^3}''(0)$.} The family $h_t^r$ has $X=X_r$ and $Y=0$, and in the first case we consider $E=E^3$ and $F=E^1\oplus E^2$.

Let $\omega^3=\omega$ be the $C^1$ form corresponding to $F=E^1\oplus E^2$, and $V^3=V$ the $C^1$ multivector corresponding to $E=E^3$ such that $\omega_3(V_3)=1$. Eventually after a composition with a linear map, we can assume that (in the given chart) we have $\omega^3(0)=dx_1\wedge dx_2\wedge\dots\wedge dx_l$, and $V^3(0)=\frac\partial{\partial x_1}\wedge\frac\partial{\partial x_2}\wedge\dots\wedge\frac\partial{\partial x_l}$. Since $\omega^3$ and $V^3$ are $C^1$ we get
$$
\omega^3(p)=dx_1\wedge dx_2\wedge\dots\wedge dx_l+\alpha^3(p), \ \ \alpha^3=O(r),\ \ d\alpha^3=O(1),
$$
$$
V^3(p)=\frac\partial{\partial x_1}\wedge\frac\partial{\partial x_2}\wedge\dots\wedge\frac\partial{\partial x_l}+T^3(p),\ \ T^3=O(r),\ \ DT^3=O(1).
$$
Since $\omega^3,\ d\omega^3,\ V^3$ and $DV^3$ are all $O(1)$, we have that $\LL_{X_r}\omega^3$ and $\LL_{X_r}V^3$ are $O(1)$, while $\LL_{X_r}\alpha^3$ and $\LL_{X_r}T^3$ are $O(r)$.

Let us estimate the Lie derivative of $\omega^3$ with respect to $X_r$. We have:
$$
\LL_{X_r}dx_1=d(\LL_{X_r}x_1)=d\left[-r\frac{\partial H}{\partial x_{l+1}}\left(\frac pr\right)\right]=-\sum_{i=1}^n\frac{\partial^2 H}{\partial x_{l+1}\partial x_i}\left(\frac pr\right)dx_i,
$$
$$
\LL_{X_r}dx_{l+1}=d(\LL_{X_r}x_{l+1})=d\left[r\frac{\partial H}{\partial x_1}\left(\frac pr\right)\right]=\sum_{i=1}^n\frac{\partial^2 H}{\partial x_1\partial x_i}\left(\frac pr\right)dx_i,
$$
and $\LL_{X_r}dx_i=0$ for all the other $i$. Then
\begin{eqnarray*}
\LL_{X_r}\omega^3&=&\LL_{X_r}(dx_1\wedge\dots\wedge dx_l)+\LL_{X_r}\alpha^3=(\LL_{X_r}dx_1)\wedge dx_2\wedge\dots\wedge dx_l+O(r)\\
&=&-\left(\sum_{i=1}^n\frac{\partial^2 H}{\partial x_{l+1}\partial x_i}\left(\frac pr\right)dx_i\right)\wedge dx_2\wedge\dots\wedge dx_l+O(r)\\
&=&-\frac{\partial^2 H}{\partial x_1\partial x_{l+1}}\left(\frac pr\right)dx_1\wedge\dots\wedge dx_l-\frac{\partial^2 H}{\partial x_{l+1}^2}\left(\frac pr\right)dx_{l+1}\wedge dx_2\wedge\dots\wedge dx_l+\\
&&-\sum_{i=l+2}^n\frac{\partial^2 H}{\partial x_{l+1}\partial x_i}\left(\frac pr\right)dx_i\wedge dx_2\wedge\dots\wedge dx_l+O(r).
\end{eqnarray*}

Now let us estimate the Lie derivative of $V^3$. We also have:
$$
\LL_{X_r}\frac\partial{\partial x_i}=\left[X_r,\frac\partial{\partial x_i}\right]=\frac{\partial X_r}{\partial x_i}=-\frac{\partial^2H}{\partial x_{l+1}\partial x_i}\left(\frac pr\right)\frac\partial{\partial x_1}+\frac{\partial^2H}{\partial x_1\partial x_i}\left(\frac pr\right)\frac\partial{\partial x_{l+1}}.
$$
Then we have:
\begin{eqnarray*}
\LL_{X_r}V^3&=&\LL_{X_r}\left(\frac\partial{\partial x_1}\wedge\dots\wedge\frac\partial{\partial x_l}+T^3\right)\\
&=&\sum_{i=1}^l\frac\partial{\partial x_1}\wedge\dots\wedge \LL_{X_r}\frac\partial{\partial x_i}\wedge\dots\wedge\frac\partial{\partial x_l}+O(r)\\
&=&-\frac{\partial^2H}{\partial x_1\partial x_{l+1}}\left(\frac pr\right)\frac\partial{\partial x_1}\wedge\frac\partial{\partial x_2}\wedge\dots\wedge\frac\partial{\partial x_l}+\\
&&+\sum_{i=1}^l\frac{\partial^2H}{\partial x_1\partial x_i}\left(\frac pr\right)\frac\partial{\partial x_1}\wedge\dots\wedge\frac\partial{\partial x_{l+1}}\wedge\dots\wedge\frac\partial{\partial x_l}+O(r).
\end{eqnarray*}

Computing $\LL_{X_r}\omega^3(V^3)$ we get
$$
\LL_{X_r}\omega^3(V^3)=-\frac{\partial^2 H}{\partial x_1\partial x_{l+1}}\left(\frac pr\right)+O(r).
$$

Computing $\LL_{X_r}\omega^3(\LL_{X_r}V^3)$ we get
$$
\LL_{X_r}\omega^3(\LL_{X_r}V^3)=\left[\frac{\partial^2 H}{\partial x_1\partial x_{l+1}}\left(\frac pr\right)\right]^2-\frac{\partial^2 H}{\partial x_{l+1}^2}\left(\frac pr\right)\frac{\partial^2H}{\partial x_1^2}\left(\frac pr\right)+O(r).
$$

Now we will estimate the second derivative of the Lyapunov exponent corresponding to $E^3$. We remark that $\LL_{X_r}\omega^3$ is supported in $B_r=B(0,r)$, since here is where $X_r$ is supported. Then we have
\begin{eqnarray*}
\lambda_{r,E^3}''(0)&=&\int_M-\LL_{X_r}\omega^3(\LL_{X_r}V^3)-\left(\LL_{X_r}\omega^3(V^3)\right)^2+\frac 2{\tilde\eta_3}\LL_{X_r}\omega^3(f_*V^{3'})d\mu\\
&=&-\int_{B_r}\LL_{X_r}\omega^3(\LL_{X_r}V^3)+\left(\LL_{X_r}\omega^3(V^3)\right)^2d\mu+\int_{B_r}\frac 2{\tilde\eta_3}\LL_{X_r}\omega^3(f_*V^{3'})d\mu\\
&:=&-I_1+I_2.
\end{eqnarray*}

Computing $I_1$ in the chart and using the change of variables $p=(x_i)_{1\leq i\leq d}=rq=(ry_i)_{1\leq i\leq d}$, $dp=r^ddq$, we get
\begin{align*}
I_1&=\int_{B_r}\LL_{X_r}\omega^3(\LL_{X_r}V^3)+\left(\LL_{X_r}\omega^3(V^3)\right)^2d\mu\\
&=\int_{B_r}\left[\frac{\partial^2 H}{\partial x_1\partial x_{l+1}}\left(\frac pr\right)\right]^2-\frac{\partial^2 H}{\partial x_{l+1}^2}\left(\frac pr\right)\frac{\partial^2H}{\partial x_1^2}\left(\frac pr\right)+\left[\frac{\partial^2 H}{\partial x_1\partial x_{l+1}}\left(\frac pr\right)\right]^2+O(r)d\mu\\
&=r^d\int_{B(0,1)}2\left[\frac{\partial^2 H}{\partial x_1\partial x_{l+1}}\right]^2-\frac{\partial^2 H}{\partial x_{l+1}^2}\frac{\partial^2H}{\partial x_1^2}d\mu+\int_{B(0,r)}O(r)d\mu\\
&=r^dK+O(r^{d+1}),
\end{align*}
where $K$ is independent of $r$, and is given by the formula
$$
K=\int_{B(0,1)}2\left[\frac{\partial^2 H}{\partial x_1\partial x_{l+1}}\right]^2-\frac{\partial^2 H}{\partial x_{l+1}^2}\frac{\partial^2H}{\partial x_1^2}d\mu.
$$
Integrating by parts the second term we obtain
$$
\int_{B(0,1)}\frac{\partial^2 H}{\partial x_{l+1}^2}\frac{\partial^2H}{\partial x_1^2}d\mu=-\int_{B(0,1)}\frac{\partial H}{\partial x_1}\frac{\partial^3H}{\partial x_1\partial x_{l+1}^2}d\mu=\int_{B(0,1)}\left[\frac{\partial^2 H}{\partial x_1\partial x_{l+1}}\right]^2d\mu,
$$
so we have in fact that, if $H$ is not constant zero, then
$$
K=\int_{B(0,1)}\left[\frac{\partial^2 H}{\partial x_1\partial x_{l+1}}\right]^2d\mu>0.
$$

Now we will estimate $I_2$. Remember that $\mathcal P=Id-\omega^3(\cdot)V^3$ is the projection to the kernel of $\omega_F=\omega^3$. Denote
$$
A^3(p):=\mathcal P(\LL_{X_r}V^3)(p)=O(1).
$$

We can apply formula \eqref{eq:cve1} and we get
$$
V^{3'}=\left[\left.\left[Id-\left(\frac{f_*}{\tilde\eta_3}\right)\right]\right|_{F\wedge E^{\wedge(k-1)}}\right]^{-1}\mathcal P(\LL_{X_r}V^3)=\left[\left.\left[Id-\left(\frac{f_*}{\tilde\eta_3}\right)\right]\right|_{F\wedge E^{\wedge(k-1)}}\right]^{-1}A^3.
$$

In this case since $E=E^3$, and $E$ dominates $F=E^1\oplus E^2$, the operator $\left.\left(\frac{f_*}{\tilde\eta_3}\right)\right|_{F\wedge E^{\wedge(k-1)}}$ is a contraction by some $\nu\in(0,1)$, so 

$$
V^{3'}=\sum_{k\geq 0}\left(\frac{f_*}{\tilde\eta_3}\right)^kA^3.
$$

The support of $A^3$ is inside $B_r$, so the support of $f_*^kA^3$ is inside $f^k(B_r)$. Let $t_r=\min \{k\geq 1, B_r\cap f^k(B_r)\neq 0\}$ be the return time of the set $B_r$ to itself under $f$. Since $p$ is not periodic we have $\lim_{r\rightarrow 0}t_r=\infty$. Since $\left.\left(\frac{f_*}{\tilde\eta_3}\right)\right|_{F\wedge E^{\wedge(k-1)}}$ is a contraction by $\nu<1$, we have that, for $C$ independent of $r$
$$
\left\|\left(\frac{f_*}{\tilde\eta_3}\right)^kA_3\right\|\leq\nu^k\|A^3\|\leq C\nu^k.
$$
Since $\LL_{X_r}\omega^3$ is also bounded independently of $r$, we have the following estimation of $I_2$
\begin{eqnarray*}
I_2&=&\int_{B_r}\frac 2{\tilde\eta}\LL_{X_r}\omega^3(f_*V^{3'}))d\mu=\int_{B_r} 2\LL_{X_r}\omega^3\left(\frac{f_*}{\tilde\eta_3}\sum_{k\geq 0}\left(\frac{f_*}{\tilde\eta_3}\right)^kA^3\right)\\
&=&\sum_{k\geq1}\int_{B_r}2\LL_{X_r}\omega^3\left(\left(\frac{f_*}{\tilde\eta_3}\right)^kA^3\right)d\mu=\sum_{k\geq t_r}\int_{B_r\cap f^k(B_r)}2\LL_{X_r}\omega^3\left(\left(\frac{f_*}{\tilde\eta_3}\right)^kA^3\right)d\mu.
\end{eqnarray*}
and
$$
|I_2|\leq\sum_{k=t_r}^{\infty}2\int_{B_r\cap f^k(B_r)}C\nu^kd\mu\leq\sum_{k=t_r}^{\infty}2\int_{B_r}C\nu^Kdm\leq C\nu^{t_r}r^d,
$$
where $C$ is again some constant independent on $r$.

In conclusion, 
\begin{equation}
\lambda_{r,E^3}''(0)=-Kr^d+O(r^{d+1})+O(r^d\nu^{t_r}),
\end{equation}
so $\lim_{r\rightarrow 0}\frac{\lambda_{r,E^3}''(0)}{r^n}=-K<0$ and as a consequence $\lambda_{r,E^3}''(0)<0$ for all $r$ sufficiently small.

{\bf Estimation of $\lambda_{r,E^2}''(0)$.} We can do a similar estimation for the integrated Lyapunov exponent corresponding to the bundle $E^2$. In this case we will consider the corresponding $\omega^2$, $V^2$, $V^{2'}$, $\tilde\eta_2$.

We obtain again
\begin{eqnarray*}
\lambda_{r,E^2}''(0)&=&\int_M-\LL_{X_r}\omega^2(\LL_{X_r}V^2)-\left(\LL_{X_r}\omega^2(V^2)\right)^2+\frac 2{\tilde\eta_2}\LL_{X_r}\omega^2(f_*V^{2'})d\mu\\
&=&-\int_{B_r}\LL_{X_r}\omega^2(\LL_{X_r}V^2)+\left(\LL_{X_r}\omega^2(V^2)\right)^2d\mu+\int_{B_r}\frac 2{\tilde\eta_2}\LL_{X_r}\omega^2(f_*V^{2'})d\mu\\
&:=&-I_3+I_4.
\end{eqnarray*}

Because of the symmetry between $E^2$ and $E^3$ with respect to $X_r$, we obtain again the estimate
$$
I_3=r^dK+O(r^{d+1}),
$$
where, if $H$ is not trivial,
$$
K=\int_{B(0,1)}\left[\frac{\partial^2 H}{\partial x_1\partial x_{l+1}}\right]^2d\mu>0.
$$

We will estimate $I_4$. We write again
$$
\omega^2(p)=dx_{l+1}\wedge\dots\wedge dx_m+\alpha^2(p), \ \ \alpha^2=O(r),\ \ d\alpha^2=O(1),
$$
$$
V^2(p)=\frac\partial{\partial x_{l+1}}\wedge\dots\wedge\frac\partial{\partial x_m}+T^2(p),\ \ T^2=O(r),\ \ DT^2=O(1).
$$
We obtain the analog formula for $\LL_{X_r}\omega^3$ and $\LL_{X_r}V^2$:
\begin{eqnarray*}
\LL_{X_r}\omega^3&=&\frac{\partial^2 H}{\partial x_1\partial x_{l+1}}\left(\frac pr\right)dx_{l+1}\wedge\dots\wedge dx_m+\frac{\partial^2 H}{\partial x_1^2}\left(\frac pr\right)dx_1\wedge dx_{l+2}\wedge\dots\wedge dx_m+\\
&&+\sum_{i\in\{2,\dots,l,m+1,\dots,n\}}\frac{\partial^2 H}{\partial x_1\partial x_i}\left(\frac pr\right)dx_i\wedge dx_{l+2}\wedge\dots\wedge dx_m+O(r),
\end{eqnarray*}
\begin{eqnarray*}
\LL_{X_r}V^2&=&\frac{\partial^2H}{\partial x_1\partial x_{l+1}}\left(\frac pr\right)\frac\partial{\partial x_{l+1}}\wedge\dots\wedge\frac\partial{\partial x_m}+\\
&&-\sum_{i=l+1}^m\frac{\partial^2H}{\partial x_{l+1}\partial x_i}\left(\frac pr\right)\frac\partial{\partial x_{l+1}}\wedge\dots\wedge\frac\partial{\partial x_1}\wedge\dots\wedge\frac\partial{\partial x_m}+O(r).
\end{eqnarray*}

Consider now $\mathcal P=Id-\omega^2(\cdot)V^2$ the projection to the kernel of $\omega_F=\omega^2$, and denote
$$
A^2(p):=\mathcal P(\LL_{X_r}V^2)(p)=\sum_{i=l+1}^m\frac{\partial^2H}{\partial x_{l+1}\partial x_i}\left(\frac pr\right)\frac\partial{\partial x_{l+1}}\wedge\dots\wedge\frac\partial{\partial x_1}\wedge\dots\wedge\frac\partial{\partial x_m}+O(r).
$$
From formula \eqref{eq:cve1} we have
$$
V^{2'}=\left[\left.\left[Id-\left(\frac{f_*}{\tilde\eta_2}\right)\right]\right|_{F\wedge E^{\wedge(k-1)}}\right]^{-1}A^2.
$$

This time the operator $\left.\left(\frac{f_*}{\tilde\eta_2}\right)\right|_{F\wedge E^{\wedge(k-1)}}$ is not a contraction, but it is hyperbolic and decomposes in the sum of the contraction $\mathcal T_1$ and the expansion $\mathcal T_3$. In our case
$$
\mathcal T_1A^2=O(r)
$$
and
$$
\mathcal T_3A^2=\sum_{i=l+1}^m\frac{\partial^2H}{\partial x_{l+1}\partial x_i}\left(\frac pr\right)\frac\partial{\partial x_{l+1}}\wedge\dots\wedge\frac\partial{\partial x_1}\wedge\dots\wedge\frac\partial{\partial x_m}+O(r).
$$

Then applying the second part of \eqref{eq:cve1} we get 
$$
V^{2'}=\sum_{k\geq 0}\left(\frac{f_*}{\tilde\eta_2}\right)^k\mathcal T_1A^2-\sum_{k\geq 1}\left(\frac{f_*}{\tilde\eta_2}\right)^{-k}\mathcal T_3A^2
$$

Now we evaluate $I_4$. Let $t_r$ be the first return time of $B_r$ under $f^{-1}$. Like in the previous case, since $\left.\left(\frac{f_*}{\tilde\eta_2}\right)\right|_{E^3\wedge E^{\wedge(k-1)}}$ is an expansion, there exist $C>0$ and $0<\nu<1$ such that
$$
\left\|\left(\frac{f_*}{\tilde\eta_2}\right)^{-k}\mathcal T_3A^2\right\|\leq C\nu^k.
$$
We have
\begin{align*}
I_4&=\int_{B_r}\frac 2{\tilde\eta_2}\LL_{X_r}\omega^2(f_*V^{2'})d\mu\\
&=\int_{B_r}2\LL_{X_r}\omega^2\left[\frac{f_*}{\tilde\eta_2}\sum_{k\geq 0}\left(\frac{f_*}{\tilde\eta_2}\right)^k\mathcal T_1A^2-\frac{f_*}{\tilde\eta_2}\sum_{k\geq 1}\left(\frac{f_*}{\tilde\eta_2}\right)^{-k}\mathcal T_3A^2\right]d\mu\\
&=2\sum_{k\geq 1}\int_{B_r}2\LL_{X_R}\omega^2\left[\left(\frac{f_*}{\tilde\eta_2}\right)^k\mathcal T_1A^2\right]d\mu-2\sum_{k\geq 0}\int_{B_r}2\LL_{X_R}\omega^2\left[\left(\frac{f_*}{\tilde\eta_2}\right)^{-k}\mathcal T_3A^2\right]d\mu\\
&=-\int_{B_r}2\LL_{X_R}\omega^2\left(\mathcal T_3A^2\right)d\mu-2\sum_{k\geq t_r}\int_{B_r}2\LL_{X_R}\omega^2\left[\left(\frac{f_*}{\tilde\eta_2}\right)^{-k}\mathcal T_3A^2\right]d\mu+O(r^{d+1})\\
&=\int_{B_r}2\frac{\partial^2 H}{\partial x_{l+1}^2}\left(\frac pr\right)\frac{\partial^2H}{\partial x_1^2}\left(\frac pr\right)d\mu+O(r^{d+1})+O(r^d\nu^{t_r})\\
&=2K+O(r^{d+1})+O(r^d\nu^{t_r}).
\end{align*}

Putting $I_3$ and $I_4$ together we obtain
\begin{equation}
\lambda_{r,E^2}''(0)=Kr^d+O(r^{d+1})+O(r^d\nu^{t_r}),
\end{equation}
so $\lim_{r\rightarrow 0}\frac{\lambda_{r,E^3}''(0)}{r^d}=K>0$ and as a consequence $\lambda_{r,E^3}''(0)>0$ for all $r$ sufficiently small.

\end{proof}


\section{Critical points and rigidity}\label{sec:critandrig}

In this section we will prove Theorem~\ref{teo:rigidity} and Corollary~\ref{cor:rigidityA}. Let us remark that assuming better regularity of the bundles E and (or) F would simplify the proof considerably. For example if $E$ is sufficiently smooth we can assume that $V$ is constant $\frac{\partial}{\partial x_1}$ in some local chart, and this simplifies considerably the expression of the derivative of the Lyapunov exponent. However we want to apply the result to dynamical foliations in order to obtain rigidity, and assuming better regularity of the invariant bundles already implies rigidity in many situations (see \cite{HK1990}).

The strategy of the proof is the following. First we see that if the diffeomorphism $f$ is critical, then by moving the derivatives in the formula \eqref{eq:derLyap1F} away from $X$, we obtain that some specific continuous one-form on $M$ is locally exact, or the integral vanishes over every local piecewise $C^1$ closed curve. Then we choose a closed curve which is a ``rectangle" formed by $\mathcal W^E$ and $\mathcal W^F$ pieces. The integral of the one-form over a $\mathcal W^E$ piece can be written in terms of the densities of the disintegrations of the volume along $\mathcal W^E$.  The integral of the one-form over a $\mathcal W^F$ piece can be written in terms of the Jacobian of the holonomy along $\mathcal W^F$ between $\mathcal W^E$ pieces. Then the vanishing of the integral over all such closed curves will imply that the disintegrations of the volume along $\mathcal W^E$ are invariant under the $\mathcal W^F$ holonomy.

In the rest of the section $\mu$ will be the Lebesgue measure on the manifold $M$, and $\mu_E$ is the Lebesgue measure on the leaves of the foliation $\mathcal W^E$ tangent to the sub-bundle $E$. Given a foliation chart $U$ for the foliation $\mathcal W^E$, following Rokhlin \cite{R1952} we can define the disintegrations of $\mu$ along the local leaves of $\mathcal W^E$ in $U$. We will denote by $m_E(x)$ the disintegration along the local leaf of $\mathcal W^E$ passing through $x$, this is defined for $\mu$-almost every $x$ in $U$. There is also a quotient measure defined on the space of local leaves in $U$. If $T$ is a transversal to the foliation $\mathcal W^E$ in $U$, the quotient measure induced on $T$ is denoted $\mu_T$.

In our case the disintegrations $m_E$ of $\mu$ along the local leaves of $\mathcal W^E$ are absolutely continuous with respect to the Lebesgue measure on the leaves, $\mu_E$. We denote by $\rho_x$ the densities of the disintegrations with respect to Lebesgue, or
$$
dm_E(x)=\rho_xd\mu_E(x).
$$
If there is no need to specify it, we will drop the $x$ from the above notations.

\begin{proof}[Proof of Theorem~\ref{teo:rigidity}]

We divide the proof in several steps.

{\bf Step 1: Moving the derivatives away from $X$.}

If $f$ is critical for $E$ and the volume $\mu$, then by \eqref{eq:derLyap1F} we know that for every $C^{3}$ divergence free (volume preserving) vector field $X$ on $M$ we have 
$$
\int_M\LL_X\omega(V)d\mu=0.
$$

Suppose that $X$ is supported in a foliation chart $U$ for $\mathcal W^E$, with some transversal $T$. Assume that in this chart we have the quotient measure $\mu_T$ and the disintegration of the volume along the $\mathcal W^E$ leaves have the density $\rho$ with respect to Lebesgue on the leaves $\mu_E$. Since $\frac V{\|V\|}$ is a unit vector field generating $\mathcal W^E$, and $\mu_E$ is the Lebesgue measure on $\mathcal W^E$, the Fundamental Theorem of Calculus gives that for every piece $\mathcal W^E(p)$ and every $C^1$ function $g:\mathcal W^E(p)\rightarrow\mathbb R$ with compact support,
$$
\int_{\mathcal W^E(p)}dg\left(\frac V{\|V\|}\right) d\mu_E=0.
$$
This means that we can integrate by parts on $\mathcal W^E(p)$. We have
\begin{align*}
0&=\int_M\LL_X\omega(V)d\mu=\int_Mi_Xd\omega(V)+di_X\omega(V)d\mu\\
&=\int_M-i_Vd\omega(X)d\mu+\int_Ud[\omega(X)](V)d\mu\\
&=\int_M-i_Vd\omega(X)d\mu+\int_T\int_{\mathcal W^E(p)}d[\omega(X)]\left(\frac V{\|V\|}\right)\rho\|V\| d\mu_Ed\mu_T\\
&=\int_M-i_Vd\omega(X)d\mu-\int_T\int_{\mathcal W^E(p)}d(\rho\|V\|)\left(\frac V{\|V\|}\right)\omega(X) d\mu_Ed\mu_T\\
&=\int_M-i_Vd\omega(X)d\mu-\int_T\int_{\mathcal W^E(p)}d\log(\rho\|V\|)(V)\omega(X)\rho d\mu_Ed\mu_T\\
&=-\int_M\left[i_Vd\omega+d\log(\rho\|V\|)(V)\omega\right](X)d\mu:=-\int_M\alpha(X)d\mu,
\end{align*}
where $\alpha=i_Vd\omega+d\log(\rho\|V\|)(V)\omega$ is a continuous 1-form in $U$. We used the fact that $\rho$ and $V$ are differentiable along the $\mathcal W^E$ leaves, which we know from the hypothesis on $E$.

\begin{rem}
If $V$ is $C^1$ then $d\log(\rho\|V\|)(V)=\hbox{div}(V)$ and the formula above can be deduced directly ($vol$ is the volume $d$-form):
$$
\int_Mdi_X\omega(V)d\mu=\int_Mdi_X\omega\wedge i_Vvol=-\int_M\omega(X)di_Vvol=\int_M\omega(X)\hbox{div}(V)d\mu.
$$
\end{rem}

{\bf Step 2: The form $\alpha=i_Vd\omega+d\log(\rho\|V\|)(V)\omega$ is exact.}

From the previous step we know that for every $C^3$ divergence free vector field $X$ supported in $U$ we have
\begin{equation}\label{eq:alphazero}
\int_U\alpha(X)d\mu=0.
\end{equation}
The proof of this step is given by the following lemma.

\begin{lema}\label{le:exactform}
Let $U$ be an open set in $\mathbb R^d$, and $\alpha$ a continuous 1-form in $U$. Suppose that for every $C^{\infty}$ divergence free vector field $X$ supported in $U$ we have $\int_U\alpha(X)d\mu=0$.

Then $\alpha$ is exact, i.e. for every piecewise $C^1$ curve $\gamma\subset U$ we have
$$
\int_{\gamma}\alpha=0.
$$
\end{lema}

\begin{proof}
First assume that $\gamma$ is a $C^{\infty}$ simple closed curve in $U$. There exists a tubular neighborhood $U_0$ of $\gamma$ which is $C^{\infty}$ diffeomorphic to $B_{\mathbb R^{n-1}}(0,1)\times\mathbb T^1$, and thus it is $C^{\infty}$ foliated by $C^{\infty}$ closed curves corresponding to the curves $\{x\}\times\mathbb T^1$, $x\in B_{\mathbb R^{n-1}}(0,1)$. Let $T$ be a $C^{\infty}$ transversal to the foliation of closed curves of the tubular neighborhood of $\gamma$.

Let $X_0$ be a $C^{\infty}$ vector field in $U_0$ tangent to the foliation of $U_0$ by closed curves, and such that the period of all the closed curves for the flow $\phi^{X_0}$ generated by $X_0$ is one (this can be done by pulling back under the diffeomorphisms the unit vector field in $B_{\mathbb R^{n-1}}(0,1)\times\mathbb T^1$ tangent to the $\{x\}\times\mathbb T^1$ curves).

We claim that we can rescale $X_0$ by a $C^{\infty}$ nonzero scalar function $f$ such that $fX_0$ is divergence free. The map $f$ must satisfy:
$$
\hbox{div}(fX_0)=f\hbox{div}(X_0)+df(X_0)=0,
$$
or
$$
d(\log f)(X_0)=-\hbox{div}(X_0).
$$
We can take the initial conditions $f(p)=1$ (or $\log f(p)=0$) for all $p\in T$, and we get that for all $t\in[0,1]$
$$
\log f\left(\phi_t^{X_0}(p)\right)=-\int_0^t\hbox{div}(X_0)\left(\phi_s^{X_0}(p)\right)ds.
$$
This formula defines the $f$ on $U_0$, but it may have a discontinuity at the return of the flow $\phi^{X_0}$ at the transversal $T$. The flow $\phi_t^{X_0}$ is periodic with period one, or $\phi_1^{X_0}=Id$, and $\det\left[D\phi_1^{X_0}\right]=1$. From the Liouville formula we have that
$$
\det\left[D\phi_1^{X_0}(p)\right]=\det\left[D\phi_0^{X_0}(p)\right]\exp\left(\int_0^1\hbox{div}(X_0)\left(\phi_s^{X_0}(p)\right)ds\right).
$$
This implies that $\int_0^1\hbox{div}(X_0)\left(\phi_s^{X_0}(p)\right)ds=0$, so $f(p)=f\left(\phi_1^{X_0}(p)\right)$, and that $f$ is well defined and $C^{\infty}$ on $U_0$, and $fX_0$ is a $C^{\infty}$ rescaling of $X_0$ which is divergence free.

Let $X:=fX_0$ be the divergence free vector field tangent to the foliation of $U_0$ by closed curves. Let us comment that the election of $X$ is not unique, in fact given any $C^{\infty}$ scalar function $g$ which is constant on the closed leaves of the foliation, we have $\hbox{div}(gX)=0$. Then without loss of generality we can assume that the period of $\phi^X$ on each closed curve $\gamma_p$ is one.

Up to rescaling, we can assume that the volume $\mu$ restricted to $U_0$ is a probability invariant measure for $\phi^{X_0}$. The ergodic decomposition of $\mu$ for the flow generated by $X$ in $U_0$ consists of the measures $\nu_p$ supported on the closed curves $\gamma_p,\ p\in T$:
$$
\int_{\gamma_p} gd\nu_p=\int_0^1g(\phi_t^X(p))dt,\ \ \forall g\in C^0(U_0,\mathbb R),
$$
and
$$
\int_{U_0}gd\mu=\int_T\int_{\gamma_p}gd\nu_pdm_T,\ \ \forall g\in C^0(U_0,\mathbb R),
$$
where $m_T$ is the quotient measure on $T$, which can be identified with the set of ergodic invariant measures for $\phi^{X_0}$. It is easy to see that $m_T$ has full support (it is in fact smooth if $T$ is smooth).

Also observe that we have
$$
\int_{\gamma_p}\alpha(X)d\nu_p=\int_0^1\alpha(X)(\phi_t^X(p))dt=\int_0^1\phi_{\cdot}^{X*}\alpha=\int_{\gamma_p}\alpha
$$
where $\phi_{\cdot}^{X*}\alpha$ is the pull back of $\alpha$ by the map $\phi_{\cdot}^X:[0,1]\rightarrow U_0$.

Now we will apply this to \eqref{eq:alphazero} with $X$ replaced by the divergence free $gX$, where $g$ is constant on the closed leaves and supported in $U_0$. We get
$$
0=\int_{U_0}\alpha(gX)d\mu=\int_Tg(p)\int_{\gamma_p}\alpha(X)d\nu_pdm_T=\int_Tg(p)\int_{\gamma_p}\alpha dm_T.
$$
Since the above relation holds for any compactly supported $C^{\infty}$ $g:T\rightarrow\mathbb R$, and the map $p\mapsto \int_{\gamma_p}\alpha$ is continuous, we have that $\int_{\gamma_p}\alpha=0$ for all $p\in T=\hbox{supp}(m_T)$, and in particular for the initial smooth simple closed curve $\gamma$.

\begin{rem}
An alternative method to the use of the ergodic decomposition would be to use the explicit formula of densities $\rho$ of the disintegrations of the volume along the curves of a flow generated by a $C^1$ vector field $X$.
\end{rem}

We continue with the proof of Lemma~\ref{le:exactform}. We obtained that $\int_{\gamma}\alpha=0$ for any simple closed $C^{\infty}$ curve $\gamma$. Now given any piecewise $C^1$ closed curve $\gamma$, since $\alpha$ is $C^0$, we can make a standard approximation of $\gamma $ with (a sum of) $C^{\infty}$ simple closed curves $\gamma_k$, in the sense that $\int_{\gamma_k}\alpha\rightarrow\int_{\gamma}\alpha$, and the conclusion of the lemma follows.
\end{proof}

\begin{rem}
The converse of the Lemma~\ref{le:exactform} is also true, and is an immediate consequence of the Birkhoff Ergodic Theorem.
\end{rem}

{\bf Step 3: The integral of $\alpha$ on $\mathcal W^E$ curves.}

We want to apply the previous step to a piecewise $C^1$ simple closed curve formed by two $\mathcal W^E$ pieces and two $\mathcal W^F$ pieces (curves inside $\mathcal W^E$ and $\mathcal W^F$). Suppose that we have such a ``rectangle'' formed by the points $a,b,c,d$ with $\gamma_{ab}$ a $\mathcal W^E$ piece between $a$ and $b$, $\gamma_{bc}$ a $\mathcal W^F$ piece between $b$ and $c$, $\gamma_{cd}$ a $\mathcal W^E$ piece between $c$ and $d$, and $\gamma_{da}$ a $\mathcal W^F$ piece between $d$ and $a$.

The restriction of $i_Vd\omega$ to $E=T\mathcal W^E$ is zero, because $V$ is generating $E$ and $i_Vd\omega(V)=d\omega(V,V)=0$. Since $\frac V{\| v\|}$ gives the arc length parametrization of $\gamma_{ab}$, we have
\begin{align*}
\int_{\gamma_{ab}}\alpha&=\int_{\gamma_{ab}}i_Vd\omega+\int_{\gamma_{ab}}d\log(\rho\|V\|)(V)\omega\\
&=\int_{\gamma_{ab}}d\log(\rho\|V\|)(V)\omega\left(\frac V{\|V\|}\right)d\mu_E=\int_{\gamma_{ab}}d\log(\rho\|V\|)\left(\frac V{\|V\|}\right)d\mu_E\\
&=\log(\rho(b)\|V(b)\|)-\log(\rho(a)\|V(a)\|).
\end{align*}
Similarly we get that
\begin{equation}\label{eq:gammacd}
\int_{\gamma_{cd}}\alpha=\log(\rho(d)\|V(d)\|)-\log(\rho(c)\|V(c)\|).
\end{equation}

{\bf Step 4: The integral of $\alpha$ on $\mathcal W^F$ curves.}

The restriction of $\omega$ on $F=\ker\omega=T\mathcal W^F$ is zero from the definition of $\omega$. Then
$$
\int_{\gamma_{bc}}\alpha=\int_{\gamma_{bc}}i_Vd\omega.
$$

Since $F=\ker\omega$ is integrable, we can choose $f,g:M\rightarrow\mathbb R$ such that $\omega=fdg$ (take $g$ constant on the leaves of $\mathcal W^F$ with $dg\neq 0$ and rescale by some $f$). If $f$ and $dg$ would be $C^1$ (this would happen if $F$ would be $C^2$), then $d\omega=df\wedge dg=d\log f\wedge\omega$, and $i_Vd\omega|_F=d\log f$.

In our case $F$ and $\omega$ are only $C^1$, so $g$ is $C^1$, and $dg$ and $f$ could be only continuous. We claim that however $dg$ and $f$ are $C^1$ along curves in $\mathcal W^F$ leaves, and the relation $i_Vd\omega|_F=d\log f$ still holds.

We will check the differentiability of $dg$ along $\mathcal W^F$. Let $X$ be a $C^1$ vector field in $F$, then the curves of the flow $\phi^X$ are inside $\mathcal W^F$. We have
$$
dg(p)=d(g\circ\phi_t^X)(p)=dg(\phi_t^X(p))\cdot D\phi_t^X(p),
$$
or
$$
dg(\phi_t^X(p))=dg(p)\cdot \left[D\phi_t^X(p)\right]^{-1}.
$$
Since $t\mapsto D\phi_t^X(p)$ is $C^1$, we get that $dg$ is $C^1$ along the curve $t\mapsto\phi_t^X(p)$, and then the same must hold for $f$ (and $\log f$) since $\omega$ is also $C^1$, so $d(\log f)(X)$ is well defined along $\mathcal W^F$. 

Now assume that the flow generated by the vector field $X$ joins $b$ with $c$ in $\mathcal W^F(b)$, or $\phi_T^X(b)=c$ for some $T>0$ and $\gamma_{bc}$ is $\phi_t^X(b)$, $t\in[0,T]$. Then
\begin{align*}
i_V&d\omega(X)(\phi_t^X(b))=-i_Xd\omega(V)(\phi_t^X(b))=-\LL_X\omega(V)(\phi_t^X(b))\\
&=-\lim_{s\rightarrow 0}\frac{\phi_s^{X*}\omega-\omega}s(V)(\phi_t^X(b))=-\lim_{s\rightarrow 0}\frac{f\circ\phi_s^Xd\left(g\circ\phi_s^X\right)-fdg}s(V)(\phi_t^X(b))\\
&=-\lim_{s\rightarrow 0}\frac{f(\phi_{t+s}^X(p))-f(\phi_t(p))}sdg(V)(\phi_t^X(p))=\frac{df(X)(\phi_t^X(p))}{f(\phi_t^X(P))}fdg(X)(\phi_t^X(p))\\
&=-d(\log f)(X)(\phi_t^X(p))\omega(V)(\phi_t^X(p))=-\frac{\partial}{\partial t}\log(f(\phi_t^X(p))).
\end{align*}
We used the facts that $\omega(V)=1$, $\LL_x\omega=di_X\omega+i_Xd\omega$ and $di_X\omega=0$ since $X\in\ker\omega$. Then
\begin{align*}
\int_{\gamma_{bc}}\alpha&=\int_{\gamma_{bc}}i_Vd\omega=\int_0^Ti_Vd\omega(X)(\phi_t^X(b))dt\\
&=-\int_0^T\frac{\partial}{\partial t}\log(f(\phi_t^X(p)))dt=\log f(b)-\log f(c)=\log\frac{f(b)}{f(c)}.
\end{align*}

Let $h_{bc}:\mathcal W^E(b)\rightarrow\mathcal W^E(c)$ be the holonomy between $\mathcal W^E(b)$ and $\mathcal W^E(c)$ along the $\mathcal W^F$ foliation and homotopic to $\gamma_{bc}$. Then $h_{bc}$ is $C^1$, and the Jacobian of $h_{bc}$ is given by
$$
Dh_{bc}(b)\left(\frac {V(b)}{\|V(b)\|}\right)=Jh_{bc}(b)\frac{V(c)}{\|V(c)\|}.
$$
Since $g\circ h_{bc}=g$ we get
\begin{align*}
dg(b)V(b)&=d(g\circ h_{bc})V(b)=\|V(b)\|dg(c)Dh_{bc}(b)\left(\frac {V(b)}{\|V(b)\|}\right)\\
&=Jh_{bc}(b)\frac{\|V(b)\|}{\|V(c)\|}dg(c)V(c)
\end{align*}
Because $\omega=fdg$ and $\omega(V)=1$ we have that $dg(V)=\frac 1f$. Then we obtain
$$
\frac{f(b)}{f(c)}=\frac{dg(c)V(c)}{dg(b)V(b)}=\frac 1{Jh_{bc}(b)}\frac{\|V(c)\|}{\|V(b)\|},
$$
so
\begin{equation}\label{eq:gammabc}
\int_{\gamma_{bc}}\alpha=\log\left(\frac{\|V(c)\|}{\|V(b)\|}\right)-\log Jh_{bc}(b).
\end{equation}
Similarly we get
\begin{equation}\label{eq:gammada}
\int_{\gamma_{da}}\alpha=\log\left(\frac{\|V(a)\|}{\|V(d)\|}\right)-\log Jh_{da}(d)=\log\left(\frac{\|V(a)\|}{\|V(d)\|}\right)+\log Jh_{bc}(a),
\end{equation}
since $h_{da}=h_{bc}^{-1}$ and $h_{bc}(a)=d$.

{\bf Step 5: Concluding the invariance of disintegrations under holonomies.}

Now we put together \eqref{eq:gammacd}, \eqref{eq:gammabc}, \eqref{eq:gammada}, and we use that $\alpha$ is locally exact. Then
\begin{eqnarray*}
0&=&\int_{\gamma_{ab}}\alpha+\int_{\gamma_{bc}}\alpha+\int_{\gamma_{cd}}\alpha+\int_{\gamma_{da}}\alpha\\
&=&\log(\rho(b)\|V(b)\|)-\log(\rho(a)\|V(a)\|)+\log\left(\frac{\|V(c)\|}{\|V(b)\|}\right)-\log Jh_{bc}(b)+\\
&&+\log(\rho(d)\|V(d)\|)-\log(\rho(c)\|V(c)\|)+\log\left(\frac{\|V(a)\|}{\|V(d)\|}\right)+\log Jh_{bc}(a)\\
&=&\log\left(\frac{Jh_{bc}(a)\rho(d)}{\rho(a)}\right)-\log\left(\frac{Jh_{bc}(b)\rho(c)}{\rho(b)}\right).
\end{eqnarray*}
Assume that we fix two nearby local $W^E$ leaves, $\mathcal W^E(b)$ and $\mathcal W^E(c)$, and let $h_{bc}$ be the local holonomy between them along $\mathcal W^F$. Then for any $a\in\mathcal W^E(b)$, the expression
\begin{equation}\label{eq:invmeasure}
\frac{Jh_{bc}(a)\rho(h_{bc}(a))}{\rho(a)}=k
\end{equation}
is constant.

Remember that $m_E(b)$ is the disintegration of $\mu$ along the local $\mathcal W^E$ leaf passing through $b$, and $\rho(a)$ is the density of $m_E(b)$ with respect to the Lebesgue measure on the local leaf, $\mu_E(b)$. On the other hand, from the formula of change of variable, we have that $Jh_{bc}(a)\rho(h_{bc}(a))$ is the density of the pull-back under the holonomy $h_{bc}$ (or the push-forward under $h_{bc}^{-1}$) of $m_E(c)$, the disintegration of $\mu$ along the local $\mathcal W^E$ leaf passing through $c$.

Then the formula \eqref{eq:invmeasure} shows that the holonomy along $\mathcal W^F$ preserves the disintegrations of $\mu$ along $\mathcal W^E$ local leaves, modulo the multiplication by a constant. In general the disintegrations are well defined modulo multiplication with constants, due to the choice of the foliation chart, so if the local foliation chart is properly chosen, then the disintegrations will be invariant under the $\mathcal W^F$ holonomies. This concludes the proof of Theorem~\ref{teo:rigidity}.

\end{proof}

\begin{rem}
Let us comment that even though we obtain the above result for the case when one foliation is one-dimensional, it seems very probable that the result should work for higher dimensional foliations too, at least under the condition that the two bundles are $C^1$ and integrable.
\end{rem}

\begin{proof}[Proof of Corollary~\ref{cor:rigidityA}]

The stable and unstable bundles of an area preserving $C^{\infty}$ diffeomorphism are of class $C^{2-}$ by \cite{PSW2004}, so we can apply Theorem~\ref{teo:rigidity} and obtain that the disintegrations of the area along the unstable foliation are invariant under the stable holonomy. Since the disintegrations are smooth, meaning that the densities are uniformly $C^{\infty}$ along the unstable leaves, then the Jacobian of the stable holonomy between unstable leaves must be also uniformly $C^{\infty}$. Because the unstable leaves are one dimensional this implies that the stable holonomies between unstable leaves are uniformly $C^{\infty}$, and the Journ\'e Lemma \cite{J1988} gives that the stable foliation is $C^{\infty}$. Once we get the smoothness of the foliations, we can use \cite{HK1990} or \cite{LMM1986} and we obtain the conclusion on rigidity.

\end{proof}

\begin{rem}
We used several dynamical results in order to simplify the proof, however one can obtain a more general result. The property of being critical is in fact a property of the two transversal bundles, without considering any dynamics. One can show that if the bundles are $C^1$ and critical in the sense above, then there exists a diffeomorphism preserving the area and taking the two foliations into two foliations of the two-torus by straight lines.
\end{rem}

\bibliographystyle{plain}
\bibliography{SVV}

\end{document}